\documentclass[9pt,reqno]{amsart}
\usepackage[T1]{fontenc}
\usepackage{lmodern}
\usepackage{microtype}
\usepackage[
	paperwidth=6.75in,paperheight=9.5in,
	inner=0.82in,outer=0.72in,top=0.75in,bottom=0.82in
]{geometry}
\usepackage{amsmath,amssymb,amsthm,mathtools,mathrsfs}
\usepackage{enumitem}
\usepackage{booktabs}
\usepackage{xcolor}
\usepackage{tikz}
\usepackage{placeins}
\usepackage[normalem]{ulem}
\usepackage[colorlinks=true,linkcolor=blue!55!black,citecolor=blue!55!black,
						urlcolor=blue!55!black]{hyperref}

\numberwithin{equation}{section}
\allowdisplaybreaks
\newtheorem{theorem}{Theorem}[section]
\newtheorem{proposition}[theorem]{Proposition}
\newtheorem{lemma}[theorem]{Lemma}
\newtheorem{corollary}[theorem]{Corollary}
\theoremstyle{definition}
\newtheorem{definition}[theorem]{Definition}
\newtheorem{construction}[theorem]{Construction}
\theoremstyle{remark}
\newtheorem{remark}[theorem]{Remark}

\newcommand{\K}{\mathbb{K}}
\newcommand{\N}{\mathbb{N}}
\newcommand{\Q}{\mathbb{Q}}
\newcommand{\Bcal}{\mathcal{B}}
\newcommand{\Kcal}{\mathcal{K}}
\newcommand{\Cal}{\mathop{\rm Cal}}
\newcommand{\dist}{\operatorname{dist}}
\newcommand{\supp}{\operatorname{supp}}
\newcommand{\ran}{\operatorname{ran}}
\newcommand{\norm}[1]{\left\lVert #1\right\rVert}
\newcommand{\abs}[1]{\left\lvert #1\right\rvert}

\title[Duals of separable spaces as Calkin algebras and Universal Ideal Quotients]
{Duals of separable Banach Spaces as Calkin Algebras and Universal Ideal Quotients}
\hypersetup{
	pdftitle={Duals of separable Banach spaces as Calkin algebras},
	pdfsubject={Universal Banach-space realization and a two-sorted Bourgain--Delbaen Calkin construction},
	pdfkeywords={Calkin algebra, Bourgain--Delbaen construction, Argyros--Haydon method,
	finite-extension method, dual Banach space, square-zero radical}
}
\author{Rui Liu}
\address{School of Mathematical Sciences and LPMC, Nankai University, Tianjin
300071, P.R. China}
\email{ruiliu@nankai.edu.cn}

\author{Jie Shen}
\address{School of Mathematical Sciences and LPMC, Nankai University, Tianjin
300071, P.R. China}
\email{1710064@mail.nankai.edu.cn}

\date{September 26, 2026}

\begin{document}

\begin{abstract}
	For \(\K=\mathbb{R}\) or \(\mathbb{C}\) and every separable Banach space \(V\), we use an oracle-relative two-sorted finite-extension construction, whose oracle records the local rational structure of \(V\), to construct a separable Banach space \(X_V\) such that
	\[
		\Cal(X_V)=\Bcal(X_V)/\Kcal(X_V)\simeq
		\begin{pmatrix}
			\K&0\\
			V^*&\K
		\end{pmatrix}
	\]
	as Banach algebras.
	The identification of \(V^*\) with the Jacobson radical is isometric.
	Consequently, every nonzero dual Banach space with a separable predual admits, after an equivalent renorming, a unital Banach-algebra structure isomorphic to the Calkin algebra of a separable Banach space.
	Moreover, there is a separable Banach space \(X\) such that every separable Banach space is isometric to \(\mathcal{J}/\Kcal(X)\) for a closed two-sided ideal \(\mathcal{J}\) of \(\Bcal(X)\), and the subspace--ideal correspondence preserves the canonical order.
\end{abstract}

\maketitle
\enlargethispage{3pt}
\tableofcontents

\section{Introduction}

For a Banach space \(Z\), write
\[ \Cal(Z)=\Bcal(Z)/\Kcal(Z), \]
where \(\Bcal(Z)\) is the algebra of bounded operators on \(Z\) and \(\Kcal(Z)\) is the ideal of compact operators on \(Z\).
Throughout, \(\N=\{0,1,\ldots\}\) and \(\N_+=\{1,2,\ldots\}\).
The name comes from Calkin's 1941 study of the Hilbert-space quotient \cite{Calkin1941}.

The main result of this paper is a universal realization theorem at the level of Banach-space isomorphism types.
The algebraic result is a realization of lower triangular Banach algebras whose semisimple quotient is \(\K\oplus\K\) and whose square-zero radical is the dual of a separable Banach space.
Let \(V\) be separable.
On \(\K\oplus V^*\oplus\K\), we consider the multiplication
\[ (\lambda,v,\mu)(\lambda',v',\mu') =\bigl(\lambda\lambda',v\lambda'+\mu v',\mu\mu'\bigr). \]
Equipped with the norm \(\norm{(\lambda,v,\mu)}=|\lambda|+\norm{v}+|\mu|\), this is the lower triangular Banach algebra with scalar diagonal and lower-left corner \(V^*\).
The algebra constructed below has two scalar characters and a square-zero lower-left radical.

For a separable coefficient space \(U\), we write \(X_U\) for the direct-sum space furnished by Theorem~\ref{thm:tri-cal} below.

\begin{theorem}[Dual as Calkin algebra]
	\label{thm:dua-cal}
	For every nonzero dual Banach space \(V^*\) with \(V\) separable, there is a separable Banach space \(Z\) such that \(V^*\) admits an equivalent norm and a unital multiplication for which
	\[ V^*\simeq\Cal(Z) \]
	as Banach algebras.
\end{theorem}

Forgetting the multiplication gives the following Banach-space consequence.

\begin{corollary}[Banach-space universality]
	\label{cor:ban-uni}
	For every nonzero separable Banach space \(V\), there is a separable Banach space \(Z_V\) such that
	\[ V^*\simeq\Cal(Z_V) \]
	as Banach spaces, with no algebra structure prescribed on \(V^*\).
	In particular, if \(E\ne\{0\}\) is separable and reflexive, then
	\[ E\simeq\Cal(Z_{E^*}) \]
	as Banach spaces and hence \(\Cal(Z_{E^*})\) is reflexive.
\end{corollary}

Both statements follow from the following more precise operator result.

\begin{theorem}[Triangular realization theorem]\label{thm:tri-cal}
	For every separable Banach space \(V\), there are separable Banach spaces \(X_0(V)\), \(X_1(V)\), and a bounded linear map
	\[ V^*\longrightarrow\Bcal(X_0(V),X_1(V)), \qquad v\longmapsto T_v, \]
	such that
	\begin{align*}
		\Bcal(X_i(V))/\Kcal(X_i(V))&=\K[I_{X_i(V)}]
		\quad(i=0,1),\\
		\Bcal(X_1(V),X_0(V))&=\Kcal(X_1(V),X_0(V)),\\
		\Bcal(X_0(V),X_1(V))/\Kcal(X_0(V),X_1(V)) &=\{[T_v]:v\in V^*\}.
	\end{align*}
	Moreover, \(v\mapsto[T_v]\) is an isometry, so
	\[ \norm{[T_v]}=\norm{v} \qquad(v\in V^*). \]
	Consequently, for \(X_V=(X_0(V)\oplus X_1(V))_{\ell_\infty}\), we have
	\[
		\Cal(X_V)\simeq
		\left\{\begin{pmatrix}\lambda&0\\v&\mu\end{pmatrix}:
		\lambda,\mu\in\K,\ v\in V^*\right\}
	\]
	as Banach algebras.
	Its Jacobson radical is isometric to \(V^*\), and its square is zero.
\end{theorem}

\begin{theorem}[Universal ideal quotients]\label{cor:ide-lat}
	There is a separable Banach space \(X\) such that every separable Banach space \(Y\) is isometric to \(\mathcal{J}/\Kcal(X)\) for some closed two-sided ideal \(\mathcal{J}\) of \(\Bcal(X)\) containing \(\Kcal(X)\).
	More precisely, if \(\mathcal{E}(X)\) denotes the ideal of inessential operators, then the closed-subspace lattice of \(\ell_\infty\) is lattice-isomorphic to the interval of closed two-sided ideals \(\mathcal{J}\) satisfying \( \Kcal(X)\subseteq\mathcal{J}\subseteq\mathcal{E}(X). \) This correspondence preserves the canonical order in both directions.
\end{theorem}

The proof follows from two established lines of construction in Banach-space theory.

\subsection{The Argyros--Haydon construction and related work}
\label{sec:intro-ah}

One line begins in 1974 with Tsirelson's construction of a reflexive space containing neither \(c_0\) nor any \(\ell_p\) \cite{Tsirelson1974}.
Figiel and Johnson gave the dual Tsirelson space its modern version of that norm, providing the recursive method used in many later constructions \cite{FigielJohnson1974}.
Maurey and Rosenthal constructed a normalized weakly null sequence with no unconditional subsequence; their argument was an early source of the coding methods used later to suppress unconditional structure \cite{MaureyRosenthal1977}.
Schlumprecht adapted recursive norming in 1991 to construct the first arbitrarily distortable space \cite{Schlumprecht1991}.
And then Odell and Schlumprecht proved that \(\ell_2\) is arbitrarily distortable \cite{OdellSchlumprecht1994}.

In 1993, Gowers and Maurey combined recursive norming, weighted averages, and special coding to solve the unconditional basic sequence problem.
Their construction gave the first hereditarily indecomposable (HI) space \cite{GowersMaurey1993}.
Here HI means that no infinite-dimensional closed subspace decomposes as a direct sum of two infinite-dimensional closed subspaces.
Gowers later constructed an HI space with an asymptotic unconditional basis \cite{Gowers1995}, and his dichotomy showed that every infinite-dimensional Banach space contains either an unconditional basic sequence or an HI subspace \cite{Gowers1996}.
Ferenczi showed in 1997 that HI can coexist with uniform convexity \cite{Ferenczi1997UC}.
He also proved that, for a complex HI space \(X\), every operator from a subspace of \(X\) into \(X\) is a scalar multiple of the inclusion plus a strictly singular operator \cite{Ferenczi1997Operators}.
Gowers and Maurey developed related coding methods further to construct Banach spaces with small operator algebras \cite{GowersMaurey1997}.

In parallel, asymptotic and mixed-Tsirelson methods developed along a related line.
Maurey, Milman, and Tomczak-Jaegermann introduced a general asymptotic framework and studied asymptotic \(\ell_p\)-spaces \cite{MaureyMilmanTomczak1995}.
Argyros and Deliyanni constructed asymptotic \(\ell_1\)-spaces through mixed-Tsirelson norms \cite{ArgyrosDeliyanni1997}.
Argyros, Deliyanni, Kutzarova, and Manoussakis developed modified mixed-Tsirelson spaces, including arbitrarily distortable and HI examples \cite{ArgyrosDeliyanniKutzarovaManoussakis1998}.
The work of Odell and Schlumprecht on Krivine sets and block finite universality provided the early form of the method later called saturation under constraints \cite{OdellSchlumprecht1995,OdellSchlumprecht2000}.
Argyros and Felouzis used interpolation and HI methods in quotient and operator-factorization problems \cite{ArgyrosFelouzis2000}.
Argyros and Tolias then gave, in 2004, a general HI construction method and organized special convex combinations, rapidly increasing sequences (RIS), the basic inequality, and operator analysis in one framework \cite{ArgyrosTolias2004}.
These methods were further developed in the construction of the first known reflexive space with the hereditary invariant subspace property \cite{ArgyrosMotakis2014}.
Argyros and Motakis later introduced a dual method covering both reflexive HI spaces and HI spaces without reflexive subspaces \cite{ArgyrosMotakis2016}.

The second line comes from the theory of \(\mathcal{L}_\infty\)-spaces, developed in the work of Lindenstrauss and Pe\l czy\'nski \cite{LindenstraussPelczynski1968} and Lindenstrauss and Rosenthal \cite{LindenstraussRosenthal1969}.
In 1980, Bourgain and Delbaen introduced a rank-by-rank finite-dimensional extension scheme and constructed new \(\mathcal{L}_\infty\)-spaces whose duals are isomorphic to \(\ell_1\) and which do not contain \(c_0\) \cite{BourgainDelbaen1980}.
Bourgain and Pisier gave a complementary construction which embeds every separable Banach space into a separable \(\mathcal{L}_\infty\)-space with a quotient having the Radon--Nikod\'ym and Schur properties \cite{BourgainPisier1983}.
Alspach established the small Szlenk index of a Bourgain--Delbaen space \cite{Alspach2000}.
Haydon showed that every infinite-dimensional closed subspace of the original \(X_{a,b}\) contains a further subspace isomorphic to some \(\ell_p\), \(1<p<\infty\) \cite{Haydon2000}.
Freeman, Odell, and Schlumprecht later proved that every Banach space with separable dual embeds into an \(\mathcal{L}_\infty\)-space whose dual is isomorphic to \(\ell_1\) \cite{FreemanOdellSchlumprecht2011}.

In 2011, Argyros and Haydon combined the Bourgain--Delbaen scheme with mixed-Tsirelson restrictions, RIS estimates, and dependent sequences.
They constructed an HI space \(X_{\mathrm{AH}}\) on which every operator is scalar plus compact, and hence \(\Cal(X_{\mathrm{AH}})\simeq\K\) \cite{ArgyrosHaydon2011}.
This answered the scalar-plus-compact question recorded as Question~1 by Lindenstrauss \cite{Lindenstrauss1976}.
Zisimopoulou subsequently introduced vector-valued Bourgain--Delbaen \(\mathcal{L}_\infty\)-sums in 2014, which became an important form of the construction \cite{Zisimopoulou2014}.
Later work developed the Bourgain--Delbaen and Argyros--Haydon methods in several directions.
Argyros, Freeman, Haydon, Odell, Raikoftsalis, Schlumprecht, and Zisimopoulou showed that every separable uniformly convex space embeds into a scalar-plus-compact space \cite{ArgyrosEtAl2012}.
Argyros, Gasparis, and Motakis proved that every infinite-dimensional separable \(\mathcal{L}_\infty\)-space is isomorphic to a Bourgain--Delbaen space \cite{ArgyrosGasparisMotakis2016}.
Manoussakis, Pelczar-Barwacz, and \'Swi\k{e}tek constructed an unconditionally saturated Bourgain--Delbaen space with the scalar-plus-compact property \cite{ManoussakisPelczarSwietek2017}.
Argyros and Motakis later constructed a scalar-plus-compact space containing no infinite-dimensional reflexive subspace \cite{ArgyrosMotakis2019}.

The Argyros--Haydon result also began the modern line of explicit Calkin realizations.
In 2012, Tarbard first moved beyond the one-dimensional algebra.
For every \(k\geq2\), he constructed a space whose Calkin algebra is the truncated-polynomial algebra \(\K[t]/(t^k)\), a unital algebra with nilpotent radical \cite{Tarbard2012}.
In his thesis, Tarbard also constructed a space whose Calkin algebra is isometrically isomorphic to the convolution algebra \(\ell_1(\N)\) \cite{Tarbard2013}.

In 2016, Motakis, Puglisi, and Zisimopoulou used Bourgain--Delbaen sums to realize \(C(K)\) for every countable compact metric space \(K\) \cite{MotakisPuglisiZisimopoulou2016}.
Motakis later proved that every compact metric space \(K\) admits a realization whose Calkin algebra is isometrically isomorphic to \(C(K)\) \cite{Motakis2024}.

Kania and Laustsen used \(X_{\mathrm{AH}}\) and one of its subspaces to obtain a three-dimensional triangular Calkin algebra in their 2017 work.
This algebra has two scalar characters and a one-dimensional square-zero off-diagonal radical \cite{KaniaLaustsen2017}.  
Their note added in proof also realizes every finite-dimensional semisimple complex algebra.

Argyros, Deliyanni, and Tolias developed a method for Banach algebras of diagonal operators and constructed an HI example on which every diagonal operator is scalar plus compact \cite{ArgyrosDeliyanniTolias2011}.
Building on this work and on Argyros--Haydon sums, Motakis, Puglisi, and Tolias proved that the algebra of scalar-plus-compact diagonal operators associated with every real Banach space with a Schauder basis is a Calkin algebra.
Their applications include the James spaces \(J_p\) and their duals with natural multiplications.
They also proved, at the Banach-space level, that every nonreflexive Banach space with an unconditional basis is isomorphic to a Calkin algebra \cite{MotakisPuglisiTolias2020} in 2020.

Motakis and Pelczar-Barwacz showed that a Calkin algebra can be infinite-dimensional and reflexive, even isomorphic to a Hilbert space.
More generally, they realized the unitization of \(U\), with coordinatewise multiplication, when \(U\) has a normalized unconditional basis with no \(c_0\) asymptotic version \cite{MotakisPelczar2025}.
Motakis and Puglisi realized the unitization of \(\Kcal(c_0)\).
Their space is an Argyros--Haydon sum of copies of one Argyros--Haydon space, with the internal and external construction parameters kept separate \cite{MotakisPuglisi2025}.

Another related example comes from the reflexive space \(X_{\rm AM}\).
Every operator on this space is scalar plus strictly singular, its strictly singular ideal is nonseparable, and the product of two strictly singular operators is compact \cite{ArgyrosMotakis2016}.
Thus, as observed in \cite{Skillicorn2015}, its Calkin algebra has a nonseparable zero-product radical.
This radical is given as a quotient of operator ideals rather than as a prescribed external Banach space.

These results clarify the scope of the present theorem.
Tarbard's truncated-polynomial examples have one scalar character and finite-dimensional nilpotent radical, while the Kania--Laustsen triangular example has two scalar characters and a one-dimensional square-zero radical.
The Argyros--Motakis--Skillicorn example has a nonseparable square-zero radical but only one scalar character, and that radical is not identified with a prescribed external Banach space.
In the present result, the radical is an arbitrary dual \(V^*\) with separable predual and the semisimple quotient is \(\K\oplus\K\).

\subsection{Recursion theory and connections with other areas of mathematics}
\label{sec:intro-recursion}

Recursion theory studies effective procedures and their dependence on information.
Besides decidability, it provides tools for comparing the information contained in mathematical objects and for organizing constructions by successive finite requirements \cite{Odifreddi1989,Soare2016}.
Turing's oracle machines formalize procedures whose fixed rules may consult a specified oracle \cite[Section~4]{Turing1939}.
The oracle can supply noncomputable information, while each completed computation uses only finitely many queries.

In group theory, Higman's embedding theorem says that a finitely generated group embeds into a finitely presented group if and only if it admits a recursively enumerable set of defining relations \cite{Higman1961}.
Thus a condition on effective presentation has an algebraic characterization by embeddings.

In number theory, the Davis--Putnam--Robinson--Matiyasevich theorem identifies the recursively enumerable subsets of \(\N\) with the Diophantine sets \cite{Matiyasevich1970}.
Membership in such a set can be expressed by the existence of natural-number solutions to a fixed polynomial equation with integer coefficients, with the set element as a parameter.
Since some recursively enumerable sets are undecidable, this yields the negative solution to Hilbert's tenth problem: no algorithm decides whether an arbitrary polynomial equation with integer coefficients has an integer solution.

In Banach-space theory, Pour-El and Richards developed computability structures compatible with linear operations, norms, and effective limits \cite[Chapters~2--3]{PourElRichards1989}.
Their First Main Theorem applies to a closed linear operator between such spaces whose domain contains a computable sequence with dense linear span in the source space and whose images of that sequence form a computable sequence.
Under these hypotheses, the operator sends every computable vector in its domain to a computable vector if and only if it is bounded.
Bri\"et, Naor and Regev used constructions of locally decodable codes from theoretical computer science to prove that, for every \(1<p<\infty\), there is a Banach space \(X\) of finite cotype such that \(\ell_p\widehat\otimes_\pi X\) fails to have finite cotype \cite{BrietNaorRegev2012}.
Gunes and Naor determined the separation modulus of \(M_n(\mathbb R)\) equipped with an arbitrary unitarily invariant norm up to universal constant factors \cite{GunesNaor2026}.
They also obtained a deterministic oracle-polynomial-time algorithm that approximates this modulus within universal constant factors using norm queries.

In real analysis, algorithmic randomness gives a precise description of the points at which computable monotone functions are differentiable.
Brattka, Miller, and Nies proved that \(z\in(0,1)\) is computably random if and only if every computable nondecreasing function \(f:[0,1]\to\mathbb R\) is differentiable at \(z\) \cite[Theorem~4.3]{BrattkaMillerNies2016}.
This refines the classical almost-everywhere differentiability theorem for monotone functions by characterizing the points that work simultaneously for all computable functions in this class.

In geometric measure theory, the point-to-set principle of Jack H.~Lutz and Neil Lutz \cite{LutzLutz2018} gives \begingroup \postdisplaypenalty=10000
\[ \dim_{\mathrm H}(E) =\min_{A\subseteq\N}\sup_{x\in E}\dim^A(x) \qquad(\varnothing\ne E\subseteq\mathbb R^n), \]
\endgroup where \(\dim^A(x)\) is the effective Hausdorff dimension of the point \(x\) relative to the oracle \(A\).
The formula turns the task of proving \(\dim_{\mathrm H}(E)\ge s\) into a pointwise problem: for every oracle \(A\) and every \(\varepsilon>0\), one seeks a point \(x\in E\) such that \(\dim^A(x)\ge s-\varepsilon\).

Using algorithmic fractal dimension and this principle, Jack H.~Lutz, Renrui Qi, and Liang Yu \cite{LutzQiYu2024} proved that, for every \(s\in[0,1]\), there is a Hamel basis \(B\) of \(\mathbb R\) over \(\Q\) with \(\dim_{\mathrm H}(B)=s\).
This is a classical geometric and algebraic existence theorem proved by computability-theoretic methods.

\subsection{Outline}
\label{sec:intro-outline}

There are also obstructions to Calkin realization.
Horv\'ath and Kania constructed unital Banach algebras that are not isomorphic to Calkin algebras of separable Banach spaces \cite{HorvathKania2021}.
Acuaviva and Acuaviva constructed a unital Banach algebra that is not isomorphic to the Calkin algebra of any Banach space \cite{AcuavivaAcuaviva2026}.
However, the underlying Banach space of the latter example is not isomorphic to any dual Banach space; see Remark~\ref{rem:acu-not-dual}.
The universality result above concerns dual Banach-space isomorphism types equipped with the displayed triangular multiplication.

The proof uses several standard schemes.
The rank-by-rank extension scheme is the Bourgain--Delbaen scheme \cite{BourgainDelbaen1980}.
The evaluation analysis, the RIS and basic-inequality method, and the exact-pair and dependent-sequence method follow the Argyros--Haydon line \cite{ArgyrosHaydon2011,Tarbard2013,Motakis2024}.
The auxiliary norm is a mixed Tsirelson norm in the sense of Argyros and Deliyanni \cite{ArgyrosDeliyanni1997}.
Vector-valued BD sums provide related background for the use of finite-dimensional fibres \cite{Zisimopoulou2014}.
The fair scheduling of finite requirements uses the recursion-theoretic finite-extension scheme described in \cite[Section~V.2 and Proposition~V.3.13]{Odifreddi1989}.

The proof of Theorem~\ref{thm:tri-cal} has two parts.
Part~I constructs the two spaces by a monotone finite-extension scheme relative to the local structure of \(V\) and the corresponding admissibility data.
The formal language, transition rules, and scheduler are fixed independently of \(V\); see the oracle formulation in Remark~\ref{rem:orc-rec}.

Increasing finite-dimensional fibres store the predual, while protected lifts make \(v\mapsto T_v\) linear and give uniform bounds.
The finite-extension and fairness properties realize every admissible finite template cofinally.
Part~I also supplies interval restriction, common-stem comparison, and probes.
The consequences used later are collected as \hyperref[itm:ana-g1]{\textup{(G1)}}--\hyperref[itm:ana-g7]{\textup{(G7)}} and \hyperref[itm:ana-ea1]{\textup{(EA1)}}--\hyperref[itm:ana-ea4]{\textup{(EA4)}}.

Part~II proves the four operator estimates.
A basic inequality gives the upper estimate for dependent sequences.
A persistent error is converted into rational packet data and then into an exact finite certificate.
The finite-extension property places this certificate in a dependent sequence, whose lower estimate gives a contradiction.
A missing-level argument treats the reverse corner.
Finally, probes make the local forward coefficients compatible and determine one \(v\in V^*\).

Readers interested only in the proof strategy may read Subsection~\ref{sec:ana-inp}, then Propositions~\ref{prop:crt-loc-orb} and~\ref{prop:rev-loc-orb}, then Lemma~\ref{lem:slf-glo}, and finally Corollary~\ref{cor:slf-rev} and Theorem~\ref{thm:fwd-glo}.

The triangular theorem has several immediate algebraic consequences.
We record two of them here.
The remaining consequences are stated and proved in Subsection~\ref{sec:str-con}.

\begin{corollary}[Universal reflexive Calkin radicals]
	\label{cor:ref-rad}
	For every separable reflexive Banach space \(E\), there is a separable Banach space \(Y_E\) such that \(\Cal(Y_E)\) is reflexive as a Banach space and
	\[
		\Cal(Y_E)\simeq
		\left\{\begin{pmatrix}\lambda&0\\e&\mu\end{pmatrix}:
		\lambda,\mu\in\K,\ e\in E\right\}.
	\]
	In particular,
	\begin{align*}
		\operatorname{rad}\Cal(Y_E)\simeq E,\qquad
		\bigl(\operatorname{rad}\Cal(Y_E)\bigr)^2=\{0\},\qquad
		\Cal(Y_E)/\operatorname{rad}\Cal(Y_E)\simeq\K\oplus\K.
	\end{align*}
	Thus every separable reflexive Banach space is isometric to the codimension-two square-zero radical of a reflexive Calkin algebra.
\end{corollary}

\begin{corollary}[An \(\ell_\infty\)-radical]
	\label{cor:ell-rad}
	There is a separable Banach space \(X_{\ell_1}\) such that
	\[
		\Cal(X_{\ell_1})\simeq
		\left\{\begin{pmatrix}\lambda&0\\v&\mu\end{pmatrix}:
		\lambda,\mu\in\K, \ v\in\ell_\infty\right\}
	\]
	as Banach algebras.
	Its Jacobson radical is isometric to \(\ell_\infty\) and has square zero.
	In particular, the whole Calkin algebra is nonseparable.
\end{corollary}

Corollaries~\ref{cor:ban-uni} and \ref{cor:ref-rad} complement the coordinatewise unitizations in \cite{MotakisPelczar2025}.
In Corollary~\ref{cor:ref-rad}, the separable reflexive space \(E\) is arbitrary and occurs as the square-zero radical of the displayed two-character triangular algebra.
The Argyros--Motakis--Skillicorn example already has a nonseparable square-zero radical given as a quotient of operator ideals.
In contrast, Corollary~\ref{cor:ell-rad} prescribes the radical as \(\ell_\infty\) and identifies the whole triangular algebra.

Write \(M[0,1]\) for the space of finite regular \(\K\)-valued Borel measures and \(JT\) for the classical James--tree space \cite{James1974}.
Further immediate instances of the triangular theorem are
\[
	\begin{array}{c|c}
		V&\operatorname{rad}\Cal(X_V)\\ \hline
		c_0&\ell_1\\
		C[0,1]&M[0,1]\\
		JT&JT^*
	\end{array}
\]

\section{Preliminaries and four-corner reduction}

For \(X=X_0\oplus X_1\), Theorem~\ref{thm:tri-cal} reduces to the following four-corner statements:
\begin{align*}
	\Bcal(X_i)/\Kcal(X_i)&=\K[I_{X_i}]\quad(i=0,1),\\
	\Bcal(X_1,X_0)&=\Kcal(X_1,X_0),\\
	\Bcal(X_0,X_1)/\Kcal(X_0,X_1)
	&=\{[T_v]:v\in V^*\}.
\end{align*}
The map \(v\mapsto[T_v]\) must also be an isometry.
Its lower estimate comes from the probes.

Parts~I and~II implement this reduction.
Part~I constructs the two-sorted datum and records the analytic consequences of the construction that are used in Part~II.
Part~II proves the RIS and dependent-sequence estimates, classifies the four corners modulo compact operators, and derives the algebraic consequences in Subsection~\ref{sec:str-con}.

Here FDD means a finite-dimensional decomposition.
The principal notation is as follows.
\begin{description}[leftmargin=34mm,style=nextline]
	\item[\(i=0,1\)] the source and target sides, respectively; \item[\(V,V^*\)] the separable predual and its dual coefficient space; \item[\(V_r\)] the nested finite-dimensional source fibres exhausting \(V\); \item[\(m_h^{-1},n_h\)] the weight and maximal age at level \(h\); \item[\(\operatorname{rank}\gamma,\operatorname{lev}(\gamma),\operatorname{age}\gamma\)] the rank, level, and age of a node; these are different integers; \item[\(e_{\gamma,\varphi}^{i*},d_{\gamma,\varphi}^{i*}\)] the ambient coordinate evaluation and its Bourgain--Delbaen (BD) biorthogonal part; \item[\(D_\gamma,U_\gamma\)] the source FDD coefficient and the full protected source coordinate attached to a target label \(\gamma\); \item[\(P_I^i\)] the FDD projection on side \(i\) associated with a rank interval \(I\); \item[\(\mathcal{K}_i\)] the scalarized ambient coordinate norming set of \(X_i\); \item[\(\mathscr{M}_{ij}(x)\)] the local orbit of \(x\) for the corner \(X_i\to X_j\); \item[\(\mathsf{shist},\mathsf{thist}\)] the same-side structural history and the target coding history of a protected node; \item[\(\langle\eta_s\rangle,\ f_s\)] the complete strong packet formed from the terminal atom \(\eta_s\) of the \(s\)-th complete inner chain and its induced functional \(f_s=e_{\eta_s}^{j*}\) on \(X_j\), respectively.
\end{description}

\subsection{The predual and its finite-dimensional fibres}
\label{sec:coe-fib}

Fix a separable Banach space \(V\), and put \(d=\dim V\in\N\cup\{\infty\}\).
If \(d=\infty\), choose a linearly independent total sequence \((h_n)_{n\ge1}\).
If \(1\le d<\infty\), choose a basis \((h_n)_{n=1}^d\); if \(d=0\), use the empty sequence.
For \(r\in\N_+\), set
\[
	V_r=
	\begin{cases}
		\operatorname{span}\{h_1,\ldots,h_r\},&d=\infty,\\
		\operatorname{span}\{h_1,\ldots,h_{\min\{r,d\}}\},&d<\infty,
	\end{cases}
\]
where the span of the empty set is \(\{0\}\).
Give \(V_r\) the ordered basis inherited from the displayed sequence, let \((V_r)_{\Q}\) denote the vectors with rational coordinates in this basis, and put
\[ V_{\Q}=\bigcup_{r\in\N_+}(V_r)_{\Q}. \]
In the complex case, rational coordinates and matrix entries lie in \(\Q+i\Q\).
A finite-dimensional operator between such coordinate spaces is called rational when its matrix has rational entries.
Then \((V_r)\) is increasing, each \(V_r\) is finite dimensional, and both \(V_{\Q}\) and \(\bigcup_rV_r\) are dense in \(V\).

For \(r,s\in\N_+\) with \(r\le s\), write \(J_{rs}:V_r\to V_s\) for the inclusion and \(R_{rs}=J_{rs}^*:V_s^*\to V_r^*\) for restriction.
The map \(J_{rs}\) is an isometry, while \(R_{rs}\) is a contractive surjection by the Hahn--Banach theorem; it has norm one whenever \(V_r\ne\{0\}\).
Finally, if \(v_r\in V_r^*\) satisfy
\[ R_{rs}v_s=v_r\quad(r,s\in\N_+,\ r\le s), \qquad \sup_{r\in\N_+}\norm{v_r}<\infty, \]
then \(v(u)=v_r(u)\) for \(u\in V_r\) defines a bounded functional on \(\bigcup_rV_r\).
It extends uniquely to \(v\in V^*\), with \(\norm{v}=\sup_r\norm{v_r}\).
Conversely, every \(v\in V^*\) gives such a compatible family by restriction.
The construction below uses only this exhaustion and restriction system.
It requires no approximation properties for \(V\), no projections \(V_s\to V_r\), and no coherent right inverses of the maps \(R_{rs}\).

\subsection{Target algebra}

We work over \(\K\), real or complex.
In the complex case every rational code is taken over \(\Q+i\Q\), and every separating inequality is written with a real part.

With \(V\) fixed as above, equip \(\K\oplus V^*\oplus\K\) with the norm
\begin{equation}\label{eq:tri-norm}
	\norm{(\lambda,v,\mu)}_{\mathfrak{A}} =|\lambda|+\norm{v}_{V^*}+|\mu|,
\end{equation}
and define
\begin{equation}\label{eq:tri-prd}
	(\lambda,v,\mu)(\lambda',v',\mu') =\bigl(\lambda\lambda',\,v\lambda'+\mu v',\, \mu\mu'\bigr).
\end{equation}
Equivalently, this is the lower triangular algebra
\begin{equation}\label{eq:tgt-alg}
	\mathfrak{A}= \left\{
	\begin{pmatrix}\lambda&0\\v&\mu\end{pmatrix}:
	\lambda,\mu\in\K,\ v\in V^*
	\right\}.
\end{equation}
Let
\[
	\mathfrak{N}=
	\left\{\begin{pmatrix}0&0\\v&0\end{pmatrix}:v\in V^*\right\}.
\]
Direct multiplication gives \(\mathfrak{N}^2=\{0\}\), so the nilpotent ideal \(\mathfrak{N}\) is contained in the Jacobson radical.
On the other hand, \(\mathfrak{A}/\mathfrak{N}\cong\K\oplus\K\), which is semisimple.
The image of the radical in a quotient lies in the radical of the quotient; hence \(\operatorname{rad}(\mathfrak{A})\subseteq\mathfrak{N}\).
Therefore
\begin{equation}\label{eq:tgt-rad}
	\operatorname{rad}(\mathfrak{A})=\mathfrak{N}, \qquad \mathfrak{N}^2=\{0\}.
\end{equation}
Obviously, the displayed norm (\ref{eq:tri-norm}) is submultiplicative.
Thus \(\mathfrak{A}\) is a Banach algebra.

\subsection{Four-corner reduction}

The construction below produces Banach spaces \(X_0,X_1\) and a bounded linear injection
\[ V^*\longrightarrow\Bcal(X_0,X_1),\qquad v\longmapsto T_v. \]
The operator theorem establishes
\begin{align}
	\Bcal(X_i)/\Kcal(X_i)&=\K[I_{X_i}]\quad(i=0,1),
	\label{eq:slf-cor}\\
	\Bcal(X_1,X_0)&=\Kcal(X_1,X_0),
	\label{eq:rev-cor}\\
	\Bcal(X_0,X_1)/\Kcal(X_0,X_1)
	&=\{[T_v]:v\in V^*\}.
	\label{eq:fwd-cor}
\end{align}

\begin{proposition}\label{prop:cor-alg}
	Assume \eqref{eq:slf-cor}--\eqref{eq:fwd-cor} and that \(X_0,X_1\) are infinite dimensional, that \(v\mapsto T_v:V^*\to\Bcal(X_0,X_1)\) is bounded and linear, and that \(v\mapsto[T_v]\) is an isomorphic embedding.
	Give \(X=(X_0\oplus X_1)_{\ell_\infty}\) its standard finite-sum norm.
	Then the map
	\[
		\Phi:\Cal(X)\longrightarrow\mathfrak{A}, \qquad \left[
		\begin{pmatrix}A&B\\C&D\end{pmatrix}
		\right]
		\longmapsto
		\begin{pmatrix}\lambda&0\\v&\mu\end{pmatrix},
	\]
	where \(A-\lambda I_{X_0}\), \(B\), \(C-T_v\), and \(D-\mu I_{X_1}\) are compact, is a Banach-algebra isomorphism.
	Here
	\[ A:X_0\to X_0,\quad B:X_1\to X_0, \quad C:X_0\to X_1,\quad D:X_1\to X_1. \]
\end{proposition}

\begin{proof}
	Let \(\jmath_i:X_i\to X\) and \(\pi_i:X\to X_i\) be the canonical injections and projections.
	If \(R\in\Kcal(X)\), then each corner \(\pi_jR\jmath_i\) is compact.
	Conversely, if all four corners of \(R\) are compact, then
	\[ R=\sum_{i,j=0}^1\jmath_j(\pi_jR\jmath_i)\pi_i \]
	is a finite sum of compact operators.
	Hence compactness of an operator on \(X_0\oplus X_1\) is equivalent to compactness of its four entries.

	By \eqref{eq:slf-cor}, the scalars \(\lambda\) and \(\mu\) are unique.
	Otherwise the identity on one of the infinite-dimensional spaces \(X_i\) would be compact.
	By the assumed injectivity of \(v\mapsto[T_v]\), the vector \(v\) is unique as well.
	Thus \(\Phi\) is well defined and injective.
	Equations \eqref{eq:slf-cor}--\eqref{eq:fwd-cor} show that every class has such a representative and that every triple \((\lambda,v,\mu)\) occurs,
	so \(\Phi\) is surjective.

	Each corner extraction \([R]\mapsto[\pi_jR\jmath_i]\) has norm at most \(\|\pi_j\|\,\|\jmath_i\|\).
	The two scalar quotient coordinates are bounded because \(\K[I_{X_i}]\) is one dimensional,
	and the \(V^*\)-coordinate is bounded because \(v\mapsto[T_v]\) is an isomorphic embedding.
	Hence \(\Phi\) is bounded.
	Conversely,
	\[
		(\lambda,v,\mu)\longmapsto
		\left[\begin{pmatrix}\lambda I_{X_0}&0\\T_v&\mu I_{X_1}\end{pmatrix}\right]
	\]
	is bounded by the boundedness of \(v\mapsto T_v\).
	This is the inverse of \(\Phi\).

	It remains to check multiplication.
	Choose representatives
	\[
		R=\begin{pmatrix}\lambda I_{X_0}+K_{00}&K_{10}\\T_v+K_{01}&\mu I_{X_1}+K_{11}\end{pmatrix},
		\qquad
		R'=\begin{pmatrix}\lambda' I_{X_0}+K'_{00}&K'_{10}\\T_{v'}+K'_{01}&\mu' I_{X_1}+K'_{11}\end{pmatrix},
	\]
	where every \(K_{ij},K'_{ij} (i,j=0,1)\) is compact.
	Since \(\Kcal\) is a two-sided operator ideal, the product is represented by
	\begin{equation}\label{eq:cal-mul}
		\begin{pmatrix}\lambda I_{X_0}&0\\T_v&\mu I_{X_1}\end{pmatrix}
		\begin{pmatrix}\lambda' I_{X_0}&0\\T_{v'}&\mu' I_{X_1}\end{pmatrix}=
		\begin{pmatrix}
			\lambda\lambda' I_{X_0}&0\\
			T_{v\lambda'+\mu v'}&\mu\mu' I_{X_1}
		\end{pmatrix}.
	\end{equation}
	Linearity of \(v\mapsto T_v\) gives the lower-left entry in the equation~(\ref{eq:cal-mul}).
	This is exactly \eqref{eq:tri-prd}, and therefore \(\Phi\) is a Banach-algebra isomorphism.
\end{proof}

\part{Oracle-recursive finite extensions}

Part~I has three steps.
Section~3 gives the local extension estimate with constants independent of the fibre dimensions.
Section~4 defines the protected module and the finite separation certificates used later.
Section~5 defines the finite requirement language and the oracle for the local structure of \(V\),
then runs the fair rank recursion.
Its final theorem records the analytic properties of the resulting two-sorted datum.
The construction is completed before any operator is chosen.

\section{Block Bourgain--Delbaen (BD) extensions}

\subsection{Operator-valued BD lemma}

The next lemma is a block-valued form of the extension estimate used in the BD and AH constructions \cite{BourgainDelbaen1980,ArgyrosHaydon2011}.
We use genuine finite-dimensional subspaces of \(V\) as source blocks.
They allow us to scalarize by \(v|_{V_r}\) without making a \(v\)-dependent choice from a finite net in \(B_{V_r^*}\).
This preserves the linearity of \(v\mapsto T_v\).

Let \((F_n)_{n\in\N_+}\) be finite-dimensional Banach spaces and put
\[ \mathcal{E}_*=\left(\bigoplus_{n=1}^\infty F_n^*\right)_{\ell_1}, \qquad \mathcal{E}=\left(\bigoplus_{n=1}^\infty F_n\right)_{\ell_\infty}. \]
Here and below, subscripts \(*\) denote chosen preduals; for notational convenience, superscripts * are retained for predual objects when no confusion can arise.
Write \(\iota_n:F_n^*\to{\mathcal{E}_*}\) for the canonical injection.
A unit block-triangular system is a family
\[ d_n^*\varphi=\iota_n\varphi-c_n^*\varphi,\qquad \varphi\in F_n^*, \]
where \(c_n^*:F_n^*\to\bigoplus_{k<n}F_k^*\).
We call \(c_n^*\varphi\) the correction functional. In the scalar case, \(c_n^*\) is the Bourgain--Delbaen extension functional in the terminology of Motakis \cite{Motakis2024}.
Denote by \(P_m^*\) the algebraic projection onto the first \(m\) ranges \(d_k^*(F_k^*)\),
along the later \(d^*\)-blocks.
We use the conventions \( P_0^*=0\) and \( c_1^*=0\).

Figure~\ref{fig:bd-triangular} displays the ambient components of \(d_r^*\), with rows indexed by \(r\) and columns by the ambient rank \(k\).
The diagonal entries are identities; each \(\bullet\) is a possibly zero component of \(-c_r^*\), and all entries with \(k>r\) vanish.

\begin{figure}[htbp]
	\centering
\begin{tikzpicture}[x=1.30cm,y=0.60cm,
	every node/.style={font=\small,inner sep=2pt}]
	\foreach \r in {1,...,5} {
		\foreach \k in {1,...,5} {
			\ifnum\k<\r
				\fill[blue!7] ({\k-1},{-\r}) rectangle (\k,{-\r+1});
			\else
				\ifnum\k=\r
					\fill[green!10] ({\k-1},{-\r}) rectangle (\k,{-\r+1});
				\else
					\fill[black!2] ({\k-1},{-\r}) rectangle (\k,{-\r+1});
				\fi
			\fi
		}
	}
	\draw[black!30,step=1] (0,-5) grid (5,0);
	\node at (2.5,1.25) {Ambient blocks};
	\node[anchor=east] at (-0.18,0.50) {BD maps};
	\foreach \k/\lab in {1/{F_1^*},2/{F_2^*},3/{F_3^*},4/{\cdots},5/{F_n^*}}
		\node at ({\k-0.5},0.50) {\(\lab\)};
	\foreach \r/\lab in {1/{d_1^*},2/{d_2^*},3/{d_3^*},4/{\vdots},5/{d_n^*}}
		\node[anchor=east] at (-0.18,{-\r+0.5}) {\(\lab\)};
	\foreach \r/\k/\entry in {
		1/1/{I_{F_1^*}},1/2/{0},1/3/{0},1/4/{\cdots},1/5/{0},
		2/1/{\bullet},2/2/{I_{F_2^*}},2/3/{0},2/4/{\cdots},2/5/{0},
		3/1/{\bullet},3/2/{\bullet},3/3/{I_{F_3^*}},3/4/{\cdots},3/5/{0},
		4/1/{\vdots},4/2/{\vdots},4/3/{\vdots},4/4/{\ddots},4/5/{\vdots},
		5/1/{\bullet},5/2/{\bullet},5/3/{\bullet},5/4/{\cdots},5/5/{I_{F_n^*}}}
		\node at ({\k-0.5},{-\r+0.5}) {\(\entry\)};
\end{tikzpicture}
	\caption{The unit lower triangular BD structure.}
	\label{fig:bd-triangular}
\end{figure}

We call the corrections in \eqref{eq:blk-t0} and \eqref{eq:blk-t1} type zero and type one, respectively. Their scalar versions correspond to Types~1 and~2 in \cite[Definition~3.3]{ArgyrosHaydon2011}. Motakis \cite[Proposition~3.2]{Motakis2024} uses types~(a) and~(b) for the analogous scalar corrections without and with a predecessor term, respectively.

\begin{lemma}[Block BD lemma]\label{lem:blk-bd}
	Fix \(0<\theta<1/2\).
	Suppose every new block \(F_{n+1}\) is a finite \(\ell_\infty\)-sum of direct summands \(G\),
	and on each \(G^*\) the correction has one of the forms
	\begin{align}
		c_{n+1}^*|_{G^*}&=\beta(I_{{\mathcal{E}_*}}-P_k^*)B^*,
		\label{eq:blk-t0}\\
		c_{n+1}^*|_{G^*}&=\iota_jA^*+\beta(I_{{\mathcal{E}_*}}-P_k^*)B^*.
		\label{eq:blk-t1}
	\end{align}
	Here \(0\le k<n\), \(0\le\beta\le\theta\),
	\[ B^*:G^*\longrightarrow\bigoplus_{k<r\le n}F_r^* \]
	is a contraction into the canonical \(\ell_1\)-sum, and in \eqref{eq:blk-t1}, \(A^*:G^*\to F_j^*\), \(j\le k\),
	is a contraction.
	Then \(\big(d_n^*(F_n^*)\big)\) is a Schauder decomposition of \(\mathcal{E}_*\), and
	\[ \sup_m\norm{P_m^*}\le M=(1-2\theta)^{-1}. \]
	Its biorthogonal blocks \(d_n(F_n)\) span the Banach space
	\[ X=\overline{\operatorname{span}}\{d_n(F_n):n\in\N_+\}\subseteq\mathcal{E}. \]
	Let \(P_m=(P_m^*)^*|_X\) and \(P_{[a,b]}=P_b-P_{a-1}\), and then the ambient coordinate maps and FDD projections satisfy
	\[ e_n:X\longrightarrow F_n,\qquad \norm{e_n}\le1,\qquad \norm{P_m}\le M,\qquad \norm{P_{[a,b]}}\le2M \]
	for \(n,m\in\N_+\) and \(1\le a\le b\).
\end{lemma}

\begin{proof}
	On the algebraic direct sum, the rank-\(n\) coefficient of \(d_n^*\varphi\) is \(\varphi\).
	Thus the change from the ambient blocks \(\iota_n(F_n^*)\) to \(d_n^*(F_n^*)\) is unit lower triangular.
	Hence it is invertible on every finite initial sum.
	Consequently, the algebraic projections \(P_m^*\) are well defined.
	Moreover, we have
	\begin{equation}\label{eq:bd-prj-cmp}
		P_m^*P_k^*=P_{\min\{m,k\}}^*.
	\end{equation}

	Assume inductively that \(\|P_s^*\|\le M\) on the first \(n\) ambient blocks, and take \(m<n+1\).
	Since \(P_m^*d_{n+1}^*\varphi=0\), we know
	\[ P_m^*\iota_{n+1}\varphi=P_m^*c_{n+1}^*\varphi. \]
	On a type-one summand, for \(\varphi\in G^*\), \eqref{eq:bd-prj-cmp} gives
	\[
		\begin{aligned}
			P_m^*c_{n+1}^*\varphi
			&=P_m^*\bigl(\iota_jA^*\varphi+\beta(I_{\mathcal{E}_*}-P_k^*)B^*\varphi\bigr)\\
			&=P_m^*\iota_jA^*\varphi
			+\beta(P_m^*-P_m^*P_k^*)B^*\varphi\\
			&=P_m^*\iota_jA^*\varphi
			+\beta(P_m^*-P_{m\wedge k}^*)B^*\varphi.
		\end{aligned}
	\]
	If \(m\le k\), then \(m\wedge k=m\), so the second term vanishes and the inductive estimate gives
	\[
		\begin{aligned}
			\norm{P_m^*\iota_{n+1}\varphi}
			=\norm{P_m^*\iota_jA^*\varphi} \le M\norm{A^*\varphi}
			\le M\norm{\varphi}.
		\end{aligned}
	\]
	If \(m>k\), then \(m\wedge k=k\) and \(j\le k<m\).
	Unit triangularity gives \(P_m^*\iota_jA^*\varphi=\iota_jA^*\varphi\), and hence
	\[
		\begin{aligned}
			\norm{P_m^*\iota_{n+1}\varphi}
			&\le \norm{A^*\varphi}
			+\beta\norm{(P_m^*-P_k^*)B^*\varphi}\\
			&\le \bigl(1+\beta(\norm{P_m^*}+\norm{P_k^*})\bigr)\norm{\varphi}\\
			&\le(1+2\beta M)\norm{\varphi}
			\le(1+2\theta M)\norm{\varphi}
			=M\norm{\varphi}.
		\end{aligned}
	\]
	On a type-zero summand there is no predecessor term, so
	\[
		\begin{aligned}
			\norm{P_m^*\iota_{n+1}\varphi}=\norm{\beta(P_m^*-P_{m\wedge k}^*)B^*\varphi}
			\le\beta\bigl(\norm{P_m^*}+\norm{P_{m\wedge k}^*}\bigr)\norm{\varphi}
			\le2\beta M\norm{\varphi}
			\le M\norm{\varphi}.
		\end{aligned}
	\]
	Since \(F_{n+1}\) is an \(\ell_\infty\)-sum of the summands \(G\), every
	\(\varphi\in F_{n+1}^*\) has a decomposition \(\varphi=\sum_G\varphi_G\) with
	\(\norm{\varphi}=\sum_G\norm{\varphi_G}\). Therefore
	\[ \norm{P_m^*\iota_{n+1}\varphi} \le\sum_G\norm{P_m^*\iota_{n+1}\varphi_G} \le M\sum_G\norm{\varphi_G} =M\norm{\varphi}. \]
	Thus
	\begin{equation}\label{eq:blk-prj-bnd}
		\norm{P_m^*\iota_{n+1}\varphi}\le M\norm{\varphi}
		\qquad(\varphi\in F_{n+1}^*).
	\end{equation}
	If \(z\) is supported by the first \(n\) ambient blocks, the ambient
	\(\ell_1\)-sum is additive across the last block. Using the inductive estimate and
	\eqref{eq:blk-prj-bnd}, we obtain
	\[
		\begin{aligned}
			\norm{P_m^*(z+\iota_{n+1}\varphi)}
			&\le\norm{P_m^*z}+\norm{P_m^*\iota_{n+1}\varphi}\\
			&\le M\norm{z}+M\norm{\varphi}\\
			&=M\bigl(\norm{z}+\norm{\varphi}\bigr)
			=M\norm{z+\iota_{n+1}\varphi}.
		\end{aligned}
	\]
	For \(m=n+1\), the projection is the identity on the first \(n+1\) blocks.
	This closes the induction and proves \(\|P_m^*\|\le M\) on every finite ambient initial sum.
	Since finitely supported vectors are dense in \({\mathcal{E}_*}\), each \(P_m^*\) extends to \({\mathcal{E}_*}\),
	the same estimate holds there, and \(P_m^*z^*\to z^*\) for every \(z^*\in{\mathcal{E}_*}\) as \(m \to \infty\).
	Thus \((d_n^*(F_n^*))\) is a Schauder decomposition of \({\mathcal{E}_*}\).

	Then we have the unique expansion \(z^*=\sum_n d_n^*\varphi_n\).
	For \(x\in F_n\), define \(d_nx\in{\mathcal{E}_*^*}=\mathcal{E}\) by
	\[ \langle d_nx,z^*\rangle=\varphi_n(x). \]
	If \(Q_n^*=P_n^*-P_{n-1}^*\), then \(Q_n^*z^*=d_n^*\varphi_n\),
	and the rank-\(n\) ambient coordinate of \(d_n^*\varphi_n\) is \(\varphi_n\).
	Therefore
	\[ \norm{\varphi_n}\le\norm{Q_n^*z^*} \le2M\norm{z^*}. \]
	The identities
	\[ \langle d_nx,d_k^*\psi\rangle=\delta_{nk}\psi(x) \]
	give us the block biorthogonality.
	Let \(X=\overline{\operatorname{span}}\{d_n(F_n):n\in\N_+\}\subseteq\mathcal{E}\).
	For \(x\in F_n\) and every \(d_k^*\psi\), biorthogonality then gives
	\begin{align*}
		\langle (P_m^*)^*d_nx,d_k^*\psi\rangle
		=\langle d_nx,P_m^*d_k^*\psi\rangle=\mathbf{1}_{\{k\le m\}}\delta_{nk}\psi(x)
		=\mathbf{1}_{\{n\le m\}}
		\langle d_nx,d_k^*\psi\rangle.
	\end{align*}
	Since the \(\big(d_k^*(F_k^*)\big)\) span \({\mathcal{E}_*}\), this proves
	\[ (P_m^*)^*d_nx=\mathbf{1}_{\{n\le m\}}d_nx. \]
	Thus we have
	\[ (P_m^*)^*(X)\subseteq X,\qquad P_m=(P_m^*)^*|_X,\qquad \operatorname{ran}P_m=\bigoplus_{n=1}^m d_n(F_n),\qquad \norm{P_m}\le\norm{P_m^*}\le M. \]
	For the interval projection \(P_{[a,b]}=P_b-P_{a-1}\), we have
	\[ \norm{P_{[a,b]}}\le\norm{P_b}+\norm{P_{a-1}}\le2M. \]
	Finally the ambient coordinate projection \(e_n:\mathcal{E}\to F_n\) is contractive in the \(\ell_\infty\)-sum norm and
	its restriction to \(X\) is therefore contractive as claimed.
\end{proof}

\begin{remark}\label{rem:vec-blk}
	\emergencystretch=1em
	The bound in Lemma~\ref{lem:blk-bd} does not depend on \(\dim F_n\).
	Hence the finite-dimensional spaces \(V_r\) can be kept as vector fibres.
\end{remark}

\subsection{Parameters and weight pools}

Retain the finite-dimensional exhaustion \((V_r)\), its compatible ordered bases,
and the rational core \(V_{\Q}\) fixed in Subsection~\ref{sec:coe-fib}.
All certificates and all finite ambient block functionals used as construction data are rational with respect to these finite-dimensional structures.
Rational points are used for the countable scheduling of requirements and
the blocks \(V_r\) carry the norm inherited from \(V\).

Fix any integer \(m_1\ge2^8\). Recursively choose each
\(m_{j+1}\) to be the least integer satisfying
\begin{equation}\label{eq:par-grw}
	m_1\ge2^8,\qquad m_{j+1}\ge m_j^5,
\end{equation}
and, writing
\[
	\begin{aligned}
		A_1&=1,& A_j&=4\max_{i<j}n_i\quad(j\ge2),\\
		D_j&=\lceil2\log_2m_j\rceil,&
		L_j&=\lceil m_j\rceil(1+A_j+\cdots+A_j^{D_j+2}),
	\end{aligned}
\]
choose \(n_j\) to be the least integer satisfying
\begin{equation}\label{eq:n-grw}
	n_j\ge 2^{j+8}L_jm_j^4.
\end{equation}
For later use in Part~II, set
\begin{equation}\label{eq:par-led-def}
	S_j=1+A_j+\cdots+A_j^{D_j+2},
	\qquad
	\Lambda_j=2^{j+12}S_jn_jm_j^4.
\end{equation}
The construction constants are fixed once and for all by
\begin{equation}\label{eq:cst-led}
	\begin{aligned}
		M&=(1-2m_1^{-1})^{-1},&
		\kappa&=2M,&
		C_{\rm tr}&=2+2\kappa,\\
		B&=4\kappa+4,&
		A_{\rm ex}&=3B+M,&
		C_{\rm off}&=8B\kappa,\\
		K_{\rm dep}&=32\kappa(C_{\rm off}+1).
	\end{aligned}
\end{equation}
Their list of dependencies is:
\begin{center}
	\begin{tabular}{c|c|c}
		constant & definition & dependence \\ \hline
		\(M\) & \((1-2m_1^{-1})^{-1}\) & \(m_1\) only \\
		\(\kappa,C_{\rm tr},B\) & \eqref{eq:cst-led} & \(M\) only \\
		\(A_{\rm ex}\) & \(3B+M\) & \(M,B\) only \\
		\(C_{\rm off}\) & \(8B\kappa\) & \(B,\kappa\) only \\
		\(K_{\rm dep}\) & \(32\kappa(C_{\rm off}+1)\) &
		\(\kappa,C_{\rm off}\) only
	\end{tabular}
\end{center}
These constants are universal and depend only on \(m_1\).

\begin{lemma}[{List of parameter estimates}]\label{lem:par-led}
	For every \(j\ge1\), we have
	\begin{gather}
		\sum_{r=1}^\infty m_r^{-1}<1,\qquad
		n_j\ge m_j^2,\qquad
		S_j\le L_j/m_j<n_j,
		\label{eq:led-elm}\\
		\frac{S_j}{n_j}\le2^{-j-8}m_j^{-5},
		\qquad
		\frac{S_j}{\Lambda_j}\le2^{-j-12}m_j^{-4},
		\label{eq:led-rat}\\
		m_1^{-(D_j+2)}\le m_j^{-2},\qquad
		m_{j+1}^{-1}\le m_j^{-5},\qquad
		(n_jm_j)^{-1}\le m_j^{-3}.
		\label{eq:led-dpt-tal}
	\end{gather}
	Furthermore, if \(d\in\N\) and an index \(p\) satisfies
	\[ m_p\ge2^{j+d+10}L_jn_jm_j^4, \]
	then \(m_p\ge\Lambda_j\).
\end{lemma}

\begin{proof}
	Since \(m_1\ge2^8\) and \(m_{r+1}\ge m_r^5\), the series is bounded by \(2^{-8}+2^{-40}+2^{-200}+\cdots<1\).
	The definition of \(L_j\), integrality of \(m_j\), and \eqref{eq:n-grw} give \(L_j=m_jS_j\), \(n_j\ge m_j^2\),
	and the remaining assertions in \eqref{eq:led-elm}.
	Since \(L_j=m_jS_j\), \eqref{eq:n-grw} gives
	\[ \frac{S_j}{n_j}\le 2^{-j-8}m_j^{-5}. \]
	Moreover, \eqref{eq:par-led-def} and \(n_j\ge1\) give
	\[ \frac{S_j}{\Lambda_j} =\frac{1}{2^{j+12}n_jm_j^4} \le 2^{-j-12}m_j^{-4}. \]
	This proves \eqref{eq:led-rat}.
	Since \(D_j\ge2\log_2m_j\) and \(m_1\ge2\),
	\[ m_1^{-(D_j+2)}\le2^{-D_j}\le m_j^{-2}. \]
	The other two estimates in \eqref{eq:led-dpt-tal} follow from \eqref{eq:par-grw} and \(n_j\ge m_j^2\).
	Finally, \(L_j=m_jS_j\) and \(m_j\ge2^8\) show that the last displayed lower bound dominates \(2^{j+12}S_jn_jm_j^4=\Lambda_j\).
\end{proof}
Partition the weight indices into pairwise disjoint infinite recursive sets
\[ \mathcal{G},\quad \mathcal{C}_{00},\quad\mathcal{C}_{11}, \quad\mathcal{C}_{01},\quad\mathcal{R}. \]
The set \(\mathcal{G}\) is used for generic averaging chains.
The three \(\mathcal{C}\)-sets carry source-self, target-self, and forward histories. The coding is a modified version of the odd-weight special coding in the AH construction \cite{ArgyrosHaydon2011}.
The set \(\mathcal{R}\) is source-only and is used to prove compactness of the reverse corner.
Each \(\mathcal{C}\)-set is split once more into seed and coded weights.

\section{\texorpdfstring{Protected module}{Protected module}}

This section works conditionally with a finite two-sorted block-triangular system whose corrections have the three forms displayed below.
All identities in the present section are finite and algebraic.
The rank recursion in Section~\ref{sec:two-dat} constructs a system satisfying these hypotheses and then completes the definitions of \(X_0\) and \(X_1\).
In this section, rank may be viewed simply as the stage at which a label is introduced or  as its FDD index equivalently.

Target blocks are finite \(\ell_\infty\)-sums of scalar pieces.
For every target label \(\alpha\), its paired source atom is denoted by \(\widehat{\alpha}\),
has the same rank as \(\alpha\), and carries the vector fibre \(E_\alpha=V_{r(\alpha)}\).
Fix one such target label \(\gamma\).
Also, we shall call \(\big(\widehat{\alpha}\big)_\alpha\) the protected mirrors.
Require \(E_\eta\subseteq E_\gamma\) for every old target label \(\eta\) used below,
including the predecessor \(\xi\) in the type-one case.
Write
\[ J_{\eta\gamma}:E_\eta\longrightarrow E_\gamma,\qquad R_{\eta\gamma}=J_{\eta\gamma}^*:E_\gamma^*\longrightarrow E_\eta^* \]
for inclusion and restriction, and let \(\epsilon_{\widehat{\eta}}^{0*}:E_\eta^*\to{\mathcal{E}_{0,*}}\) be the canonical source ambient-coordinate injection.

Fix \(k<n=\operatorname{rank}\gamma\), finitely many target labels \(\eta\) of ranks strictly between \(k\) and \(n\),
and scalars \(a_\eta\) with \(\sum_\eta|a_\eta|\le1\).
Put
\[ I=\{k+1,\ldots,n-1\},\qquad B^*=\sum_\eta a_\eta e_\eta^{1*}. \]
Since \(B^*\) has ambient support in \(I\), unit triangularity gives
\begin{equation}\label{eq:prt-tal-win}
	P_I^{1*}B^*=(I_{{\mathcal{E}_{1,*}}}-P_k^{1*})B^*.
\end{equation}
Suppose that the scalar target correction on \(\gamma\) is, respectively,
\begin{align}
	c_\gamma^{1*}&=0, \label{eq:tgt-ter}\\
	c_\gamma^{1*}&=\beta P_I^{1*}\sum_\eta a_\eta e_\eta^{1*},
	\label{eq:tgt-t0}\\
	c_\gamma^{1*}&=e_\xi^{1*}
	+\beta P_I^{1*}\sum_\eta a_\eta e_\eta^{1*},
	\label{eq:tgt-t1}
\end{align}
Here \(h\ge1\), \(\beta=m_h^{-1}\), and, in \eqref{eq:tgt-t1}, \(\xi\) is an old target label of rank \(k\).
On the paired source vector piece, for every \(\varphi\in E_\gamma^*\), define
\begin{align}
	c_{\widehat{\gamma}}^{0*}\varphi&=0,
	\label{eq:src-lft-ter}\\
	c_{\widehat{\gamma}}^{0*}\varphi
	&=\beta P_I^{0*}\sum_\eta a_\eta
	\epsilon_{\widehat{\eta}}^{0*}(R_{\eta\gamma}\varphi),
	\label{eq:src-lft-t0}\\
	c_{\widehat{\gamma}}^{0*}\varphi
	&=\epsilon_{\widehat{\xi}}^{0*}(R_{\xi\gamma}\varphi)
	+\beta P_I^{0*}\sum_\eta a_\eta
	\epsilon_{\widehat{\eta}}^{0*}(R_{\eta\gamma}\varphi).
	\label{eq:src-lft-t1}
\end{align}
It follows that
\[ \norm{\sum_\eta a_\eta \epsilon_{\widehat{\eta}}^{0*}(R_{\eta\gamma}\varphi)}_{\ell_1} \le\sum_\eta|a_\eta|\,\norm{R_{\eta\gamma}\varphi} \le\norm{\varphi}. \]
The first term in \eqref{eq:src-lft-t1} is the canonical injection of a contraction into one old block.
It has norm at most one, and \(\beta\le m_1^{-1}<1/2\).
Thus these are the two correction types in Lemma~\ref{lem:blk-bd}.
The projected terms are controlled by the projection estimate in that lemma as well.

Let \(X_0^{\rm alg}\) denote the algebraic source FDD, and let \(D_\gamma:X_0^{\rm alg}\to E_\gamma\) be the FDD coefficient map determined by
\[ \varphi(D_\gamma x) =d_{\widehat{\gamma}}^{0*}(\varphi)(x). \]
Let \(e_{\widehat{\gamma}}^0:X_0\to E_\gamma\) be the ambient source-coordinate map.
Pairing \eqref{eq:src-lft-ter}--\eqref{eq:src-lft-t1} with an arbitrary \(\varphi\in E_\gamma^*\) gives
\begin{align}
	e_{\widehat{\gamma}}^0
	&=D_\gamma, \label{eq:lft-ter}\\
	e_{\widehat{\gamma}}^0
	&=D_\gamma+\beta\sum_\eta a_\eta
	J_{\eta\gamma}e_{\widehat{\eta}}^0P_I^0,
	\label{eq:lft-t0}\\
	e_{\widehat{\gamma}}^0
	&=D_\gamma+J_{\xi\gamma}e_{\widehat{\xi}}^0
	+\beta\sum_\eta a_\eta
	J_{\eta\gamma}e_{\widehat{\eta}}^0P_I^0.
	\label{eq:lft-t1}
\end{align}
Put
\( U_\gamma:=e_{\widehat{\gamma}}^0. \)
Since this is an ambient coordinate evaluation, we have
\begin{equation}\label{eq:u-bnd}
	\sup_\gamma\norm{U_\gamma x}_V\le\norm{x}_{X_0}.
\end{equation}

On the algebraic source FDD define
\begin{equation}\label{eq:def-tv}
	T_vx=\sum_\gamma v(D_\gamma x)d_\gamma^1,\qquad v\in V^*,
\end{equation}
Thus \(T_v\) depends only on the coefficients in the paired fibre \(E_\gamma\).
If \(x\in X_0^{\rm alg}\), only finitely many FDD coefficients \(D_\gamma x\) are nonzero,
so the sum in \eqref{eq:def-tv} is finite.
Moreover,
\[ D_\gamma(P_I^0x) =\mathbf{1}_{\{\operatorname{rank}\gamma\in I\}}D_\gamma x, \qquad P_I^1d_\gamma^1 =\mathbf{1}_{\{\operatorname{rank}\gamma\in I\}}d_\gamma^1. \]
The paired atoms have the same ranks. And by summing these identities, we know \(P_I^1T_v=T_vP_I^0\) on \(X_0^{\rm alg}\).
The norm estimate below therefore extends \(T_v\) uniquely from \(X_0^{\rm alg}\) to all of \(X_0\).

The protected lift was chosen so that target evaluation commutes with scalarization by every \(v\in V^*\).
This is where the vector fibres are used.

\begin{lemma}[Scalarization identity]\label{lem:int}
	For every target node \(\gamma\), every finite source vector \(x\), and every \(v\in V^*\),
	\begin{equation}\label{eq:int}
		e_\gamma^{1*}(T_vx)=v(U_\gamma x).
	\end{equation}
	Consequently \(\|T_v\|\le\|v\|\).
\end{lemma}

\begin{proof}
	We induct on the rank of \(\gamma\).
	If \(\gamma\) is terminal, then \(e_\gamma^{1*}=d_\gamma^{1*}\).
	The definition of \(T_v\) and \eqref{eq:lft-ter} give
	\[ e_\gamma^{1*}(T_vx) =d_\gamma^{1*}(T_vx) =v(D_\gamma x) =v(U_\gamma x). \]

	Suppose next that \(\gamma\) has the type-zero correction \eqref{eq:tgt-t0}.
	All labels \(\eta\) in the finite sum have smaller rank, so the inductive hypothesis can apply to them.
	By \(P_I^1T_v=T_vP_I^0\) and \eqref{eq:lft-t0}, we obtain
	\begin{align*}
		e_\gamma^{1*}(T_vx)
		&=d_\gamma^{1*}(T_vx)
		+\beta\sum_\eta a_\eta
		e_\eta^{1*}(P_I^1T_vx)\\
		&=v(D_\gamma x)
		+\beta\sum_\eta a_\eta
		v(U_\eta P_I^0x)\\
		&=v\!\left(D_\gamma x
		+\beta\sum_\eta a_\eta
		J_{\eta\gamma}U_\eta P_I^0x\right)
		=v(U_\gamma x).
	\end{align*}

	Finally suppose that \(\gamma\) has the type-one correction \eqref{eq:tgt-t1}.
	The label \(\xi\) and all labels \(\eta\) in the finite sum again have smaller rank. By the same argument and \eqref{eq:lft-t1}, we have
	\begin{align*}
		e_\gamma^{1*}(T_vx)
		&=d_\gamma^{1*}(T_vx)+e_\xi^{1*}(T_vx)
		+\beta\sum_\eta a_\eta
		e_\eta^{1*}(P_I^1T_vx)\\
		&=v(D_\gamma x)+v(U_\xi x)
		+\beta\sum_\eta a_\eta v(U_\eta P_I^0x)\\
		&=v\!\left(D_\gamma x+J_{\xi\gamma}U_\xi x
		+\beta\sum_\eta a_\eta
		J_{\eta\gamma}U_\eta P_I^0x\right)
		=v(U_\gamma x),
	\end{align*}
	This completes the induction and proves \eqref{eq:int} for every target coordinate.
	Finally, by \eqref{eq:u-bnd},
	\[ \norm{T_vx}_{X_1} =\sup_\gamma\abs{e_\gamma^{1*}(T_vx)} \le\norm{v}\sup_\gamma\norm{U_\gamma x}_V \le\norm{v}\norm{x}, \]

\end{proof}

For a finite scalar target functional
\( h=\sum_{\eta\in F}a_\eta e_\eta^{1*}, \)
let
\begin{equation}\label{eq:uh}
	U_hx=\sum_{\eta\in F}a_\eta U_\eta x\in V.
\end{equation}
The summands are viewed in a common finite-dimensional fibre \(V_r\).
Lemma~\ref{lem:int} then gives
\begin{equation}\label{eq:h-int}
	h(T_wx)=w(U_hx)\qquad(w\in V^*).
\end{equation}

Recall that a finite dimensional subspace is rational if it has a basis of rational vectors.
At a finite stage, let \(x\) be a rational vector in one of \(X_0, X_1\).
For a corner \(i\to j\), define its finite dimensional orbit by
\begin{equation}\label{eq:fou-orb}
	\begin{aligned}
		\mathscr{M}_{00}(x)&=\K x,&
		\mathscr{M}_{11}(x)&=\K x,\\
		\mathscr{M}_{01}(x)&=\{T_vx:v\in V^*\},&
		\mathscr{M}_{10}(x)&=\{0\}.
	\end{aligned}
\end{equation}
These are finite-dimensional rational subspaces whenever \(x\) is a rational algebraic vector.
Only the forward case needs verification.
For \(x\in X_0^{\rm alg}\), put
\[ F=\{\gamma:D_\gamma x\ne0\}, \qquad E_x=\operatorname{span}\{D_\gamma x:\gamma\in F\}\subseteq V. \]
The set \(F\) is finite, and \(E_x\) is finite dimensional.
By the Hahn--Banach theorem, the restriction map \(V^*\to E_x^*\) is onto, and therefore
\[ \mathscr{M}_{01}(x) =\left\{\sum_{\gamma\in F}a(D_\gamma x)d_\gamma^1: a\in E_x^*\right\}. \]
This is the range of the finite-dimensional linear map
\[ E_x^*\longrightarrow\operatorname{span}\{d_\gamma^1:\gamma\in F\}, \qquad a\longmapsto\sum_{\gamma\in F}a(D_\gamma x)d_\gamma^1. \]
When \(x\) is rational, the vectors \(D_\gamma x\) have rational coordinates in a common \(V_r\),
so the map has a representation by a rational matrix.
Gaussian elimination shows that its range is a rational finite-dimensional subspace.

For a finite interval \(I\), write \(F_n^j\) for the \(n\)-th ambient FDD block used in the construction of \(X_j\),
and put
\[ \mathcal{E}_j(I)=\left(\bigoplus_{n\in I}F_n^j\right)_{\ell_\infty}, \qquad {\mathcal{E}_{j,*}(I)}=\left(\bigoplus_{n\in I}(F_n^j)^*\right)_{\ell_1}. \]
Thus the ambient predual norm is an \(\ell_1\)-sum of the actual block-dual norms;
when all blocks are scalar it is the usual scalar \(\ell_1\)-norm.
For brevity, \(\|h\|_{\ell_1}\) for any later finite ambient functional always means this block \(\ell_1\)-sum norm.
A \emph{separation certificate} is a tuple
\[ \mathfrak{c}=(i,j,x,z,h,I,L,\varepsilon), \qquad L=\mathscr{M}_{ij}(x), \]
where \(I\subseteq\N_+\) is a finite rank interval, \(\varepsilon\in{\Q_+}\), and
\[ \supp_{\rm FDD}z\cup\supp_{\rm FDD}L\subseteq I, \qquad \supp_{\rm FDD}L:=\bigcup_{y\in L}\supp_{\rm FDD}y. \]
Moreover, \(h\in{\mathcal{E}_{j,*}(I)}\) is rational,
\begin{equation}\label{eq:sep-crt}
	\norm{h}_{{\mathcal{E}_{j,*}(I)}}\le1,\qquad h|_L=0,\qquad
	\operatorname{Re}h(z)>\varepsilon
\end{equation}
for the rational output \(z\) which is to be separated.

\section{Two-sorted BD recursion}
\label{sec:two-dat}

The Argyros--Haydon construction builds on the coding of special histories in the work of Gowers and Maurey \cite[Sections~1 and~3]{GowersMaurey1993}.
In the AH operator argument, one assumes that a bounded operator \(T\) satisfies \(\dist(Tx_k,\mathbb R x_k)\ge\varepsilon>0\) along a RIS.
For arbitrarily large \(j\), the argument constructs a mean \(\bar z\) of a dependent sequence such that
\[ \norm{\bar z}\le \frac{C}{m_{2j-1}^{2}}, \qquad \norm{T\bar z}\ge \frac{c}{m_{2j-1}}, \]
where \(C,c>0\) are independent of \(j\).
The odd index \(2j-1\) marks an outer operation restricted by history coding; even weights provide the rich averaging operations used in the inner constructions.
If the same outer scale admitted the unrestricted averaging operations available at even weights, the argument of \cite[Proposition~4.8]{ArgyrosHaydon2011} would give a lower bound of order \(m_{2j-1}^{-1}\) for \(\norm{\bar z}\), since the selected blocks have norms bounded away from zero.
Comparison with competing coded histories gives the first bound, while the chosen outer evaluation gives the second. Together they contradict the boundedness of \(T\).
At each step, the actual predecessor determines the next coded weight, and its rank fixes the lower support cutoff for the next block.
The inner exact pair must then be joined to the outer chain across the corresponding rank interval.
In this method, the rank sets must therefore be specified completely so that each recursive step is well defined.
This fixed grammar ensures that every admitted chain is legal and that no legal continuation required by the proof is omitted \cite[Sections~4, 6, and~7]{ArgyrosHaydon2011}.

Here, inner realization and outer continuation are specified by separate finite requirements, \hyperref[sec:req-cat]{\textup{(R3)}} and \hyperref[sec:req-cat]{\textup{(R4)}}.
A compatible finite certificate records the data required for the outer step, and a fair no-injury schedule eventually realizes every admissible requirement whose dependencies are available.
The resulting finite-extension property allows later operator arguments to select suitable realizations inside the completed datum.
The procedure is uniform in \(V\), with \(\mathcal O_V\) supplying the finite-dimensional information needed to decide local admissibility.

\noindent\textbf{Construction architecture.}
The main components have the following roles.
\begin{center}
	{\small
	\renewcommand{\arraystretch}{1.12}
	\begin{tabular}{@{}l l l@{}}
		\toprule
		Layer & Main objects & Role \\
		\midrule
		Static objects
			& \parbox[t]{0.29\textwidth}{\raggedright atoms, fibres, records, structural and coding histories}
			& \parbox[t]{0.45\textwidth}{\raggedright They form the finite state. Once introduced, they remain fixed.} \\
		Dynamic requests
			& \parbox[t]{0.29\textwidth}{\raggedright \hyperref[sec:req-cat]{\textup{(R1)}--\textup{(R5)}}, admissibility, the fair scheduler}
			& \parbox[t]{0.45\textwidth}{\raggedright Each rank adds a probe pair and realizes at most one active requirement. Cutoff and copy parameters provide cofinal realizations.} \\
		Analytic objects
			& \parbox[t]{0.29\textwidth}{\raggedright evaluations, represented packets, separation certificates}
			& \parbox[t]{0.45\textwidth}{\raggedright Packets define corrections; evaluations read the atoms; certificates supply the data for continuations.} \\
		Use in Part~II
			& \parbox[t]{0.29\textwidth}{\raggedright \hyperref[itm:ana-g1]{\textup{(G1)}}--\hyperref[itm:ana-g7]{\textup{(G7)}} and \hyperref[itm:ana-ea1]{\textup{(EA1)}}--\hyperref[itm:ana-ea4]{\textup{(EA4)}}}
			& \parbox[t]{0.45\textwidth}{\raggedright These are the construction's analytic consequences used in the operator arguments.} \\
		\bottomrule
	\end{tabular}}
\end{center}

Subsections~\ref{sec:fin-req} and~\ref{sec:stg-cns} define the finite requirements and prove the finite-extension property.
The evaluation and structural properties then lead to Theorem~\ref{thm:dat-cns}.
Readers interested primarily in Part~II may proceed to the analytic inputs in Subsection~\ref{sec:ana-inp}.

\subsection{Finite states and requirements}\label{sec:fin-req}
We now construct the datum from which the two spaces are obtained.
At each rank, an admissible extension depends on finitely many earlier records and on the local structure of \(V\) carried by the finite-dimensional fibres \(V_r\).
The terms chain, cell, stem, and restriction refer to Definitions~\ref{def:fml-pkt}--\ref{def:crt-stm}.
All requirements and finite dependency strings belong to the fixed countable language below.
The fair rank recursion is completed before Part~II.
Later operator arguments may select atoms already present, but they do not change the datum.
The same datum simultaneously determines the family \((T_v)_{v\in V^*}\).

We use the finite-extension method in an oracle-relative form;
see \cite[Section~V.2]{Odifreddi1989}.
In its usual form, the conditions are finite strings ordered by extension. Equivalently, by \cite[Proposition~V.3.13]{Odifreddi1989}, we may regard infinite end extensions as a Baire space whose basic open sets are determined by finite conditions.
Eventual satisfaction of each activated requirement is dense open above its activating condition, so branches satisfying every activated requirement form a comeager set.
We do not use this category argument below, but instead obtain one such branch directly from persistence, no injury, and least-code fair scheduling.

Here a finite condition is the two-sorted datum below a given rank, ordered by end extension.
Since the conditions and requirements are finitely coded, the same scheduling principle applies.
The legality of each BD extension is verified in Lemma~\ref{lem:req-adm},
and fairness for the conditional requirements is verified in Lemma~\ref{lem:fai-per}.

\begin{remark}[Oracle-recursive form]\label{rem:orc-rec}
	Fix once and for all a recursive tagged syntax, independent of \(V\), for finite words built from a scalar-field tag,
	ranks, dimensions, rational coordinate vectors, rational matrices, rational subspaces, and finite lists.
	We identify these finite words with natural numbers by a fixed G\"odel coding \cite{Godel1931}.
	We use Turing's oracle terminology \cite{Turing1939}.
	Relative to the exhaustion and rational structures fixed in Subsection~\ref{sec:coe-fib},
	let \(\mathcal{O}_V\subseteq\N\) be the set consisting of the scalar-field tag and the tagged codes of the following true statements:
	\begin{enumerate}[label={\textup{(\arabic*)}}]
		\item \(\dim V_r=k\);
		\item a coded rational vector, matrix, subspace, or finite list is well typed for the indicated fibres \(V_r\), their duals, or the finite sums used below;
		\item \(\norm{x}\mathrel{\triangleleft}q\);
		\item \(\norm{A}\mathrel{\triangleleft}q\);
		\item \(\sum_{l=1}^s|a_l|\norm{A_l}\mathrel{\triangleleft}q\).
	\end{enumerate}
	Here \(r,k,s\in\N\), \(q\in\Q\), the \(a_l\)'s are rational scalars,
	\(\triangleleft\) is one of \(<,\leq,>,\geq\), and all displayed objects have well-typed rational codes.
	Thus \(\mathcal{O}_V\) records the rational finite-dimensional structure of the fibre system together with its norm and operator-norm information.
	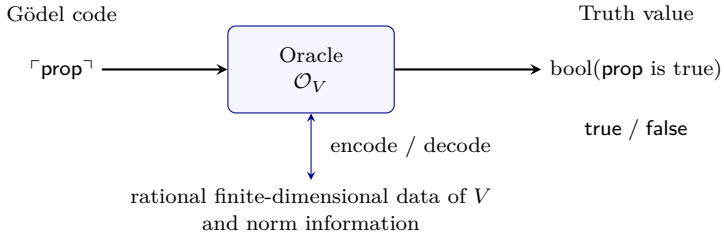
\begin{figure}[htbp]
		\centering
\begin{tikzpicture}[x=1cm,y=1cm,>=stealth,
	every node/.style={font=\small,align=center}]
	\node (input) at (0.9,0) {\(\ulcorner\mathsf{prop}\urcorner\)};
	\node[above=8pt] at (input.north) {G\"odel code};
	\node[draw=blue!55!black,fill=blue!4,rounded corners=3pt,
		minimum width=2.2cm,minimum height=1.15cm] (oracle) at (4.2,0)
		{Oracle\\\(\mathcal O_V\)};
	\node (output) at (8.5,0)
		{\(\operatorname{bool}(\mathsf{prop}\text{ is true})\)};
	\node[above=8pt] at (output.north) {Truth value};
	\node[below=8pt] at (output.south) {\(\mathsf{true}\;/\;\mathsf{false}\)};
	\draw[->,thick] (input.east) -- (oracle.west);
	\draw[->,thick] (oracle.east) -- (output.west);
	\node (vdata) at (4.2,-1.85)
		{rational finite-dimensional data of \(V\)\\and norm information};
	\draw[<->,blue!55!black] (vdata.north) -- node[right=4pt,text=black]
		{encode / decode} (oracle.south);
\end{tikzpicture}
		\caption{Oracle \(\mathcal{O}_V\).}
		\label{fig:oracle-query}
	\end{figure}

	Before rank \(n\) is constructed, the current state consists of all data from lower ranks.
	All such data are finite words over the fixed countable alphabet introduced below.
	Moreover, the five requirement types form a countable language of finite extensions.
	An admissibility test therefore reads only finitely many old records.
	Its algebraic, support, and rank tests are decided from their finite codes,
	whereas every remaining analytic test is one of the queries recorded by \(\mathcal{O}_V\).
	Hence the test makes only finitely many oracle queries.
	Once an admissible requirement is selected, the new fibres, records, histories,
	and correction maps are produced by a finite transition rule.

	At stage \(n\), only requirement codes at most \(n\) are checked,
	so each transition is a finite procedure.
	By Lemma~\ref{lem:req-adm}, every correction uses only lower-rank atoms,
	and a target atom and its simultaneous protected mirror do not depend on one another.
	Consequently, Construction~\ref{con:rnk-rec} is an oracle-recursive monotone finite-extension construction.
	Equivalently, there is a single oracle program \(\mathfrak{M}\), independent of \(V\),
	such that \(\mathfrak{M}^{\mathcal{O}_V}\) produces codes for the increasing sequence of finite initial data.
	This formulation does not assert that \(\mathcal{O}_V\) is computable or that the resulting presentation is computable without the oracle.
	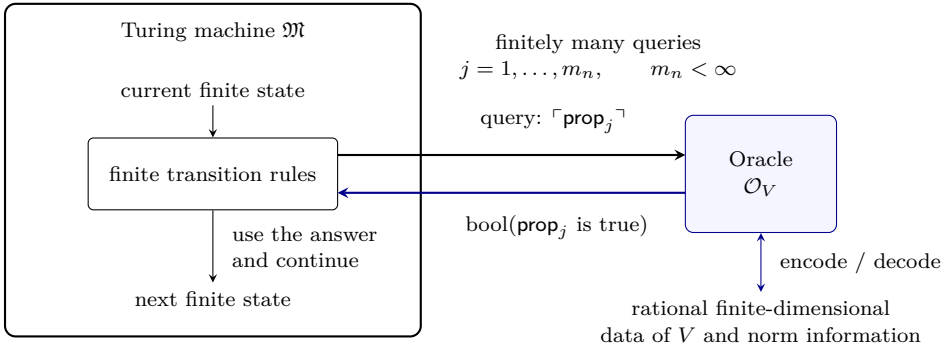
\begin{figure}[htbp]
		\centering
\begin{tikzpicture}[x=1cm,y=1cm,>=stealth,
	every node/.style={font=\small,align=center}]
	\draw[thick,rounded corners=4pt] (0,-2.15) rectangle (5.5,2.25);
	\node at (2.75,1.88) {Turing machine \(\mathfrak M\)};
	\node (input) at (2.75,1.12) {current finite state};
	\node[draw,rounded corners=2pt,minimum width=3.3cm,
		minimum height=0.95cm] (compute) at (2.75,0)
		{finite transition rules};
	\draw[->] (input.south) -- (compute.north);
	\node (output) at (2.75,-1.65) {next finite state};
	\draw[->] (compute.south) -- node[right=4pt,align=left]
		{use the answer\\and continue} (output.north);
	\node[draw=blue!55!black,fill=blue!4,rounded corners=3pt,
		minimum width=2cm,minimum height=1.55cm] (oracle) at (10,0)
		{Oracle\\\(\mathcal O_V\)};
	\node[align=center] (vdata) at (10,-1.95)
		{rational finite-dimensional\\data of \(V\) and norm information};
	\draw[<->,blue!55!black] (vdata.north) -- node[right=4pt,text=black]
		{encode / decode} (oracle.south);
	\node at (7.85,1.55)
		{finitely many queries\\\(j=1,\ldots,m_n,\qquad m_n<\infty\)};
	\draw[->,thick] (4.4,0.25) -- (9,0.25);
	\node[above=5pt] at (7.25,0.25)
		{query: \(\ulcorner\mathsf{prop}_j\urcorner\)};
	\draw[->,thick,blue!55!black] (9,-0.25) -- (4.4,-0.25);
	\node[below=5pt] at (7.3,-0.25)
		{\(\operatorname{bool}(\mathsf{prop}_j\text{ is true})\)};
\end{tikzpicture}
		\caption{Turing machine \(\mathfrak{M}\) and \(\mathfrak{M}^{\mathcal{O}_V}\)}
		\label{fig:oracle-machine-interaction}
	\end{figure}

\end{remark}

For a finite set \(A\) and finite-dimensional spaces \((G_\alpha)_{\alpha \in A}\), write
\[ \ell_\infty(A;G_\alpha) =\left(\bigoplus_{\alpha\in A}G_\alpha\right)_{\ell_\infty}, \qquad \ell_1(A;G_\alpha^*) =\left(\bigoplus_{\alpha\in A}G_\alpha^*\right)_{\ell_1}. \]
Each \(V_r\) is equipped with the rational structure induced by its ordered basis from Subsection~\ref{sec:coe-fib}.
The scalar types have their canonical coordinate structures,
and all other finite-dimensional structures below are built from these.
For each rank \(n\in\N_+\), there are finite sets
\[ \Delta_n^1,\qquad \Delta_n^{0,\mathrm{so}},\qquad \Delta_n^0= \{\widehat{\gamma}:\gamma\in\Delta_n^1\}\cup \Delta_n^{0,\mathrm{so}}. \]
Every \(\gamma\in\Delta_n^1\) is a scalar target atom and has a paired source atom \(\widehat{\gamma}\).
Its source fibre is \(E_\gamma=V_{r(\gamma)}\).
Every atom in \(\Delta_n^{0,\mathrm{so}}\) is scalar and source-only.
Thus
\begin{equation}\label{eq:fml-fn}
	F_n^1=\ell_\infty(\Delta_n^1;\K),
	\qquad
	F_n^0=
	\left(\bigoplus_{\gamma\in\Delta_n^1}E_\gamma
	\oplus
	\bigoplus_{\alpha\in\Delta_n^{0,\mathrm{so}}}\K
	\right)_{\ell_\infty}.
\end{equation}
For \(i=0,1\), let \(G_\alpha^i\) denote the fibre of the atom \(\alpha\in\Delta_n^i\), and let
\[ \epsilon_\alpha^{i*}:(G_\alpha^i)^*\longrightarrow(F_n^i)^* \]
be the canonical ambient injection.
Explicitly, \(\epsilon_\alpha^{i*}(\psi)((x_\beta)_\beta)=\psi(x_\alpha)\) and if \(G_\alpha^i=\K\), then \(1\in(G_\alpha^i)^*\) denotes the functional \(z\mapsto z\).
We use the same symbol after embedding \((F_n^i)^*\) into the finite ambient \(\ell_1\)-sum of all preceding blocks.

\begin{definition}[Represented operator packets]
	\label{def:fml-pkt}
	Let \(G\) be finite dimensional and suppose that ranks smaller than \(n\) have already been constructed.
	A represented packet from \(G^*\) to side \(i\) is an ordered finite list
	\[ \mathbf{b}=((a_l,\eta_l,R_l))_{l=1}^s, \qquad R_l:G^*\longrightarrow(G_{\eta_l}^i)^*, \qquad \eta_l\in \bigcup_{k<n}\Delta_k^i \quad(1\le l\le s), \]
	where the \(a_l\)'s and the \(R_l\)'s are rational and
	\begin{equation}\label{eq:pkt-ctr}
		\sum_{l=1}^s|a_l|\,\norm{R_l}\le1.
	\end{equation}
	Its value and raw rank support are
	\begin{equation}\label{eq:pkt-val-sup}
		B_{\mathbf{b}}^*\varphi
		=\sum_{l=1}^sa_l\epsilon_{\eta_l}^{i*}(R_l\varphi),
		\qquad
		\supp_{\rm raw}\mathbf{b}
		=\{\operatorname{rank}\eta_l:1\le l\le s\}.
	\end{equation}
	The list, including zero entries and their order, is part of the packet.
	Thus two represented packets can be different even if they induce the same operator.
	Interval restriction retains empty entries in their original positions.
\end{definition}

For a scalar fibre \(G=G^*=\K\), every rational finite ambient functional of raw \(\ell_1\)-norm at most one is represented in this way:
if \(b=\sum_la_l{\epsilon_{\eta_l}^{i*}(\psi_l)}\), take \(R_l(\lambda)=\lambda\psi_l\), and we have \(B_\mathbf{b}^*(\lambda)=\sum_{l=1}^sa_l \epsilon_{\eta_l}^{i*}(R_l \lambda)=\lambda b, \) where \(\mathbf{b}=((a_l,\eta_l,R_l))_{l=1}^s\).
Moreover, \eqref{eq:pkt-ctr} guarantees that \(B_{\mathbf{b}}^*\) is a contraction into the raw ambient \(\ell_1\)-sum. Thus \(B_\mathbf{b}^*(1)=b\) and every such functional is a special case of a represented operator packet.
And when \(G_\eta^i=\K\), a one-entry packet representing the ambient coordinate \(\epsilon_\eta^{i*}(1)\) is denoted by \(\langle\eta\rangle\).

The rank, weight, age, predecessor, and type-zero/type-one fields below are standard BD--AH chain data \cite{ArgyrosHaydon2011,Tarbard2013,Motakis2024}.
We next give every atom a record.
First we assume \(\alpha\) is not a probe which we will define later, and its \emph{core record} is
\begin{equation}\label{eq:atm-rec}
	\mathfrak{r}_0(\alpha)=
	\bigl(i,n,\mathsf{t},h,a,\chi,\xi,k,\mathbf{b},
	\mathsf{sem},o,G_\alpha^i\bigr).
\end{equation}
\begin{center}
	{\small
	\renewcommand{\arraystretch}{1.08}
	\begin{tabular}{@{}c p{0.76\textwidth}@{}}
		\toprule
		Parameter & Meaning \\
		\midrule
		\(i\) & the side containing \(\alpha\) \\
		\(n\) & the rank at which \(\alpha\) is introduced \\
	\(\mathsf{t}\) & the chain type \\
		\(h\) & the root level, which determines the weight \(m_h^{-1}\) \\
	\(a\) & the age of \(\alpha\) in its chain \\
	\(\chi\) & the initial cutoff of the chain \\
		\(\xi\) & the stored predecessor identifier, or \(\varnothing\) at age one \\
		\(k\) & the local cutoff, namely the left endpoint of the current cell \((k,n)\) \\
		\(\mathbf{b}\) & the represented packet entering the correction of \(\alpha\) \\
		\(\mathsf{sem}\) & the tagged semantic record of the requirement \\
		\(o\) & the copy index \\
		\(G_\alpha^i\) & the finite-dimensional fibre assigned to \(\alpha\) \\
		\bottomrule
	\end{tabular}}
\end{center}
Fix once and for all an injective coding \((i,n,s)\mapsto\operatorname{id}(i,n,s)\)
from \(\{0,1\}\times\N_+\times\N\) into \(\N_+\), where \(i\), \(n\), and \(s\) denote the side, rank, and slot, respectively.
At rank \(n\), slot zero on each side is reserved for its probe;
all other atoms created at that rank are assigned positive slots in lexicographic order of side and requirement role.
If \(\alpha\) occupies slot \(s\) on side \(i\) at rank \(n\), set \(\operatorname{id}\alpha=\operatorname{id}(i,n,s)\).
Thus identifiers are assigned deterministically and are never reused.
The predecessor index \(\xi\) contains only the already assigned integer identifier of the predecessor on the same side.
Whenever rank, level, or age is applied to such an identifier, it means the unique old atom carrying that identifier.
After the core record has been formed, let
\begin{equation}\label{eq:nsl-his}
	\mathsf{shist}(\alpha)=
	\begin{cases}
		((\operatorname{id}\alpha,\mathfrak{r}_0(\alpha))),&a=1,\\
		\mathsf{shist}(\xi)^\frown
		((\operatorname{id}\alpha,\mathfrak{r}_0(\alpha))),&a>1,
	\end{cases}
\end{equation}
This is the \emph{structural history} of the same-side correction chain.
To separate it from the history used for coding, define
\begin{equation}\label{eq:two-his-fld}
	\mathsf{thist}(\alpha)=
	\begin{cases}
		\mathsf{shist}(\gamma),&\alpha=\widehat{\gamma}
		\text{ is a protected mirror},\\
		\mathsf{shist}(\alpha),&\alpha\text{ is target or source-only},
	\end{cases}
\end{equation}
The complete immutable record is finally
\begin{equation}\label{eq:cpl-atm-rec}
	\mathfrak{r}(\alpha)=
	\bigl(\mathfrak{r}_0(\alpha),
	\mathsf{shist}(\alpha),
	\mathsf{thist}(\alpha)\bigr).
\end{equation}
We use the notation
\[
	\begin{aligned}
		\operatorname{rank}\alpha&=n,& \operatorname{lev}(\alpha)&=h,&
		\operatorname{age}(\alpha)&=a,\\
		\chi(\alpha)&=\chi,& \operatorname{cut}(\alpha)&=k,&
		\mathbf{b}_\alpha&=\mathbf{b}.
	\end{aligned}
\]
for the corresponding index of \(\mathfrak{r}_0(\alpha)\).
If the predecessor field is nonempty, \(\operatorname{pred}(\alpha)\) denotes the unique old atom carrying the stored identifier \(\xi\).
For atoms \(\alpha,\beta\), write \(\beta\prec\alpha\) if \(\beta=\operatorname{pred}(\alpha)\).
Thus both histories are finite words of identifiers and core snapshots,
and neither contains the full record currently being defined.
The structural history is used for predecessor recursion, evaluation analysis, and interval restriction.
And the target coding history is the immutable comparison key used by the code.
They agree except on a protected mirror.
The fields satisfy the following compatibility conditions:
\begin{enumerate}[label=\textup{(D\arabic*)},leftmargin=*]
	\item \(i\in\{0,1\}\), \(n=\operatorname{rank}\alpha\), and \(\mathsf{t}\) is one of
	\[ \mathsf{generic},\quad \mathsf{inner},\quad \mathsf{outer},\quad \mathsf{reverse}. \]
	\item \(h\ge1\) is the root level, \(m_h^{-1}\) is the root weight, and \(1\le a\le n_h\) is the age.
	The correction kind is \(\mathsf{type~0}\) when \(a=1\) and \(\mathsf{type~1}\) when \(a>1\);
	for a probe atom it is \(\mathsf{terminal}\).
	\item \(\chi\) is the initial cutoff of the chain.
	If \(a=1\), then \(\xi=\varnothing\) and the local cutoff is \(k=\chi\).
	If \(a>1\), then \(\xi\) is the unique predecessor,
	\[ \operatorname{rank}\xi=k<n,\qquad {\operatorname{lev}(\xi)}=h,\qquad \operatorname{age}(\xi)=a-1, \]
	and \(\xi\) has the same root key as \(\alpha\), as defined below.
	A probe atom is never allowed as predecessor.
	\item The packet is supported in the open cell of the new edge:
	\begin{equation}\label{eq:pkt-cel}
		\supp_{\rm raw}\mathbf{b}\subseteq(k,n).
	\end{equation}
	The cell of \(\alpha\) is the labelled interval \(\mathsf{cell}(\alpha)=(k,n)\).
	For a chain \(\alpha_1\prec\cdots\prec\alpha_a\), its cells are therefore successive and pairwise disjoint.
	\item For a non-seed requirement \(\mathsf{Q}\), let \(\mathbf{v}(\mathsf{Q})\) be the ordered list, in the fixed field order, of all vectors of \(V_{\Q}\) occurring in \(\mathsf{Q}\), as specified in the requirement part below.
	Let \(\mathsf{s}(\mathsf{Q})\in \{\mathsf{target},\mathsf{source\mbox{-}only}\}\) be the output sort of the requirement. And
	for \textup{(R5)}, it is \(\mathsf{source\mbox{-}only}\).
	The semantic field is the following tuple (the symbols occurring in it are defined in the requirement part below):
	\begin{equation}\label{eq:sem-rec-tab}
		\mathsf{sem}(\alpha)=
		\begin{cases}
			(\mathsf{R}1,\mathsf{s}(\mathsf{Q}),\mathbf{v}(\mathsf{Q})),
			&\mathsf{Q}\text{ has type \textup{(R1)}},\\
			(\mathsf{R}3,c,\vartheta,\pi,\mathsf{s}(\mathsf{Q}),
			\mathbf{v}(\mathsf{Q})),
			&\mathsf{Q}\text{ has type \textup{(R3)}},\\
			(\mathsf{R}4,c,\vartheta,\pi,\mathfrak{c},\mathsf{s}(\mathsf{Q}),
			\mathbf{v}(\mathsf{Q})),
			&\mathsf{Q}\text{ has type \textup{(R4)}},\\
			(\mathsf{R}5,\mathsf{source\mbox{-}only},\mathbf{v}(\mathsf{Q})),
			&\mathsf{Q}\text{ has type \textup{(R5)}}.
		\end{cases}
	\end{equation}
	A protected mirror receives exactly the semantic tuple of its target partner. A probe has empty semantic field. And a seed requirement creates no atom.
	The copy index \(o\) distinguishes repeated realizations of the same rational requirement.
	\item Each atom \(\alpha\in\Delta_n^i\) indexes a coordinate summand \(G_\alpha^i\) of \(F_n^i\).
	Its intrinsic level is \(h\).
	Hence the sum of the fibres corresponding to atoms at any chosen set of levels is a contractive coordinate summand of the ambient block \(F_n^i\).
\end{enumerate}
The \emph{root key} just used is the part of a chain record which must remain fixed under a type-one continuation.
It is
\[
	\begin{cases}
		(i,\mathsf{generic},h,\chi),&\text{for a generic chain},\\
		(i,\mathsf{inner},h,\chi,\vartheta,\pi),
		&\text{for an inner chain attached to }(\vartheta,\pi),\\
		(i,\mathsf{outer},h,\chi,\vartheta),
		&\text{for an outer chain over the seed }\vartheta,\\
		(0,\mathsf{reverse},h,\chi),&\text{for a reverse chain}.
	\end{cases}
\]
The evolving outer stem and its certificates are not part of the outer root key.
A protected mirror has side entry zero and otherwise inherits the root key of its target.
For generic chains, equality of root keys is supplemented by the sort condition in \textup{(R1)}.
For the two rank-\(n\) probe atoms set \(o=n\).
A probe atom on side \(i\) then has
\[ \mathfrak{r}_0(\alpha)= (i,n,\mathsf{probe},0,0,n,\varnothing,n, \varnothing,\varnothing,o,G), \qquad \mathsf{shist}(\alpha)= ((\operatorname{id}\alpha,\mathfrak{r}_0(\alpha))). \]
For the target probe \(\gamma\), put \(\mathsf{thist}(\gamma)=\mathsf{shist}(\gamma)\); for its paired source probe,
put \(\mathsf{thist}(\widehat{\gamma})=\mathsf{shist}(\gamma)\).
Both probes have weight zero and zero correction.
And no other weight-zero atoms are used.

The correction of a scalar source-only atom is
\begin{equation}\label{eq:src-onl-crr}
	c_\alpha^{0*}(\lambda)=
	\begin{cases}
		\displaystyle
		\frac{\lambda}{m_h}(I-P_k^{0*})B_{\mathbf{b}}^*(1),&a=1,\\[6pt]
		\displaystyle
		\lambda\epsilon_\xi^{0*}(1)
		+\frac{\lambda}{m_h}(I-P_k^{0*})B_{\mathbf{b}}^*(1),&a>1.
	\end{cases}
\end{equation}
The correction of a scalar target atom is
\begin{equation}\label{eq:tgt-sca-crr}
	c_\alpha^{1*}(\lambda)=
	\begin{cases}
		\displaystyle
		\frac{\lambda}{m_h}(I-P_k^{1*})B_{\mathbf{b}}^*(1),&a=1,\\[6pt]
		\displaystyle
		\lambda\epsilon_\xi^{1*}(1)
		+\frac{\lambda}{m_h}(I-P_k^{1*})B_{\mathbf{b}}^*(1),&a>1.
	\end{cases}
\end{equation}
Thus, with \(A^*\) equal to the identity on \(\K\), these are exactly the two forms in Lemma~\ref{lem:blk-bd}.
If the new atom has rank \(n\), the raw packet has ambient rank support below \(n\),
hence its \(d^*\)-support is contained in \(\{1,\ldots,n-1\}\), and triangularity gives
\begin{equation}\label{eq:tal-win-cns}
	(I-P_k^{i*})B_{\mathbf{b}}^*
	=P_{(k,n)}^{i*}B_{\mathbf{b}}^*.
\end{equation}
Thus the \(P_I^{i*}\) in \eqref{eq:tgt-t0}--\eqref{eq:src-lft-t1} is precisely the cell \(I=(k,n)\) of the formal datum.
Thus the two notations do not impose different corrections.

Every target packet has a canonical protected mirror.
More precisely, choose \(r(\gamma)\) so large that \(E_\gamma\) contains the source fibres of the predecessor and of all atoms occurring in \(\mathbf{b}\),
as well as the finitely many vectors of \(V_{\Q}\) recorded in \(\mathsf{sem}\).
If \(R_l\) in a scalar target packet is multiplication by \(r_l\), set
\[ \widehat{R}_l\varphi=r_lR_{\eta_l\gamma}\varphi, \qquad \widehat{\mathbf{b}}=((a_l,\widehat{\eta}_l,\widehat{R}_l))_{l=1}^s, \qquad \widehat{B}_{\mathbf{b}}^*\varphi =\sum_la_l\epsilon_{\widehat{\eta}_l}^{0*} (\widehat{R}_l\varphi). \]
Then the correction on the paired atom is
\begin{equation}\label{eq:fml-prt-crr}
	c_{\widehat{\gamma}}^{0*}\varphi=
	\begin{cases}
		\displaystyle
		\frac{1}{m_h}(I-P_k^{0*})\widehat{B}_{\mathbf{b}}^*\varphi,&a=1,\\[6pt]
		\displaystyle
		\epsilon_{\widehat{\xi}}^{0*}(R_{\xi\gamma}\varphi)
		+\frac{1}{m_h}(I-P_k^{0*})\widehat{B}_{\mathbf{b}}^*\varphi,&a>1.
	\end{cases}
\end{equation}
The protected core record is now clear.
It is a side-zero record whose packet field is \(\widehat{\mathbf{b}}\) and whose predecessor field is \(\operatorname{id}(\widehat{\xi})\) when \(a>1\) (and is empty when \(a=1\)); in the predecessor recursion we write \(R_{\widehat{\xi},\widehat{\gamma}}:=R_{\xi\gamma}\).
Thus
\begin{equation}\label{eq:prt-two-his}
	\mathsf{shist}(\widehat{\gamma})=
	\begin{cases}
		((\operatorname{id}\widehat{\gamma},
		\mathfrak{r}_0(\widehat{\gamma}))),&a=1,\\
		\mathsf{shist}(\widehat{\xi})^\frown
		((\operatorname{id}\widehat{\gamma},
		\mathfrak{r}_0(\widehat{\gamma}))),&a>1,
	\end{cases}
	\qquad
	\mathsf{thist}(\widehat{\gamma})=\mathsf{shist}(\gamma).
\end{equation}
The source atom receives the same type, level, age, cutoffs, semantic record, and copy index as its target partner.
Formula \eqref{eq:fml-prt-crr} and every source evaluation analysis follow \(\mathsf{shist}\);
coding and protected collision comparison follow \(\mathsf{thist}\).
By construction, we have
\[ \sum_l|a_l|\norm{\widehat{R}_l} \le\sum_l|a_l||r_l|\le1. \]
Consequently, \eqref{eq:fml-prt-crr} is again an admissible block BD correction.
Its packet consists of the paired copies of the target packet.

Figure~\ref{fig:protected-mirror} shows how a non-probe target record is copied to its protected mirror.
The upper row is unchanged; predecessor and packet labels are replaced by their source partners, with the restriction maps shown in the figure.
The predecessor row is omitted at age one.
The last row copies the target structural history into \(\mathsf{thist}(\widehat{\gamma})\); the mirror's own structural history still follows the source predecessors.

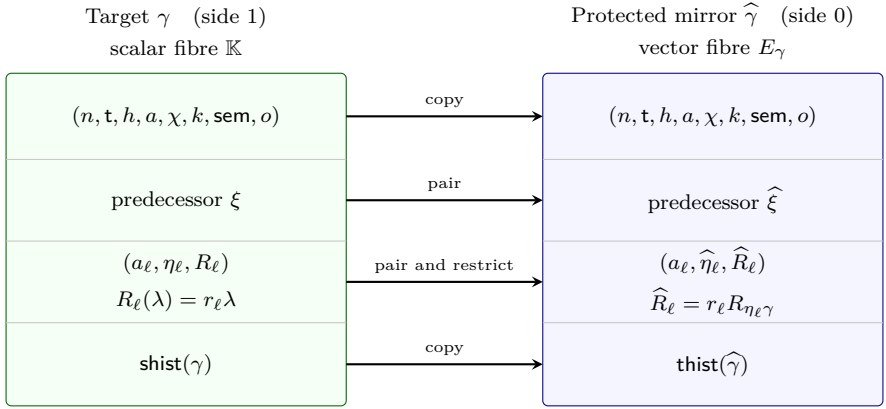
\begin{figure}[htbp]
	\centering
\begin{tikzpicture}[x=1cm,y=1cm,>=stealth,
	every node/.style={font=\small,align=center,inner sep=3pt}]
	\filldraw[draw=green!40!black,fill=green!4,rounded corners=2pt]
		(0,-0.28) rectangle (4.5,4.12);
	\filldraw[draw=blue!55!black,fill=blue!4,rounded corners=2pt]
		(7.1,-0.28) rectangle (11.6,4.12);
	\node at (2.25,4.68) {Target \(\gamma\)\quad (side \(1\))\\[2pt]scalar fibre \(\K\)};
	\node at (9.35,4.68) {Protected mirror \(\widehat\gamma\)\quad (side \(0\))\\[2pt]vector fibre \(E_\gamma\)};
	\foreach \y in {2.98,1.90,0.82} {
		\draw[black!22] (0,\y) -- (4.5,\y);
		\draw[black!22] (7.1,\y) -- (11.6,\y);
	}
	\node at (2.25,3.55) {\((n,\mathsf t,h,a,\chi,k,\mathsf{sem},o)\)};
	\node at (9.35,3.55) {\((n,\mathsf t,h,a,\chi,k,\mathsf{sem},o)\)};
	\node at (2.25,2.44) {predecessor \(\xi\)};
	\node at (9.35,2.44) {predecessor \(\widehat\xi\)};
	\node at (2.25,1.36) {\((a_\ell,\eta_\ell,R_\ell)\)\\[4pt]\(R_\ell(\lambda)=r_\ell\lambda\)};
	\node at (9.35,1.36) {\((a_\ell,\widehat\eta_\ell,\widehat R_\ell)\)\\[4pt]\(\widehat R_\ell=r_\ell R_{\eta_\ell\gamma}\)};
	\node at (2.25,0.27) {\(\mathsf{shist}(\gamma)\)};
	\node at (9.35,0.27) {\(\mathsf{thist}(\widehat\gamma)\)};
	\draw[->,thick] (4.5,3.55) -- node[above,font=\scriptsize] {copy} (7.1,3.55);
	\draw[->,thick] (4.5,2.44) -- node[above,font=\scriptsize] {pair} (7.1,2.44);
	\draw[->,thick] (4.5,1.36) -- node[above,font=\scriptsize] {pair and restrict} (7.1,1.36);
	\draw[->,thick] (4.5,0.27) -- node[above,font=\scriptsize] {copy} (7.1,0.27);
\end{tikzpicture}
	\caption{Information copied to a protected mirror.}
	\label{fig:protected-mirror}
\end{figure}

For every atom and every \(\varphi\in(G_\alpha^i)^*\), define
\begin{equation}\label{eq:fml-dst}
	d_{\alpha,\varphi}^{i*}
	=\epsilon_\alpha^{i*}(\varphi)-c_\alpha^{i*}(\varphi).
\end{equation}
For each rank \(r\), write elements of \((F_r^i)^*\) as \((\varphi_\alpha)_{\alpha\in\Delta_r^i}\) using the fibre decomposition \eqref{eq:fml-fn}.
The corresponding block map is
\[ d_r^{i*}((\varphi_\alpha)_{\alpha\in\Delta_r^i})=\sum_{\alpha\in\Delta_r^i}d_{\alpha,\varphi_\alpha}^{i*}. \]
Let \(\iota_r^i\) be the canonical ambient injection of \((F_r^i)^*\), and set
\[ c_r^{i*}((\varphi_\alpha)_{\alpha\in\Delta_r^i}) =\sum_{\alpha\in\Delta_r^i}c_\alpha^{i*}(\varphi_\alpha). \]
Summing \eqref{eq:fml-dst} over \(\alpha\in\Delta_r^i\) gives the block BD form of Lemma~\ref{lem:blk-bd}:
\[ d_r^{i*}=\iota_r^i-c_r^{i*}, \qquad c_r^{i*}:(F_r^i)^*\longrightarrow\bigoplus_{s<r}(F_s^i)^*. \]
Once ranks up to \(n\) have been constructed, \(P_k^{i*}\), \(k\le n\),
denotes the algebraic projection onto \(\bigoplus_{r\le k}d_r^{i*}((F_r^i)^*)\) along the later \(d^{i*}\)-blocks.
Thus every correction at rank \(n+1\) uses only objects defined by rank \(n\).

The following finite adjoint construction defines the primal extensions.
For \(N\ge1\), set
\[ \mathcal{E}_{[1,N]}^i= \left(\bigoplus_{r=1}^NF_r^i\right)_{\ell_\infty}, \qquad {\mathcal{E}_{[1,N],*}^{i}}= \left(\bigoplus_{r=1}^N(F_r^i)^*\right)_{\ell_1}. \]
Unit triangularity implies
\[ \bigoplus_{r=1}^n d_r^{i*}((F_r^i)^*) ={\mathcal{E}_{[1,n],*}^{i}} \]
as vector spaces.
Consequently, for \(n\le N\), let \(P_n^{i*}\) be computed in the triangular system through rank \(N\); then
\begin{equation}\label{eq:fin-dua-ret}
	Q_{n,N}^{i*}:{\mathcal{E}_{[1,N],*}^{i}}\longrightarrow
	{\mathcal{E}_{[1,n],*}^{i}},
	\qquad Q_{n,N}^{i*}f=P_n^{i*}f,
\end{equation}
is well-defined.
Actually, if \(L\ge N\ge n\), elimination of the top ambient coordinates shows that the \(d^{i*}\)-expansion of every
\(f\in\mathcal{E}_{[1,N],*}^{i}\) uses no block above \(N\).
Hence
\begin{equation}\label{eq:fin-dua-cns}
	Q_{n,L}^{i*}\big|_{\mathcal{E}_{[1,N],*}^{i}}
	=Q_{n,N}^{i*}\qquad(L\ge N\ge n).
\end{equation}
In particular, \(Q_{n,N}^{i*}\) is independent of all later ranks and \(\norm{Q_{n,N}^{i*}}\le M\).

Define the finite extension by the adjoint relation
\begin{equation}\label{eq:fin-prm-adj}
	i_{n,N}^i=(Q_{n,N}^{i*})^*:
	\mathcal{E}_{[1,n]}^i\longrightarrow\mathcal{E}_{[1,N]}^i;
	\qquad
	\langle f,i_{n,N}^ix\rangle
	=\langle P_n^{i*}f,x\rangle .
\end{equation}
Thus \(i_{n,n}^i\) is the identity and \(\norm{i_{n,N}^i}\le M\).
If \(L\ge N\ge n\), then \eqref{eq:fin-dua-cns}, tested against every \(f\in{\mathcal{E}_{[1,N],*}^{i}}\), gives
\begin{equation}\label{eq:fin-prm-cns}
	(i_{n,L}^ix)|_{[1,N]}=i_{n,N}^ix.
\end{equation}
The compatible family in \eqref{eq:fin-prm-cns} therefore defines a coordinatewise vector \(i_n^ix\).
Moreover,
\[ \sup_{N\ge n}\norm{(i_n^ix)|_{[1,N]}} =\sup_{N\ge n}\norm{i_{n,N}^ix} \le M\norm{x}, \]
so \(i_n^ix\) belongs to the ambient \(\ell_\infty\)-sum.
This adjoint definition gives the coordinate recursion used later:
\begin{align}
	(i_n^ix)|_{[1,n]}&=x,\label{eq:ext-hed}\\
	\varphi\bigl((i_n^ix)_\alpha\bigr)
	&=c_\alpha^{i*}(\varphi)
	\bigl((i_n^ix)|_{[1,\operatorname{rank}\alpha)}\bigr)
	\quad(\operatorname{rank}\alpha>n).
	\label{eq:ext-rec}
\end{align}
Indeed, if \(r=\operatorname{rank}\alpha>n\), then \(P_n^{i*}d_{\alpha,\varphi}^{i*}=0\), and hence
\[ P_n^{i*}\epsilon_\alpha^{i*}(\varphi) =P_n^{i*}c_\alpha^{i*}(\varphi). \]
Since \(c_\alpha^{i*}(\varphi)\in\mathcal{E}_{[1,r-1],*}^i\) and
\((i_n^ix)|_{[1,r]}=i_{n,r}^ix\), applying \eqref{eq:fin-prm-adj} with \(N=r\) gives
\begin{align*}
	\varphi\bigl((i_n^ix)_\alpha\bigr)
	&=\bigl\langle\epsilon_\alpha^{i*}(\varphi),i_{n,r}^ix\bigr\rangle
	=\bigl\langle P_n^{i*}\epsilon_\alpha^{i*}(\varphi),x\bigr\rangle\\
	&=\bigl\langle P_n^{i*}c_\alpha^{i*}(\varphi),x\bigr\rangle
	=\bigl\langle c_\alpha^{i*}(\varphi),i_{n,r}^ix\bigr\rangle
	=c_\alpha^{i*}(\varphi)
	\bigl((i_n^ix)|_{[1,r)}\bigr),
\end{align*}
which is \eqref{eq:ext-rec}.
Conversely, finite-dimensional duality separates the points of every fibre \(G_\alpha^i\);
therefore by induction on \(r\), \eqref{eq:ext-hed}--\eqref{eq:ext-rec} determine every coordinate uniquely.
This also proves directly that the recursive and adjoint definitions agree.

The extensions are nested in the precise sense
\begin{equation}\label{eq:nst-prm-ext}
	i_m^i\bigl((i_n^ix)|_{[1,m]}\bigr)=i_n^ix
	\qquad(n\le m).
\end{equation}
Both sides have the same first \(m\) blocks, and beyond \(m\) they satisfy the same uniquely solvable recursion \eqref{eq:ext-rec}.
Put
\begin{equation}\label{eq:fml-xi-di}
	X_i^{\rm alg}=\bigcup_n i_n^i
	\left(\bigoplus_{r\le n}F_r^i\right),
	\qquad X_i=\overline{X_i^{\rm alg}},
	\qquad
	d_\alpha^i(g)=i_{\operatorname{rank}\alpha}^i
	(0,\ldots,0,g,0,\ldots).
\end{equation}
\begin{lemma}[Atomic block biorthogonality]\label{lem:fml-blk-bio}
	Let \(i\in\{0,1\}\), let \(\alpha\) and \(\beta\) be atoms on side \(i\), and let
	\(g\in G_\alpha^i\) and \(\psi\in(G_\beta^i)^*\). Then
	\begin{equation}\label{eq:fml-blk-bio}
		d_{\beta,\psi}^{i*}\bigl(d_\alpha^i(g)\bigr)
		=\begin{cases}
			\psi(g),&\beta=\alpha,\\
			0,&\beta\ne\alpha.
		\end{cases}
	\end{equation}
\end{lemma}

\begin{proof}
	Put \(s=\operatorname{rank}\alpha\) and \(t=\operatorname{rank}\beta\), and let
	\(u\in\mathcal{E}_{[1,s]}^i\) be supported on the \(\alpha\)-coordinate with value \(g\).
	Thus \(d_\alpha^i(g)=i_s^iu\).
	At any finite stage containing both atoms, by \eqref{eq:fin-prm-adj}, we have
	\[ d_{\beta,\psi}^{i*}\bigl(d_\alpha^i(g)\bigr) =\bigl\langle P_s^{i*}d_{\beta,\psi}^{i*},u\bigr\rangle. \]
	If \(t>s\), then \(P_s^{i*}d_{\beta,\psi}^{i*}=0\).
	If \(t<s\), then \(P_s^{i*}d_{\beta,\psi}^{i*}=d_{\beta,\psi}^{i*}\), whose ambient support lies before the support of \(u\).
	Finally, if \(t=s\), then \(c_\beta^{i*}(\psi)\) is supported below rank \(s\), and hence
	\[
		\bigl\langle d_{\beta,\psi}^{i*},u\bigr\rangle
		=\bigl\langle\epsilon_\beta^{i*}(\psi),u\bigr\rangle
		=\begin{cases}
			\psi(g),&\beta=\alpha,\\
			0,&\beta\ne\alpha.
		\end{cases}
	\]
	This therefore proves \eqref{eq:fml-blk-bio}.
\end{proof}

Lemma~\ref{lem:fml-blk-bio} yields the FDD coefficient maps below.
For every atom \(\alpha\) define its FDD coefficient map
\begin{equation}\label{eq:fml-coe-map}
	D_\alpha^i:X_i^{\rm alg}\longrightarrow G_\alpha^i,
	\qquad
	\psi(D_\alpha^ix)=d_{\alpha,\psi}^{i*}(x)
	\quad(\psi\in(G_\alpha^i)^*).
\end{equation}
For a paired target label \(\gamma\), the operator denoted earlier by \(D_\gamma:X_0^{\rm alg}\to E_\gamma\) is precisely \(D_{\widehat{\gamma}}^0\).
For \(x\in X_i^{\rm alg}\), define
\begin{equation}\label{eq:fml-fdd-sup}
	\supp_{\rm FDD}x
	=\{r:\text{there is }\alpha\in\Delta_r^i
	\text{ with }D_\alpha^ix\ne0\},
\end{equation}
and let \(\ran x\) be the smallest integer interval containing this set.
For a finite-dimensional algebraic subspace \(L\subseteq X_i^{\rm alg}\),
put \(\supp_{\rm FDD}L=\bigcup_{x\in L}\supp_{\rm FDD}x\).
Now write
\begin{equation}\label{eq:fml-ki}
	\mathcal{K}_i=
	\{e_{\alpha,\psi}^{i*}
	:=\epsilon_\alpha^{i*}(\psi)|_{X_i}:
	\alpha\in\bigcup_n\Delta_n^i,\
	\psi\in B_{(G_\alpha^i)^*}\}.
\end{equation}
These coordinates are contractive and norm \(X_i\).
By Lemma~\ref{lem:blk-bd}, we have
\begin{equation}\label{eq:fml-ext-bnd}
	\norm{i_n^i}\le M,\qquad
	\norm{P_{[1,n]}^i}\le M,\qquad
	\norm{P_I^i}\le2M.
\end{equation}
Equations~\eqref{eq:fin-prm-adj} and \eqref{eq:fin-prm-cns} identify every finite truncation with the adjoint supplied by Lemma~\ref{lem:blk-bd},
so \eqref{eq:ext-rec} agrees with the usual BD extension.

\subsubsection{Stems and the five requirements}
\label{sec:req-cat}

Split each pool \(\mathcal{C}_c\) into two disjoint infinite recursive sets
\[ \mathcal{C}_c^{\rm seed} \cup \mathcal{C}_c^{\rm code}, \qquad c\in\{00,11,01\}. \]
Let \(\mathscr{A}\) be the fixed countable syntactic alphabet consisting of natural-number identifiers,
side and type symbols, integer level and cutoff fields, rational packet matrices, rational finite vectors and subspaces,
and rational certificate fields.
Put \(\mathscr{W}=\mathscr{A}^{<\omega}\).
A \emph{formal stem} is a well-typed word in \(\mathscr{W}\).
The compatibility conditions imposed during the rank recursion select the \emph{realized stems} from this fixed preconstruction domain.

Before the BD recursion, let us fix recursive injections
\begin{equation}\label{eq:fml-sig}
	\sigma_c:
	\{(q,\pi):q\in\mathcal{C}_c^{\rm seed},\
	\pi\in\mathscr{W}\text{ a formal stem}\}
	\longrightarrow\mathcal{C}_c^{\rm code}
\end{equation}
with pairwise disjoint ranges such that if \(p=\sigma_c(q,\pi)\), then we have
\begin{equation}\label{eq:qnt-cod-grw}
	\begin{aligned}
		p&>\max\{q,R(\pi)\},\\
		m_p&\ge2^{q+|\pi|+10}L_qn_qm_q^4.
	\end{aligned}
\end{equation}
Here \(R(\pi)\) is the maximum of \(0\) and all rank, level, and cutoff fields recorded in \(\pi\), including the seed birth rank.
When the corner is fixed, we write simply \(\sigma(q,\pi)\) for \(\sigma_c(q,\pi)\).
Enumerate each countable domain and assign to every input the least unused member of the appropriate coded pool which satisfies \eqref{eq:qnt-cod-grw}.
This defines all the maps recursively before the BD construction starts.

\begin{definition}[Seeds, realized stems, and certificates]
	\label{def:crt-stm}
	A seed is an administrative record
	\[ \vartheta=(c,q,\chi,o), \qquad c\in\{00,11,01\},\quad q\in\mathcal{C}_c^{\rm seed}. \]
	And it is born at a finite rank \(\chi\).
	A \emph{realized outer stem} over \(\vartheta\) is
	\[ \pi=(\vartheta; (\operatorname{id}\zeta_1,\mathfrak{r}_0(\zeta_1),\mathfrak{c}_1), \ldots, (\operatorname{id}\zeta_a,\mathfrak{r}_0(\zeta_a),\mathfrak{c}_a)), \qquad a\le n_q, \]
	where the \(\zeta_s\)'s form an outer predecessor chain of level \(q\).
	Recursively, if \(\pi_{s-1}\) is the word obtained by stopping before the \(s\)-th displayed triple,
	then \(\mathfrak{c}_s\) is compatible with \(\pi_{s-1}\),
	the correction packet of \(\zeta_s\) is \(\langle\eta_s\rangle\), where \(\eta_s\) is the terminal inner atom named by \(\mathfrak{c}_s\),
	and the predecessor of \(\zeta_s\) is \(\zeta_{s-1}\) when \(s>1\).
	Its endpoint is
	\[
		\rho(\pi)=
		\begin{cases}
			\chi,&a=0,\\
			\operatorname{rank}\zeta_a,&a>0.
		\end{cases}
	\]
	In particular, changing a copy index will change the realized stem.

	Let \(i\to j\) be the corner specified by \(c\).
	Suppose that \(\eta\) is the last atom of an inner chain attached to \(\pi\),
	and that this chain has level \(p=\sigma_c(q,\pi)\) and age \(n_p\).
	A certificate compatible with \((\pi,\eta)\) is a finite rational record
	\[ \mathfrak{c}=(i,j,y,z,e_\eta^{j*},I, \mathscr{M}_{ij}(y),\varepsilon) \]
	such that
	\begin{enumerate}[label={\textup{(\arabic*)}}]
		\item the union of the FDD supports of \(y\), \(z\), and \(\mathscr{M}_{ij}(y)\) is contained in \((\rho(\pi),\operatorname{rank}\eta]\),
		and the finite rank interval \(I\) satisfies
		\[ \supp_{\rm FDD}y\cup\supp_{\rm FDD}z \cup\supp_{\rm FDD}\mathscr{M}_{ij}(y)\subseteq I, \qquad \operatorname{rank}\eta\in I; \]
		\item \(e_\eta^{j*}|_{\mathscr{M}_{ij}(y)}=0\) and \(\operatorname{Re}e_\eta^{j*}(z)>\varepsilon\),
		where \(\varepsilon\in{\Q_+}\);
		\item in the forward case the finite matrix defining \(v\mapsto T_vy\) and the vector \(U_\eta y\) are rational in a common  fibre \(V_r\).
	\end{enumerate}
	This is an instance of \eqref{eq:sep-crt} with \(h=e_\eta^{j*}\) and \(L=\mathscr{M}_{ij}(y)\).
	The extra fields record the support and forward rationality required by the stage construction.
\end{definition}

\begin{figure}[htbp]
	\centering
\begin{tikzpicture}[x=1cm,y=0.8cm,>=stealth,
	every node/.style={font=\small},
	record/.style={draw=blue!55!black,rounded corners=2pt,
		fill=blue!4,align=center,minimum height=1.08cm,inner sep=4pt},
	stemstage/.style={draw=black!65,rounded corners=2pt,
		fill=black!2,align=center,minimum height=1.12cm,inner sep=4pt},
	cert/.style={draw=green!40!black,rounded corners=2pt,
		fill=green!4,align=center,minimum height=1.12cm,inner sep=4pt}]
	\draw[blue!55!black,rounded corners=3pt] (0,1.30) rectangle (12.05,3.30);
	\node[fill=white,inner sep=3pt] at (6.025,3.30)
		{current stem \(\pi_a=(\vartheta;t_1,\ldots,t_a)\)};
	\node[record,minimum width=2.15cm] (seed) at (1.30,2.25)
		{seed record\\\(\vartheta=(c,q,\chi,o)\)};
	\node at (2.80,2.25) {\(\,;\,\)};
	\node[record,minimum width=3.30cm] (first) at (4.70,2.25)
		{\(t_1\)\\\((\operatorname{id}\zeta_1,\mathfrak r_0(\zeta_1),\mathfrak c_1)\)};
	\node at (7.35,2.25) {\(,\ \ldots\ ,\)};
	\node[record,minimum width=3.30cm] (last) at (10.15,2.25)
		{\(t_a\)\\\((\operatorname{id}\zeta_a,\mathfrak r_0(\zeta_a),\mathfrak c_a)\)};

	\node[stemstage,minimum width=3.05cm] (code) at (2.05,-0.05)
		{next inner level\\\(p=\sigma_c(q,\pi_a)\)\\cutoff \(\rho(\pi_a)\)};
	\draw[->] (2.05,1.30) -- node[right,align=left,font=\scriptsize]
		{whole stored word} (code.north);
	\node[stemstage,minimum width=3.30cm] (inner) at (6.15,-0.05)
		{chosen complete inner chain\\\(\nu_1\to\cdots\to\nu_{n_p}=\eta\)\\level \(p\), age \(n_p\)};
	\draw[->] (code.east) -- node[above,font=\scriptsize]{\textup{(R3)}} (inner.west);
	\node[cert,minimum width=3.20cm] (certificate) at (10.20,-0.05)
		{chosen certificate \(\mathfrak c_{a+1}\)\\compatible with \((\pi_a,\eta)\)};
	\draw[->] (inner.east) -- (certificate.west);

	\node[stemstage,minimum width=3.20cm] (outer) at (10.20,-2.12)
		{new outer atom \(\zeta_{a+1}\)\\level \(q\), age \(a+1\)\\packet \(\langle\eta\rangle\)};
	\draw[->] (certificate.south) -- node[right,font=\scriptsize]
		{\textup{(R4)}} (outer.north);
	\draw[->,dashed,blue!55!black]
		(last.east) -- (12.40,2.25) -- (12.40,-2.12) -- (outer.east);
	\node[rotate=90,font=\scriptsize,fill=white,inner sep=2pt]
		at (12.40,0.78) {outer predecessor \(\zeta_a\)};
	\node[record,minimum width=7.20cm] (extended) at (4.05,-2.12)
		{\(\pi_{a+1}=\pi_a{}^\frown t_{a+1}\)\\
		 \(t_{a+1}=(\operatorname{id}\zeta_{a+1},\mathfrak r_0(\zeta_{a+1}),\mathfrak c_{a+1})\)};
	\draw[->] (outer.west) -- (extended.east);
	\node[font=\scriptsize] at (6.10,-3.20)
		{\(\rho(\pi_a)<\operatorname{rank}\eta<\operatorname{rank}\zeta_{a+1}\)};
\end{tikzpicture}
	\caption{A realized outer stem and one legal extension.}
	\label{fig:realized-outer-stem}
\end{figure}
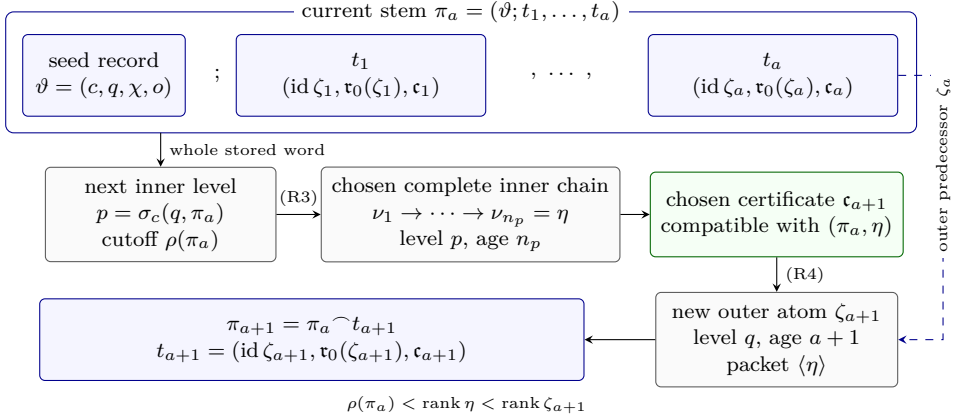

Figure~\ref{fig:realized-outer-stem} shows the stored records and their role in a one-step extension.
The coding history records the support and interval data of its certificates, the initial and local cutoffs,
the corner type, the copy indices, and the full ordered stem of core snapshots.
For every realized stem \(\pi\) over a seed of level \(q\), the pair \((q,\pi)\) is an input of \(\sigma_c\).
And the semantic \(V\)-record is part of the tagged word \(\pi\).

The admissible one-step extension requirements are as follows.
Each requirement contains a numerical lower bound \(N\) for the rank at which it is met and a copy index \(o\).
Its packet and all of its dependencies must already exist when the stage construction acts on it.

\begin{enumerate}[label=\textup{(R\arabic*)},leftmargin=*]
	\item \emph{Generic requirement.}
	\begin{enumerate}[label=\textup{(\alph*)},leftmargin=*]
		\item \emph{Input.} Choose \(i\in\{0,1\}\), \(h\in\mathcal{G}\), an initial cutoff \(\chi\),
		and either an empty chain or a generic chain of level \(h\) and age \(a<n_h\).
		A nonempty chain is target when \(i=1\) and scalar source-only when \(i=0\);
		a protected mirror is continued only with its target partner.
		\item \emph{Cutoff and packet.} In the empty case put \(k=\chi\).
		Otherwise, let \(\xi\) be the last atom, require \(\chi=\chi(\xi)\), and put \(k=\operatorname{rank}\xi\).
		Choose a represented scalar packet \(\mathbf{b}\) on side \(i\) with \(\supp_{\rm raw}\mathbf{b}\subseteq(k,\infty)\).
		\item \emph{Realization.} Add a generic atom of age one in the empty case and age \(a+1\) otherwise.
		A side-one atom has its protected source mirror; a side-zero atom is source-only.
	\end{enumerate}

	\item \emph{Seed requirement.}
	\begin{enumerate}[label=\textup{(\alph*)},leftmargin=*]
		\item \emph{Input.} Choose \(c\) and \(q\in\mathcal{C}_c^{\rm seed}\).
		\item \emph{Realization.} At rank \(n>N\), register the fresh seed \(\vartheta=(c,q,n,o)\).
		No corrected atom is added; the rank-\(n\) probe pair remains.
	\end{enumerate}

	\item \emph{Inner requirement.}
	\begin{enumerate}[label=\textup{(\alph*)},leftmargin=*]
		\item \emph{Input.} Choose a seed \(\vartheta\) and a realized outer stem \(\pi\) over \(\vartheta\), and put
		\(p=\sigma_c(q,\pi)\).
		Choose either the empty inner chain attached to \(\pi\), or an attached inner chain of level \(p\) and age \(a<n_p\).
		\item \emph{Cutoff and packet.} Put \(k=\rho(\pi)\) in the empty case.
		Otherwise, let \(\xi\) be the last inner atom and put \(k=\operatorname{rank}\xi\).
		Choose a represented scalar packet \(\mathbf{b}\) on the output side \(j\) with
		\(\supp_{\rm raw}\mathbf{b}\subseteq(k,\infty)\).
		\item \emph{Realization.} Add an atom of type \(\mathsf{inner}\), level \(p\),
		attachment \((\vartheta,\pi)\), initial cutoff \(\rho(\pi)\), and age one or \(a+1\), respectively.
		If \(j=1\), add its protected source mirror; if \(j=0\), the atom is scalar and source-only.
	\end{enumerate}

	\item \emph{Outer requirement.}
	\begin{enumerate}[label=\textup{(\alph*)},leftmargin=*]
		\item \emph{Input.} Choose a seed \(\vartheta=(c,q,\chi,o)\), a realized stem \(\pi\) over \(\vartheta\) of length \(a<n_q\),
		a complete attached inner chain as in Definition~\ref{def:crt-stm}, and a compatible certificate \(\mathfrak{c}\).
		\item \emph{Cutoff and packet.} Let \(\eta\) be the last inner atom and use the one-entry packet \(\langle\eta\rangle\).
		The local cutoff is \(\rho(\pi)\), and the predecessor is \(\zeta_a\) when \(a>0\).
		\item \emph{Realization.} Add an outer atom of level \(q\), age \(a+1\), and initial cutoff \(\chi\),
		with the tagged \(\mathsf{R}4\)-record from \eqref{eq:sem-rec-tab}.
		It extends the stem to
		\[ \pi^\frown (\operatorname{id}\zeta_{a+1},\mathfrak{r}_0(\zeta_{a+1}),\mathfrak{c}). \]
		If \(j=1\), add its protected source mirror; if \(j=0\), the atom is scalar and source-only.
	\end{enumerate}

	\item \emph{Reverse requirement.}
	\begin{enumerate}[label=\textup{(\alph*)},leftmargin=*]
		\item \emph{Input.} Choose \(r\in\mathcal{R}\), an initial cutoff \(\chi\),
		and either an empty source-only reverse chain or one of age \(a<n_r\).
		\item \emph{Cutoff and packet.} In the empty case put \(k=\chi\).
		Otherwise, let \(\xi\) be the last atom, require \(\chi=\chi(\xi)\), and put \(k=\operatorname{rank}\xi\).
		Choose a represented scalar source packet \(\mathbf{b}\) with \(\supp_{\rm raw}\mathbf{b}\subseteq(k,\infty)\).
		\item \emph{Realization.} Add a source-only reverse atom of level \(r\) and age one or \(a+1\), respectively.
		No target atom has level \(r\).
	\end{enumerate}
\end{enumerate}

For a non-seed requirement \(\mathsf{Q}\), define its required level by
\[
	\operatorname{lev}(\mathsf{Q})=
	\begin{cases}
		h,&\mathsf{Q}\text{ is of type \textup{(R1)}},\\
		p=\sigma_c(q,\pi),&\mathsf{Q}\text{ is of type \textup{(R3)}},\\
		q,&\mathsf{Q}\text{ is of type \textup{(R4)}},\\
		r,&\mathsf{Q}\text{ is of type \textup{(R5)}}.
	\end{cases}
\]
The seed requirement \textup{(R2)} has no BD weight and hence no required level.
It is admissible at rank \(n\) precisely when its named old data have rank below \(n\) and \(n>N\).
A non-seed requirement is \emph{admissible at rank \(n\)} if all named atoms and seed records have rank smaller than \(n\),
its local cutoff satisfies \(0\le k<n\), its raw support is contained in \((k,n)\), the stated age bound holds, \(n>N\),
and \(n\ge\operatorname{lev}(\mathsf{Q})\).

In every continuation requirement the predecessor has the same root key as the new atom and has the sort permitted by the requirement.
Equivalently, its type, level, initial cutoff, and, when present, attachment agree with the prescribed fields,
and the additional target/source-only alternative in \textup{(R1)} is respected.
For a target requirement producing \(\gamma\), let the new source fibre \(E_\gamma=V_s\) be the least \(V_s\) containing the protected-source fibres corresponding to the predecessor and packet entries and all rational \(V\)-data in the semantic record.
Thus the protected maps in \eqref{eq:fml-prt-crr} are determined, rather than chosen afterwards.

For an outer requirement, the \emph{complete strong packet} is the one-entry packet \(\langle\eta\rangle\),
where \(\eta\) is the last atom of the attached inner chain of age \(n_p\).
The associated outer edge spans the rank interval
\[ \bigl(\rho(\pi),\operatorname{rank}\zeta_{a+1}\bigr). \]
The ranks of all atoms in the attached inner chain, the FDD support of the certified vector,
and the raw support of the strong packet are contained in this interval.

\begin{remark}[Order of the seed, inner, and outer requirements]
	\label{rem:crt-req-pat}
	The requirements \textup{(R2)}--\textup{(R4)} have the following dependency order for a fixed seed.
	They need not be realized at consecutive ranks.
	Let \textup{(R2)} produce a seed \(\vartheta\) of level \(q\), and let \(\pi_s\) be the outer stem after \(s\) uses of \textup{(R4)}, where \(\pi_0=(\vartheta;)\).
	Put
	\[ p_s=\sigma_c(q,\pi_s)\qquad(0\le s<n_q). \]
	For each \(0\le s<n_q\), \(n_{p_s}\) uses of \textup{(R3)} build the complete level-\(p_s\) inner chain attached to \(\pi_s\).
	Once its terminal atom and a compatible certificate are available, one use of \textup{(R4)} extends \(\pi_s\) to \(\pi_{s+1}\).
	Thus
	\[
		\begin{aligned}
			\textup{(R2)}
			&\longrightarrow
			\underbrace{\textup{(R3)}\longrightarrow\cdots\longrightarrow\textup{(R3)}}_{n_{p_0}\text{ times}}
			\longrightarrow\textup{(R4)}\\
			&\longrightarrow
			\underbrace{\textup{(R3)}\longrightarrow\cdots\longrightarrow\textup{(R3)}}_{n_{p_1}\text{ times}}
			\longrightarrow\textup{(R4)}
			\longrightarrow\cdots\\
			&\longrightarrow
			\underbrace{\textup{(R3)}\longrightarrow\cdots\longrightarrow\textup{(R3)}}_{n_{p_{n_q-1}}\text{ times}}
			\longrightarrow\textup{(R4)}.
		\end{aligned}
	\]
	Hence \textup{(R2)} is used once, and for each \(0\le s<n_q\), a block of \(n_{p_s}\) uses of \textup{(R3)} is followed by one use of \textup{(R4)}.
	Other requirements may be realized between these dependent steps.
\end{remark}

Figure~\ref{fig:finite-certificate} displays the certificate from
Definition~\ref{def:crt-stm} and its use in \textup{(R4)}.
Here \(\pi\) has outer age \(a<n_q\), and \(\eta\) is the terminal atom of
its complete attached inner chain of level \(p=\sigma_c(q,\pi)\) and age \(n_p\).
In the figure, the coordinate \(f\) annihilates the entire orbit \(L\) and strictly separates \(z\) from \(L\).
The operator \(S\) is not part of the finite rational record.

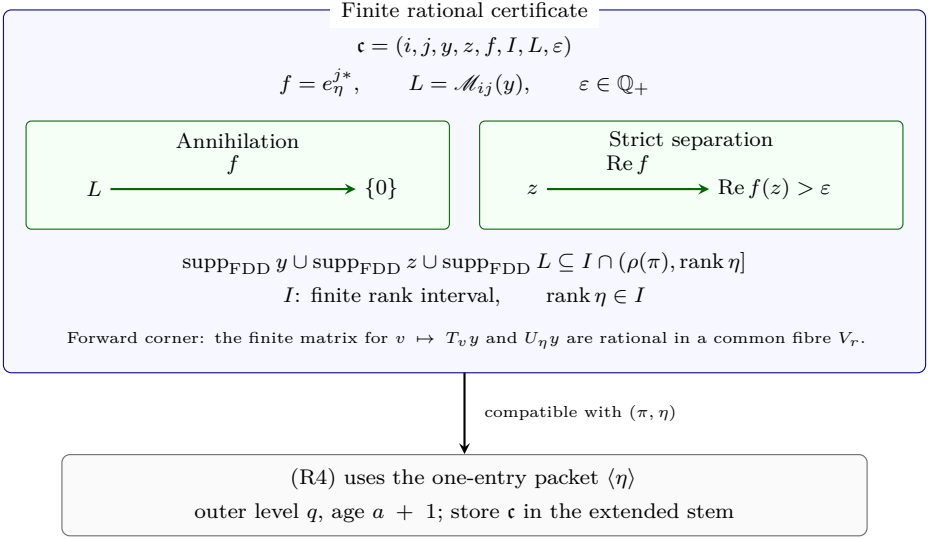
\begin{figure}[!htbp]
	\centering
\begin{tikzpicture}[x=1cm,y=1cm,>=stealth,
	every node/.style={font=\small,align=center}]
	\draw[draw=blue!55!black,fill=blue!3,rounded corners=3pt]
		(0,1.05) rectangle (12.20,5.85);
	\node[fill=white,inner sep=4pt] at (6.10,5.85)
		{Finite rational certificate};
	\node at (6.10,5.37)
		{\(\mathfrak c=(i,j,y,z,f,I,L,\varepsilon)\)};
	\node at (6.10,4.89)
		{\(f=e_\eta^{j*},\qquad L=\mathscr M_{ij}(y),\qquad
			\varepsilon\in\Q_+\)};

	\draw[draw=green!40!black,fill=green!4,rounded corners=2pt]
		(0.30,2.95) rectangle (5.90,4.38);
	\draw[draw=green!40!black,fill=green!4,rounded corners=2pt]
		(6.30,2.95) rectangle (11.90,4.38);
	\node at (3.10,4.12) {Annihilation};
	\node at (9.10,4.12) {Strict separation};
	\node (orbit) at (1.20,3.48) {\(L\)};
	\node (zero) at (5.00,3.48) {\(\{0\}\)};
	\draw[->,thick,green!40!black] (orbit.east) --
		node[above=2pt,text=black] {\(f\)} (zero.west);
	\node (output) at (7.00,3.48) {\(z\)};
	\node (value) at (10.20,3.48) {\(\operatorname{Re}f(z)>\varepsilon\)};
	\draw[->,thick,green!40!black] (output.east) --
		node[above=2pt,text=black] {\(\operatorname{Re}f\)} (value.west);

	\node at (6.10,2.50)
		{\(\supp_{\rm FDD}y\cup\supp_{\rm FDD}z\cup\supp_{\rm FDD}L
			 \subseteq I\cap(\rho(\pi),\operatorname{rank}\eta]\)};
	\node at (6.10,2.05)
		{\(I\): finite rank interval,\qquad \(\operatorname{rank}\eta\in I\)};
	\node[text width=11.35cm,font=\scriptsize] at (6.10,1.48)
		{Forward corner: the finite matrix for \(v\mapsto T_vy\) and
		 \(U_\eta y\) are rational in a common fibre \(V_r\).};

	\node[draw=black!60,fill=black!2,rounded corners=3pt,
		text width=10.30cm,minimum height=1.05cm,inner sep=5pt] (outer) at (6.10,-0.57)
		{\textup{(R4)} uses the one-entry packet \(\langle\eta\rangle\)\\[3pt]
		 outer level \(q\), age \(a+1\); store \(\mathfrak c\) in the extended stem};
	\draw[->,thick] (6.10,1.05) --
		node[right=4pt,font=\scriptsize] {compatible with \((\pi,\eta)\)} (outer.north);
\end{tikzpicture}
	\caption{A finite certificate for an outer continuation.}
	\label{fig:finite-certificate}
\end{figure}
\FloatBarrier

\begin{lemma}[Admissibility of the five requirement types]
	\label{lem:req-adm}
	Let a requirement from \textup{(R1)}--\textup{(R5)} be admissible at rank \(n\).
	Then every atom which occurs in its correction has rank strictly smaller than \(n\),
	and no correction uses an atom created at rank \(n\).
	More precisely,
	\begin{enumerate}[label={\textup{(\arabic*)}}]
		\item an \textup{(R1)} realization adds either one scalar source-only atom,
		or one scalar target atom together with one protected source atom;
		its predecessor and every entry of its packet are old;
		\item an \textup{(R2)} realization adds no non-probe atom and records the seed with birth rank \(n\);
		\item an \textup{(R3)} realization uses the already constructed last inner predecessor, if any,
		and an old packet in \((\rho(\pi),n)\); the value \(p=\sigma_c(q,\pi)\) was fixed before the rank recursion;
		\item an \textup{(R4)} realization uses the old last outer predecessor, if any,
		and the old terminal atom \(\eta\) of the completed inner chain.
		Thus its one-entry strong packet has raw support \(\{\operatorname{rank}\eta\}\subseteq(\rho(\pi),n)\);
		\item an \textup{(R5)} realization adds exactly one scalar source-only atom, with old predecessor and old packet.
		It never creates a target atom.
	\end{enumerate}
	In the target cases of \textup{(1)}, \textup{(3)}, and \textup{(4)},
	the simultaneous protected mirror depends only on the hatted copies of those same old predecessor and packet atoms.
	Hence the target atom and its new mirror do not depend on one another.
	Thus, at every rank, we have
	\begin{equation}\label{eq:rnk-car-bnd}
		|\Delta_n^1|\le2,
		\qquad |\Delta_n^0|\le2,
		\qquad |\Delta_n^{0,\mathrm{so}}|\le1.
	\end{equation}
	Every nonzero correction added at rank \(n\) is therefore an admissible type-zero or type-one correction of Lemma~\ref{lem:blk-bd} with \(\beta=m_{\operatorname{lev}(\mathsf{Q})}^{-1}\), a contractive packet map and, in the type-one case,
	a contractive predecessor map.
\end{lemma}

\begin{proof}
	Admissibility gives the following rank checks.
	\begin{enumerate}[label=\textup{(R\arabic*)},leftmargin=*]
		\item The packet is supported in \((k,n)\), and the predecessor, when present, has rank \(k<n\).
		Hence \eqref{eq:src-onl-crr} and \eqref{eq:tgt-sca-crr} use only old atoms.
		The sort condition makes every directly named side-zero predecessor scalar and source-only.
		\item A seed requirement is administrative and adds no corrected atom.
		\item The seed, stem, predecessor chain, and packet are named old dependencies,
		while \(p=\sigma_c(q,\pi)\) is fixed before rank \(n\).
		Thus both correction types use only the old packet and, when present, the old predecessor.
		\item The completed inner chain, its terminal atom \(\eta\), the certificate, and the previous outer stem are old.
		Moreover, \(\rho(\pi)<\operatorname{rank}\eta<n\), so \(\langle\eta\rangle\) is a contractive old packet;
		the previous outer atom, when present, is also old.
		\item The source packet is supported in \((k,n)\), and the optional predecessor has rank \(k<n\).
		Hence \eqref{eq:src-onl-crr} uses only old source atoms and adds no target or protected atom.
	\end{enumerate}
	In every target case, \eqref{eq:fml-prt-crr} uses only the already existing protected partners of these old atoms.
	The choice of \(E_\gamma\) makes all restriction maps defined and contractive,
	so a target atom and its simultaneous mirror do not depend on one another.
	The background probe pair and the possible outputs listed above give \eqref{eq:rnk-car-bnd}.
	Finally, each predecessor map is either the identity or a contractive restriction between nested fibres. Then packet contractivity follows from \eqref{eq:pkt-ctr}.
	Lemma~\ref{lem:blk-bd} therefore applies.
\end{proof}

\subsection{Stage recursion and the finite-extension property}
\label{sec:stg-cns}

We now verify extension, persistence, and fairness for this conditional system.
\begin{enumerate}[label=\textup{Step \arabic*.}]
	\item \emph{Coding.} Fix a G\"odel numbering of the finite formal requirements.
	For every rational template, every finite dependency string, and every pair \((N,o)\in\N_+^2\), the corresponding requirement is coded.
	\item \emph{Activation.} A requirement becomes active once its named dependencies are old and all admissibility conditions hold.
	This depends on only finitely many old records and finitely many queries to \(\mathcal{O}_V\), so activity is persistent.
	\item \emph{Selection.} At rank \(n\):
	\begin{enumerate}[label=\textup{(\arabic*)}]
		\item inspect the requirement codes \(s\le n\);
		\item retain those that are unmet and admissible at rank \(n\);
		\item if any remain, act on the least-coded one; otherwise select no requirement.
	\end{enumerate}
	\item \emph{Consequences.} Every active requirement is eventually met.
	Varying \(N\) and \(o\) gives realizations above every cutoff and arbitrarily many copies,
	while later ranks add new realizations but never redefine old records.
\end{enumerate}

\begin{construction}[Rank recursion]\label{con:rnk-rec}
	Suppose that all ranks smaller than \(n\) have been defined.
	\begin{enumerate}[label=\textup{Step \arabic*.}]
		\item \emph{Probe pair.}
		\begin{enumerate}[label=\textup{(\arabic*)}]
			\item Put \(\gamma_n^{\rm pr}\) in \(\Delta_n^1\) and its paired atom in \(\Delta_n^0\).
			\item Set \(E_{\gamma_n^{\rm pr}}=V_n\), assign their identifiers and records by the probe formula,
			and give both probes zero correction.
		\end{enumerate}
		\item \emph{Requirement.}
		\begin{enumerate}[label=\textup{(\arabic*)}]
			\item Among the unmet requirements of G\"odel number at most \(n\) that are admissible at rank \(n\),
			select the least-coded one; if none exists, proceed to Step 3 without adding an extra atom.
			\item If the selected requirement is a seed, register its seed and proceed to Step 3.
			\item For a selected non-seed requirement, first fix the identifier, fibre, and core record of its output atom.
			\item For a target output, form its structural and coding histories by \eqref{eq:nsl-his}--\eqref{eq:two-his-fld},
			and then form the protected mirror record and histories by \eqref{eq:prt-two-his}.
			For a scalar source-only output, form both histories directly from its core record.
		\end{enumerate}
		\item \emph{Rank closure.}
		\begin{enumerate}[label=\textup{(\arabic*)}]
			\item After all same-rank fibres, identifiers, core records, and histories are fixed,
			define \(F_n^i\) by \eqref{eq:fml-fn}.
			\item Define the corrections by \eqref{eq:src-onl-crr}, \eqref{eq:tgt-sca-crr}, or \eqref{eq:fml-prt-crr}.
			\item Define the new \(d^{i*}\)-block by \eqref{eq:fml-dst}.
			\item Corrections use only old ranks.
			A mirror uses the completed target coding history as its only same-rank datum, never a target correction.
		\end{enumerate}
	\end{enumerate}
\end{construction}

The terms \emph{fairness}, \emph{persistence}, and \emph{no injury} refer to classical notions in recursion theory; see \cite[Section~V.2]{Odifreddi1989}, \cite[Chapters~6--7]{Soare2016}, and \cite{Cooper2004}.

\begin{lemma}[Fairness, persistence, and no injury]\label{lem:fai-per}
	Every requirement which becomes active is met after its prescribed cutoff.
	No later stage changes the record, correction, either of the two histories, or code of an old atom.
	Thus the construction has the no-injury property.
\end{lemma}
\begin{proof}\quad
	\begin{itemize}[leftmargin=*]
		\item \emph{Persistence.} Activity depends only on finitely many old records, strict rank inequalities,
		algebraic compatibility, and fixed finite-dimensional norm relations.
		Once these conditions hold, later ranks cannot destroy them.
		\item \emph{Fairness.} Fix an active requirement of G\"odel number \(s\).
		Once the current rank is at least \(s\), it belongs to the finite search.
		Only finitely many smaller codes can precede it, and each active requirement is acted on at most once.
		Hence it eventually becomes the least active unmet requirement and is met at a rank larger than \(N\).
		\item \emph{No injury.} Every new correction uses only older ranks, every fresh atom has a new identifier,
		and no record, correction, history, or code of an old atom is redefined.
		\item \emph{Cofinality and copies.} For a fixed rational template, distinct pairs \((N,o)\) give distinct requirement codes.
		Applying the fairness argument to each pair gives realizations above every cutoff and arbitrarily many copies.
	\end{itemize}
\end{proof}

A \emph{finite-extension chain} is a finite ordered list
\(\mathscr{S}=(\mathcal{R}_1,\ldots,\mathcal{R}_L)\) of requirement templates such that \(\mathcal{R}_1\) names only old data.
And for \(t<L\), the realized identifiers, ranks, core snapshots, and extended stem produced by \(\mathcal{R}_t\) may be inserted into \(\mathcal{R}_{t+1}\).

Finitely many waiting ranks may intervene.
After \(\mathcal{R}_t\) is realized, we may wait until all entries of the packet for
\(\mathcal{R}_{t+1}\) have been constructed beyond the new predecessor rank.
These entries are then old and have rank below the realization of \(\mathcal{R}_{t+1}\).
Since every finite dependency string is coded, the resulting requirement is included in the scheduling.

\begin{proposition}[Finite-extension realization property]\label{prop:fin-ext}
	The following selections can be made inside the already constructed datum.
	\begin{enumerate}[label={\textup{(\arabic*)}},leftmargin=*]
		\item Every finite admissible rational generic or reverse chain template has a realization after every cutoff.
		Every one-step admissible continuation template of an existing realized chain has a realization after every cutoff.
		\item Seeds of every allowed seed level occur after every cutoff.
		Given a realized outer stem \(\pi\), every finite adaptive sequence of admissible rational packets has a realization as an inner chain of level \(\sigma_c(q,\pi)\), up to age \(n_{\sigma_c(q,\pi)}\).
		\item Given a complete inner chain and a compatible certificate,
		an outer continuation carrying precisely that certificate occurs after every cutoff.
		Iterating this assertion realizes every finite ordered list of compatible certificates, up to outer age \(n_q\).
		\item Every target selection in \textup{(1)}--\textup{(3)} has its protected source mirror at the same ranks and with the corresponding protected-source predecessor chain.
		The resulting chain consists only of protected source atoms.
	\end{enumerate}
	Here adaptive means that the next packet or certificate may be selected after the rank of the preceding realization is known.
\end{proposition}

\begin{proof}
	We argue by induction on the length of a finite-extension chain.
	At the first layer, insert the old identifiers and choose the bound \(N\) and an unused copy index.
	The resulting requirement is active at all sufficiently large ranks,
	so Lemma~\ref{lem:fai-per} realizes it above \(N\).

	Suppose that the first \(t<L\) layers have been realized.
	By the chain condition, all data named by layer \(t+1\) are old and its packet is supported after the local cutoff.
	Choosing \(N\) beyond their ranks and an unused copy index gives an \textup{(R1)}, \textup{(R3)}, or \textup{(R5)} requirement
	that becomes active and is realized by Lemma~\ref{lem:fai-per}.
	Together with the seed case of the first-layer argument, this proves \textup{(1)} and \textup{(2)}.
	For an outer step, completion of the inner chain makes its last atom \(\eta\), its certificate, and the previous stem old.
	Hence \textup{(R4)} applies with packet \(\langle\eta\rangle\), and Lemma~\ref{lem:fai-per} realizes it, extending the stem.
	Iteration with the next level \(\sigma_c(q,\pi)\) proves \textup{(3)}.
	Finally, Construction~\ref{con:rnk-rec} adds the protected source mirror at the same rank in every target layer, proving \textup{(4)}.
\end{proof}

From the operator point of view, Proposition~\ref{prop:fin-ext} provides a countable supply of finite rational realizations.
If an operator violates one of the local estimates used later,
finite-dimensional separation first gives rational packet data with small annihilation error.
Part~II selects a complete inner chain already present in the datum and uses a root perturbation to turn these data into a compatible terminal-coordinate certificate.
It then chooses one of the cofinally pre-existing outer continuations guaranteed by Proposition~\ref{prop:fin-ext}\textup{(3)} that carries this certificate. The datum and the spaces remain fixed throughout Part~II.

\subsection{Evaluation analyses, interval restriction, and common stems}
\label{sec:eva-rst}

Let \(\alpha=\alpha_a\) be a nonzero-weight atom, and read its same-side predecessor chain from \(\mathsf{shist}(\alpha)\).
Thus let
\[ \alpha_1\prec\alpha_2\prec\cdots\prec\alpha_a=\alpha \]
be its predecessor chain, and put \(p_s=\operatorname{rank}\alpha_s\) and \(p_0=\chi(\alpha)\).
If the final fibre is scalar, let \(\varphi_a=\lambda\in B_{\K}\) and set
\(\varphi_s=\varphi_a\) for every \(1\le s<a\).
For a protected vector evaluation, start with \(\varphi_a\in B_{(G_\alpha^0)^*}\) and define recursively
\[ \varphi_{s-1}=R_{\alpha_{s-1},\alpha_s}\varphi_s \qquad(2\le s\le a). \]
At every predecessor step the age drops by one and the rank drops strictly,
and the age-one record has empty predecessor index.
So this recursion is finite.
In a protected chain the symbols \(\alpha_s\) are the hatted same-side atoms stored in \(\mathsf{shist}\),
while the displayed \(R\)'s are exactly the restriction maps stored in their core records.
Let
\[ b_{s,\varphi}^{i*} =B_{\mathbf{b}_{\alpha_s}}^*\varphi_s. \]
In a scalar source-only chain, this means the functional represented by the stored packet.

The next identity is the usual Bourgain--Delbaen evaluation analysis \cite{ArgyrosHaydon2011,Tarbard2013,Motakis2024}, written for scalarized vector-fibre evaluations.
\begin{lemma}[Evaluation analysis]
	\label{lem:uni-eva}
	For every scalarized evaluation \(e_{\alpha,\varphi}^{i*}=\epsilon_\alpha^{i*}(\varphi)|_{X_i}\) of level \(h>0\),
	\begin{equation}\label{eq:ful-vec-ea}
		e_{\alpha,\varphi}^{i*}
		=\sum_{s=1}^a d_{\alpha_s,\varphi_s}^{i*}
		+\frac{1}{m_h}\sum_{s=1}^a
		(I-P_{p_{s-1}}^{i*})b_{s,\varphi}^{i*}.
	\end{equation}
	Here \(a\le n_h\), \(\supp_{\rm raw}\mathbf{b}_{\alpha_s}\subseteq(p_{s-1},p_s)\), and
	\[ \varphi\longmapsto\varphi_s,\qquad \varphi\longmapsto b_{s,\varphi}^{i*} \]
	are both contractions.
\end{lemma}

\begin{proof}
	For each \(1\le s\le a\), \eqref{eq:fml-dst} and the correction formula in the core record of \(\alpha_s\) give the exact one-edge identity
	\[
		\begin{aligned}
			e_{\alpha_s,\varphi_s}^{i*}
			&=d_{\alpha_s,\varphi_s}^{i*}
			+c_{\alpha_s}^{i*}(\varphi_s)|_{X_i}\\
			&=d_{\alpha_s,\varphi_s}^{i*}
			+\mathbf{1}_{\{s>1\}}e_{\alpha_{s-1},\varphi_{s-1}}^{i*}
			+\frac{1}{m_h}(I-P_{p_{s-1}}^{i*})b_{s,\varphi}^{i*}.
		\end{aligned}
	\]
	For a scalar chain this is \eqref{eq:src-onl-crr} or \eqref{eq:tgt-sca-crr}
	with the same scalar \(\lambda\) on each edge.
	For a protected chain it is \eqref{eq:fml-prt-crr} and the predecessor scalarization is precisely \(R_{\alpha_{s-1},\alpha_s}\varphi_s=\varphi_{s-1}\).
	The cases \(a\le2\) follow directly from the preceding identity.
	\begin{align*}
			a=1:\quad e_{\alpha,\varphi}^{i*}
			&=d_{\alpha_1,\varphi_1}^{i*}
			+\frac{1}{m_h}(I-P_{p_0}^{i*})b_{1,\varphi}^{i*},\\
			a=2:\quad e_{\alpha,\varphi}^{i*}
			&=d_{\alpha_2,\varphi_2}^{i*}
			+\frac{1}{m_h}(I-P_{p_1}^{i*})b_{2,\varphi}^{i*}
			+e_{\alpha_1,\varphi_1}^{i*}\\
			&=\sum_{s=1}^2 d_{\alpha_s,\varphi_s}^{i*}
			+\frac{1}{m_h}\sum_{s=1}^2(I-P_{p_{s-1}}^{i*})b_{s,\varphi}^{i*}.
	\end{align*}
	For \(a>2\), unfolding again therefore gives
	\[
		\begin{aligned}
			e_{\alpha,\varphi}^{i*}
			&=d_{\alpha_a,\varphi_a}^{i*}
			+\frac{1}{m_h}(I-P_{p_{a-1}}^{i*})b_{a,\varphi}^{i*}
			+e_{\alpha_{a-1},\varphi_{a-1}}^{i*}\\
			&=\sum_{s=a-1}^a d_{\alpha_s,\varphi_s}^{i*}
			+\frac{1}{m_h}\sum_{s=a-1}^a(I-P_{p_{s-1}}^{i*})b_{s,\varphi}^{i*}
			+e_{\alpha_{a-2},\varphi_{a-2}}^{i*}\\
			&=\cdots=\sum_{s=1}^a d_{\alpha_s,\varphi_s}^{i*}
			+\frac{1}{m_h}\sum_{s=1}^a(I-P_{p_{s-1}}^{i*})b_{s,\varphi}^{i*}.
		\end{aligned}
	\]

	The maps \(\varphi\mapsto\varphi_s\) are contractions because the predecessor maps are identities in scalar chains and contractive restrictions in protected chains.
	Hence \eqref{eq:pkt-ctr} makes each map \(\varphi\mapsto b_{s,\varphi}^{i*}\) a contraction.
	For fixed \(\varphi_s\in B_{(G_{\alpha_s}^i)^*}\), define \(\psi_l\in B_{(G_{\eta_l}^i)^*}\) by
	\[
		\psi_l:=
		\begin{cases}
			\frac{R_l\varphi_s}{\norm{R_l\varphi_s}},&R_l\varphi_s\ne0,\\
			0,&R_l\varphi_s=0,
		\end{cases}.
	\]
	Then
	\[ b_{s,\varphi}^{i*} =\sum_l a_l\norm{R_l\varphi_s}e_{\eta_l,\psi_l}^{i*}, \qquad \sum_l|a_l|\norm{R_l\varphi_s}\le1. \]
	Thus every scalarized packet belongs to the absolutely convex hull of the scalarized ambient coordinates in its packet cell.
\end{proof}

\begin{remark}[Uniqueness]
	When \(\varphi\ne0\), the highest nonzero block of \(e_{\alpha,\varphi}^{i*}\) determines \((\alpha,\varphi)\),
	and the core records determine the whole stored analysis.
	When \(\varphi=0\), uniqueness refers to the stored analysis of the labelled evaluation \((\alpha,0)\), since the zero functional alone does not determine \(\alpha\).
\end{remark}

In \eqref{eq:ful-vec-ea}, the \(s\)-th correction and its raw support are \(m_h^{-1}(I-P_{p_{s-1}}^{i*})b_{s,\varphi}^{i*}\) and \(\supp_{\rm raw}\mathbf{b}_{\alpha_s}\subseteq(p_{s-1},p_s)\).
Before applying the finite-stage extension, by \eqref{eq:fml-fn}, we know that for \(L\subseteq\N\) and \(x=(x_\beta)_{\beta\in\Delta_n^i}\in F_n^i\),
\[ \left\|\bigl(\mathbf{1}_{\{\operatorname{lev}(\beta)\in L\}}x_\beta\bigr)_{\beta\in\Delta_n^i}\right\|_{F_n^i} \le\norm{x}_{F_n^i}. \]

For an FDD interval \(I\), apply \(P_I^{i*}\) to \eqref{eq:ful-vec-ea},
\[ P_I^{i*}e_{\alpha,\varphi}^{i*} =\sum_{s=1}^aP_I^{i*}d_{\alpha_s,\varphi_s}^{i*} +\frac{1}{m_h}\sum_{s=1}^aP_I^{i*}(I-P_{p_{s-1}}^{i*})b_{s,\varphi}^{i*}, \]
and expand \(P_I^{i*}b_{s,\varphi}^{i*} =\sum_l a_l\norm{R_l\varphi_s}P_I^{i*}e_{\eta_l,\psi_l}^{i*}\).
An entry with empty restricted FDD support remains as \(\varnothing\) in the stored list, although it contributes nothing to the sum.
The atom identifiers, \(\mathsf{shist}\), \(\mathsf{thist}\), and copy indices remain unchanged.
For an inner or outer history, so do the realized stems and inner codes of retained terms.
Apply the second formula recursively to every coordinate evaluation in a packet.
At a probe, we have
\[
	e_{\eta,\psi}^{i*}\longmapsto
	\begin{cases}
		d_{\eta,\psi}^{i*},&\operatorname{rank}\eta\in I,\\
		\varnothing,&\operatorname{rank}\eta\notin I.
	\end{cases}
\]
This is the \emph{evaluation-restriction recursion} for \(P_I^{i*}e_{\alpha,\varphi}^{i*}\).
It terminates because every recursive packet entry has smaller rank than its parent.

For the chain above, put
\[ e_s=e_{\alpha_s,\varphi_s}^{i*},\qquad e_0=0, \qquad t_s=\frac{1}{m_h}(I-P_{p_{s-1}}^{i*})b_{s,\varphi}^{i*}, \qquad t_{a+1}=0. \]
For an integer \(w\), let \(r_w\) be determined by
\[ p_{r_w}\le w<p_{r_w+1}, \]
with the conventions \(r_w=0\) for \(w<p_1\) and \(r_w=a\) for \(w\ge p_a\).
Subtracting the evaluation analysis of \(e_{r_w}\) from that of \(e_a\),
and using the successive packet cells, gives
\begin{equation}\label{eq:rst-tel}
	\begin{aligned}
		P_{(w,\infty)}^{i*}e_a
		&=\sum_{s=r_w+1}^a d_{\alpha_s,\varphi_s}^{i*}
		+\sum_{s=r_w+2}^a t_s
		+P_{(w,\infty)}^{i*}t_{r_w+1}\\
		&=e_a-e_{r_w}+(P_{(w,\infty)}^{i*}-I)t_{r_w+1}.
	\end{aligned}
\end{equation}

\begin{lemma}[Interval restriction formula]\label{lem:fml-rst}
	If \(u<v\), then
	\begin{equation}\label{eq:two-bnd-rst}
		P_{(u,v]}^{i*}e_a=e_{r_v}-e_{r_u}+(P_{(u,\infty)}^{i*}-I)t_{r_u+1}-(P_{(v,\infty)}^{i*}-I)t_{r_v+1}.
\end{equation}
\end{lemma}

\begin{proof}
	Subtracting \eqref{eq:rst-tel} at \(v\) from the identity at \(u\) gives \eqref{eq:two-bnd-rst}.
\end{proof}

\begin{remark}[Interval restriction]
	\label{rem:int-rst}
	For a scalarized coordinate evaluation \(Q\) on side \(i\), interval restriction is \(QP_I^i\).
	A packet term \(FP_K^i\) becomes \(FP_{K\cap I}^i\).
	Thus the stored history, predecessors, labels, and \(\sigma_c\)-codes are unchanged.
	Formula~\eqref{eq:two-bnd-rst} shows that only the first and last packet cells meeting \(I\) can be cut at its endpoints.
	Every intervening cell is retained completely, and no interior predecessor fragment is created.
	Hence at most two strong-packet terms in an outer analysis are partially retained.
\end{remark}

\begin{lemma}[Common-stem comparison]
	\label{lem:com-stm}
	Let \(\mathcal{H}_1\) and \(\mathcal{H}_2\) be outer coding histories with the same corner \(c\) and seed level \(q\).
	Each history is read from \(\mathsf{thist}\), which equals \(\mathsf{shist}\) except for a protected mirror.
	Then the following hold.
	\begin{enumerate}[label={\textup{(\arabic*)}},leftmargin=*]
		\item \emph{Different seeds.}
		If their seed records differ, every inner level in one history differs from every inner level in the other.

		\item \emph{Agreement on the common stem.}
		Suppose their seed records agree, and let \(\pi\) be their longest common realized stem.
		Their labelled strong packets agree along \(\pi\).

		\item \emph{The first packets after the common stem.}
		If both histories continue beyond \(\pi\), let \(\mathcal B_1\) and \(\mathcal B_2\) be their respective first packets after \(\pi\).
		Both have inner level
		\[ h_*:=\sigma_c(q,\pi). \]
		Their labelled packets need not agree.

		\item \emph{No other matching levels beyond the common stem.}
		Under the assumption in \textup{(2)}, if a packet after \(\pi\) in \(\mathcal H_1\) and a packet after \(\pi\) in \(\mathcal H_2\) have the same inner level,
		they must be the pair \((\mathcal B_1,\mathcal B_2)\) in \textup{(3)}.
		If either history ends at \(\pi\), this pair is absent.
	\end{enumerate}
\end{lemma}

\begin{proof}
	A packet following a realized stem \(\tau\) has inner level \(\sigma_c(q,\tau)\).
	Since \(c\) and \(q\) are fixed and \(\sigma_c\) is injective,
	\[ \sigma_c(q,\tau_1)=\sigma_c(q,\tau_2) \quad\Longrightarrow\quad \tau_1=\tau_2. \]
	We apply this observation to the four assertions.
	\begin{enumerate}[label={\textup{(\arabic*)}},leftmargin=*]
		\item If the seed records differ, every stem word from one history differs from every stem word from the other, because each word retains its seed record.
		Thus the corresponding inner levels differ.

		\item Common stored stems carry the same labelled strong packets, so the packets agree along \(\pi\).

		\item Each first packet after \(\pi\) uses precisely \(\pi\) as its preceding stem.
		Its inner level is therefore \(\sigma_c(q,\pi)=h_*\).

		\item For any pair of packets after \(\pi\) other than the first pair, their preceding stems are distinct:
		either one is \(\pi\) and the other strictly extends it, or both strictly extend \(\pi\) along different histories.
		In the latter case, equality would give a common realized stem longer than \(\pi\).
		Injectivity now gives distinct inner levels.
	\end{enumerate}
\end{proof}

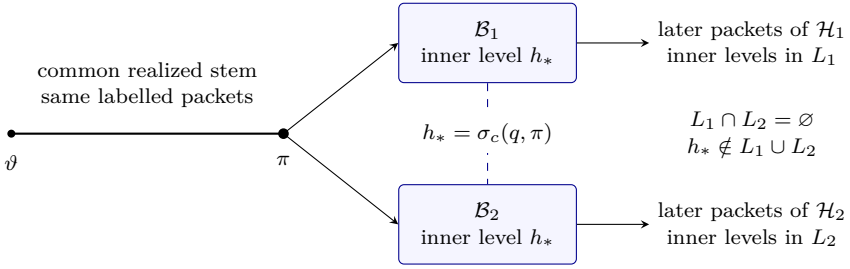
\begin{figure}[htbp]
	\centering
\begin{tikzpicture}[x=1cm,y=1cm,>=stealth,
	every node/.style={font=\small},
	packet/.style={draw=blue!55!black,rounded corners=2pt,
		fill=blue!4,align=center,minimum width=2.35cm,minimum height=1.05cm},
	tail/.style={align=center,minimum width=2.25cm}]
	\coordinate (seed) at (0.3,0);
	\coordinate (split) at (3.9,0);
	\draw[thick] (seed) -- (split);
	\fill (seed) circle (1.5pt);
	\fill (split) circle (2pt);
	\node[below=5pt] at (seed) {\(\vartheta\)};
	\node[below=5pt] at (split) {\(\pi\)};
	\node[align=center] at (2.1,0.65) {common realized stem\\same labelled packets};
	\node[packet] (first1) at (6.6,1.2) {\(\mathcal B_1\)\\inner level \(h_*\)};
	\node[packet] (first2) at (6.6,-1.2) {\(\mathcal B_2\)\\inner level \(h_*\)};
	\draw[->] (split) -- (first1.west);
	\draw[->] (split) -- (first2.west);
	\node[tail] (tail1) at (10.1,1.2) {later packets of \(\mathcal H_1\)\\inner levels in \(L_1\)};
	\node[tail] (tail2) at (10.1,-1.2) {later packets of \(\mathcal H_2\)\\inner levels in \(L_2\)};
	\draw[->] (first1.east) -- (tail1.west);
	\draw[->] (first2.east) -- (tail2.west);
	\draw[dashed,blue!55!black] (first1.south) -- (first2.north);
	\node[fill=white,inner sep=3pt] at (6.6,0) {\(h_*=\sigma_c(q,\pi)\)};
	\node[align=center] at (10.1,0) {\(L_1\cap L_2=\varnothing\)\\\(h_*\notin L_1\cup L_2\)};
\end{tikzpicture}
	\caption{Common-stem comparison.}
	\label{fig:common-stem-comparison}
\end{figure}

\subsection{Probes, shielding, and cofinal scheduling}
\label{sec:fre-prb}

Let
\[ \mathcal{P}=\{\gamma_n^{\rm pr}:n\in\N_+\}. \]
By Construction~\ref{con:rnk-rec}, \(E_{\gamma_n^{\rm pr}}=V_n\), both corrections are zero,
and no member of \(\mathcal{P}\) can be a type-one predecessor.
For \(\gamma\in\mathcal{P}\), define
\[ j_\gamma:E_\gamma\longrightarrow X_0,\qquad j_\gamma q=d_{\widehat{\gamma}}^0(q),\qquad z_\gamma=d_\gamma^1(1),\qquad \phi_\gamma=d_{\gamma,1}^{1*}. \]
Let \(\beta\) be a probe atom on side \(i\), let \(g\in G_\beta^i\), and put \(x=d_\beta^i(g)\). Since a probe is not a predecessor, the evaluation analysis, Lemma~\ref{lem:fml-blk-bio}, the disjointness of the packet cells, and packet contractivity give
\[ |e_{\alpha,\varphi}^{i*}(x)| \le\frac{\norm{x}}{m_h} \qquad\bigl(\operatorname{lev}(\alpha)=h>0,\ \varphi\in B_{(G_\alpha^i)^*}\bigr). \]
Every level-zero atom is a zero-correction probe. Since the scalarized ambient coordinates norm \(X_i\), Lemma~\ref{lem:fml-blk-bio} gives
\[ \norm{g}\le\norm{x} \le\max\left\{\norm{g},\frac{\norm{x}}{m_1}\right\}. \]
Thus \(\norm{d_\beta^i(g)}=\norm{g}\). Moreover, \(\phi_\gamma=e_{\gamma,1}^{1*}\) is contractive and \(\phi_\gamma(z_\gamma)=1\). Hence
\begin{equation}\label{eq:ter-par}
	\norm{j_\gamma q}=\norm{q},\qquad
	\norm{z_\gamma}=\norm{\phi_\gamma}=1,
	\qquad \phi_\gamma(z_\gamma)=1
	\quad(q\in E_\gamma).
\end{equation}
Equation~\eqref{eq:def-tv} gives \(T_vj_\gamma q=v(q)z_\gamma\), and
applying \(\phi_\gamma\) recovers \(v(q)\).
This operator test explains the term \emph{probe}.
Figure~\ref{fig:probe-readout} shows this exact scalar readout.
The lower panel records the scalar functional associated with a general operator.

\begin{figure}[htbp]
	\centering
\begin{tikzpicture}[x=1cm,y=1cm,>=stealth,
	every node/.style={font=\small,align=center},
	probe/.style={draw=blue!55!black,fill=blue!4,rounded corners=3pt,
		minimum height=0.85cm,inner sep=5pt},
	reading/.style={draw=green!45!black,fill=green!4,rounded corners=3pt,
		minimum height=0.85cm,inner sep=5pt}]
	\node[probe,minimum width=1.7cm] (fibre) at (0.9,0)
		{\(u\in E_\gamma\)};
	\node[probe,minimum width=2cm] (source) at (3.6,0)
		{\(j_\gamma u\in X_0\)};
	\node[reading,minimum width=2.3cm] (target) at (7,0)
		{\(v(u)z_\gamma\in X_1\)};
	\node[reading,minimum width=1.6cm] (scalar) at (10.1,0)
		{\(v(u)\in\K\)};

	\node[above=5pt] at (fibre.north) {Vector fibre};
	\node[above=5pt] at (source.north) {Source probe \(\widehat\gamma\)};
	\node[above=5pt] at (target.north) {Target probe \(\gamma\)};
	\node[above=5pt] at (scalar.north) {Scalar value};

	\draw[->,thick,blue!55!black] (fibre.east) --
		node[above=3pt,text=black] {\(j_\gamma\)} (source.west);
	\draw[->,thick] (source.east) --
		node[above=3pt] {\(T_v\)} (target.west);
	\draw[->,thick,green!45!black] (target.east) --
		node[above=3pt,text=black] {\(\phi_\gamma\)} (scalar.west);
	\node[below=5pt] at (target.south) {\(\phi_\gamma(z_\gamma)=1\)};

	\node[draw=black!35,fill=black!2,rounded corners=3pt,
		text width=10.35cm,minimum height=1cm,inner sep=5pt] at (5.5,-1.75)
		{For a general \(S\in\Bcal(X_0,X_1)\):\\[3pt]
		 \(a_\gamma:=\phi_\gamma Sj_\gamma\in E_\gamma^*,\qquad
			 a_\gamma(u)=\phi_\gamma(Sj_\gamma u)\).};
\end{tikzpicture}
	\caption{Coefficient readout by a paired probe.}
	\label{fig:probe-readout}
\end{figure}
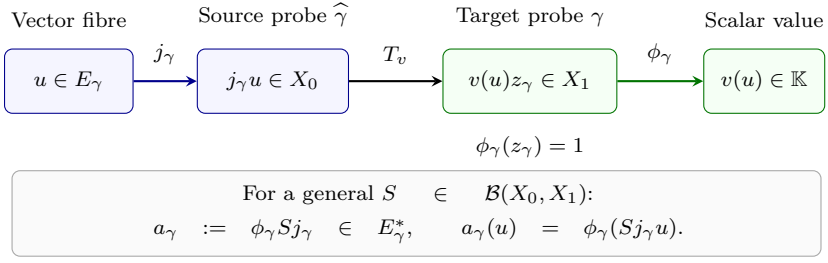

We use \(\supp_{\rm FDD}\) and \(\ran\) as defined in \eqref{eq:fml-fdd-sup}.
For nonzero finitely supported vectors, write \(x<y\) when \(\max\ran x<\min\ran y\).
For an atom \(\alpha\), write
\[ x<\alpha\iff\max\ran x<\operatorname{rank}\alpha, \qquad \alpha<x\iff\operatorname{rank}\alpha<\min\ran x, \]
and write \(\alpha<\beta\) when \(\operatorname{rank}\alpha<\operatorname{rank}\beta\).

\begin{definition}[Rapidly increasing sequences]
	\label{def:fml-ris}
	A block sequence \((x_k)\) in \(X_i\) is a \(C\)-RIS if \(\norm{x_k}\le C\) and there are strictly increasing positive integers \((j_k)\) such that
	\begin{equation}\label{eq:fml-ris-def}
		j_{k+1}>\max\ran x_k,
		\qquad
		|e_{\alpha,\psi}^{i*}(x_k)|\le\frac{C}{m_h}
	\end{equation}
	whenever \(0<{\operatorname{lev}(\alpha)}=h<j_k\) and \(\psi\in B_{(G_\alpha^i)^*}\).
\end{definition}
A block sequence consists of nonzero vectors, and zero terms are omitted whenever a family is regarded as a block sequence or an RIS.
Level-zero probes are treated separately in the following proposition.

\begin{proposition}[Probe shielding and RIS scheduling]\label{prop:fml-prb}\quad
	\begin{enumerate}[label={\textup{(P\arabic*)}},leftmargin=*]
		\item For every \(r,N\in\N_+\) there is \(\gamma\in\mathcal{P}\) with
		\[ \operatorname{rank}\gamma>N,\qquad E_\gamma\supseteq V_r. \]

		\item For every \(h>0\) and \(\gamma\in\mathcal{P}\),
		\begin{equation}\label{eq:fml-prb-shd}
			\begin{aligned}
				\sup_{\substack{\eta\ \text{target on side }1\\
				\operatorname{lev}(\eta)=h}}
				|e_\eta^{1*}(z_\gamma)|
				&\le\frac{M}{m_h},\\
				\sup_{\substack{\alpha\ \text{ on side }0,\ {\operatorname{lev}(\alpha)}=h\\
				\psi\in B_{(G_\alpha^0)^*},\
				q\in B_{E_\gamma}}}
				|e_{\alpha,\psi}^{0*}(j_\gamma q)|
				&\le\frac{M}{m_h},\\
				\sup_{\substack{\eta\ \text{target on side }1,\ \operatorname{lev}(\eta)=h\\
				q\in B_{E_\gamma}}}
				\norm{U_\eta j_\gamma q}_V
				&\le\frac{M}{m_h}.
			\end{aligned}
		\end{equation}
		In addition, the following scheduling assertions hold.
		\begin{enumerate}[label=\textup{(\alph*)},leftmargin=*]
			\item given arbitrary lower bounds \(N_k\), probes \(\gamma_k\) and strictly increasing RIS indices \(s_k\) can be chosen so that
			\[ N_k<s_k<\operatorname{rank}\gamma_k<s_{k+1}. \]
			If \(q_k\in B_{E_{\gamma_k}}\),
			\((j_{\gamma_k}q_k)\) and \((z_{\gamma_k})\) are \(M\)-RISs;
			\item given arbitrary lower bounds \(N_k\), pairs of probes and strictly increasing RIS indices can be chosen so that
			\[ N_k<s_k<\operatorname{rank}\gamma_k <\operatorname{rank}\delta_k<s_{k+1}. \]
			If \(q_k\in B_{E_{\gamma_k}\cap E_{\delta_k}}=B_{E_{\gamma_k}}\),
			both component sequences are \(M\)-RISs and \((j_{\gamma_k}q_k+j_{\delta_k}q_k)\) is a \(2M\)-RIS;
			\item if \((x_k)\subseteq X_0\) is a \(C\)-RIS and \(N_k\) are arbitrary lower bounds,
			we may pass to a subsequence and choose
			\[ x_k<\gamma_k^{\rm pr}<x_{k+1},\qquad \gamma_k^{\rm pr}\in\mathcal{P}, \]
			with \(\operatorname{rank}\gamma_k^{\rm pr}>N_k\), so that uniformly for \(q_k\in B_{E_{\gamma_k^{\rm pr}}}\),
			the sequence \((x_k+j_{\gamma_k^{\rm pr}}q_k)\) is a \((C+M)\)-RIS.
		\end{enumerate}
	\end{enumerate}
\end{proposition}

\begin{proof}
	For \textup{(P1)}, given a fibre \(V_r\) and a rank cutoff \(N\), choose \(n>\max\{N,r\}\).
	The rank-\(n\) probe satisfies \(E_{\gamma_n^{\rm pr}}=V_n\supset V_r\), so the required fibres occur cofinally.

	For \textup{(P2)}, the calculation preceding \eqref{eq:ter-par} gives
	\[ \bigl|e_{\alpha,\varphi}^{i*}(d_\beta^i(g))\bigr| \le\frac{\norm{g}}{m_h} \qquad\bigl(\operatorname{lev}(\alpha)=h>0,\ \varphi\in B_{(G_\alpha^i)^*}\bigr). \]
	By applying the preceding estimate to the cases \((i,\beta,g)=(1,\gamma,1)\) and \((0,\widehat{\gamma},q)\), and then taking the relevant suprema, we have the first two inequalities of \eqref{eq:fml-prb-shd}. Moreover,
	\[
		\sup_{\substack{\eta\ \mathrm{target},\ \operatorname{lev}(\eta)=h\\q\in B_{E_\gamma}}}
		\norm{U_\eta j_\gamma q}_V
		=\sup_{\substack{\eta\ \mathrm{target},\ \operatorname{lev}(\eta)=h\\
			q\in B_{E_\gamma},\ \varphi\in B_{E_\eta^*}}}
		\bigl|e_{\widehat{\eta},\varphi}^{0*}(j_\gamma q)\bigr|
		\le\frac{M}{m_h}.
	\]
	This proves \eqref{eq:fml-prb-shd}.

	For \,\textup{(a)}, choose recursively
	\[ r_k=\operatorname{rank}\gamma_k,\qquad N_k<s_k<r_k<s_{k+1}. \]
	For \(q_k\in B_{E_{\gamma_k}}\), \(0<h<s_k\), and scalarizations \(e_\eta^{1*}\), \(e_{\alpha,\psi}^{0*}\) of level \(h\), we have
	\[
		\begin{aligned}
			\supp_{\rm FDD}z_{\gamma_k}&=\{r_k\},&
			\supp_{\rm FDD}j_{\gamma_k}q_k&\subseteq\{r_k\},\\
			\norm{z_{\gamma_k}}&\le M,&
			\norm{j_{\gamma_k}q_k}&\le M,\\
			|e_\eta^{1*}(z_{\gamma_k})|&\le\frac{M}{m_h},&
			|e_{\alpha,\psi}^{0*}(j_{\gamma_k}q_k)|&\le\frac{M}{m_h}.
		\end{aligned}
	\]
	These are the two \(M\)-RIS assertions.

	For \,\textup{(b)}, choose recursively
	\[ r_k=\operatorname{rank}\gamma_k,\qquad t_k=\operatorname{rank}\delta_k,\qquad N_k<s_k<r_k<t_k<s_{k+1}. \]
	If \(q_k\in B_{E_{\gamma_k}}\) and
	\(y_k=j_{\gamma_k}q_k+j_{\delta_k}q_k\), then every source scalarization
	\(e_{\alpha,\psi}^{0*}\) of level \(0<h<s_k\) satisfies
	\[ \supp_{\rm FDD}y_k\subseteq\{r_k,t_k\},\qquad \norm{y_k}\le2M,\qquad |e_{\alpha,\psi}^{0*}(y_k)|\le\frac{2M}{m_h}. \]
	Thus the component sequences are \(M\)-RISs and \((y_k)\) is a \(2M\)-RIS.

	For \,\textup{(c)}, let \((u_l)\) be RIS indices for the original \(C\)-RIS.
	Choose \((l_k)\) and probes \((\gamma_k^{\rm pr})\) so that, with \(r_k=\operatorname{rank}\gamma_k^{\rm pr}\),
	\[ x_{l_k}<\gamma_k^{\rm pr}<x_{l_{k+1}},\qquad r_k>N_k,\qquad u_{l_{k+1}}>r_k. \]
	For \(q_k\in B_{E_{\gamma_k^{\rm pr}}}\), put \(y_k=x_{l_k}+j_{\gamma_k^{\rm pr}}q_k\).
	Then every source scalarization \(e_{\alpha,\psi}^{0*}\) of level \(0<h<u_{l_k}\) satisfies
	\[
		\begin{aligned}
			y_k&<y_{k+1},&
			\max\ran y_k&\le r_k<u_{l_{k+1}},\\
			\norm{y_k}&\le C+M,&
			|e_{\alpha,\psi}^{0*}(y_k)|&\le\frac{C+M}{m_h}.
		\end{aligned}
	\]
	Hence \((y_k)\) is a \((C+M)\)-RIS and the estimates hold uniformly in \((q_k)\).
\end{proof}

The following theorem records the construction properties used in Part~II.
Shrinking, the paired-mean estimate, and the operator-orbit estimates will also be established in Part~II.

\begin{theorem}[Construction and analytic input]\label{thm:dat-cns}
	Construction~\ref{con:rnk-rec} is a well-defined two-sorted block BD datum and defines separable spaces \(X_0,X_1\) with FDDs.
	\begin{enumerate}[label={\textup{(\arabic*)}},leftmargin=28pt]
		\item Each \(F_n^i\) is finite dimensional, and
		\[ \sup_n\norm{P_n^{i*}}\le M=(1-2/m_1)^{-1} \qquad(i=0,1). \]

		\item The target and protected-source identities \eqref{eq:tgt-ter}--\eqref{eq:src-lft-t1} hold, and for every target atom \(\gamma\), \(x\in X_0^{\rm alg}\), and \(v\in V^*\),
		\[ e_\gamma^{1*}(T_vx)=v(U_\gamma x),\qquad \norm{T_v}\le\norm{v}. \]

		\item Every nonzero-weight evaluation has the scalarized analysis \eqref{eq:ful-vec-ea}.
		Generic, inner, outer, reverse, and protected evaluations exhaust these cases.

		\item Every active finite rational requirement is met, and every fixed admissible template has cofinal realizations.
		Interval restriction preserves packet codes and can cut only the first and last packet cells meeting the interval.
		Outer histories with different seeds have disjoint sets of inner levels, while histories with the same seed agree along their common realized stem and can share a later level only immediately after that stem.
		The probes occur cofinally with fibres containing every prescribed \(V_r\) and satisfy \eqref{eq:fml-prb-shd} together with the RIS scheduling conclusions of Proposition~\ref{prop:fml-prb}.
	\end{enumerate}
\end{theorem}

\begin{proof}
	At every rank, seed requirements add no corrected atoms, probe corrections are zero, and every nonzero correction satisfies
	\[ \norm{A^*}\le1,\qquad \norm{B^*}\le1,\qquad 0\le\beta\le m_1^{-1}<\frac12. \]
	By Lemma~\ref{lem:blk-bd} and rank induction, we have a well-defined construction and \textup{(1)}.

	Whenever \(\eta_l\) or \(\xi\) occurs in the correction of a target atom \(\gamma\),
	\[ E_{\eta_l}\subseteq E_\gamma,\qquad R_{\eta_l\gamma}=J_{\eta_l\gamma}^*,\qquad E_\xi\subseteq E_\gamma,\quad R_{\xi\gamma}=J_{\xi\gamma}^*\quad(\xi\ne\varnothing). \]
	By \eqref{eq:fml-prt-crr}, \eqref{eq:prt-two-his}, and Lemma~\ref{lem:int}, we have \textup{(2)}.
	By Lemma~\ref{lem:uni-eva}, we have \textup{(3)}.

	By Lemma~\ref{lem:fai-per} and Proposition~\ref{prop:fin-ext}, we have the requirement and cofinal-realization assertions.
	By Lemma~\ref{lem:fml-rst}, Remark~\ref{rem:int-rst}, and Lemma~\ref{lem:com-stm}, we have the interval-restriction and common-stem assertions.
	By Proposition~\ref{prop:fml-prb}, we have the probe assertions, and hence \textup{(4)} follows.
\end{proof}

\subsection{Analytic input}\label{sec:ana-inp}

\subsubsection{Components of the analytic input}

The following consequences of Theorem~\ref{thm:dat-cns} are denoted by
\hyperref[itm:ana-g1]{\textup{(G1)}}--\hyperref[itm:ana-g7]{\textup{(G7)}}
and used throughout Part~II.

\begin{enumerate}[label=\textup{(G\arabic*)},leftmargin=*]
	\item\label{itm:ana-g1} \emph{Generic realization.} For every \(h\in\mathcal{G}\), every finite admissible rational generic chain template of level \(h\) has cofinally many realizations on either side.
	If \(\alpha_1\prec\cdots\prec\alpha_a\), \(p_0=\chi\), \(p_s=\operatorname{rank}\alpha_s\), and
	\(\mathbf{b}_s=((a_{s,l},\eta_{s,l},R_{s,l}))_{l=1}^{t_s}\) is the packet in the \(s\)-th cell, then
	\[ a\le n_h,\qquad \supp_{\rm raw}\mathbf{b}_s\subseteq(p_{s-1},p_s),\qquad \sum_{l=1}^{t_s}|a_{s,l}|\norm{R_{s,l}}\le1 \quad(1\le s\le a). \]
	This follows from Proposition~\ref{prop:fin-ext}{\textup{(1)}}, \eqref{eq:pkt-ctr}, and \eqref{eq:pkt-cel}.

	\item\label{itm:ana-g2} \emph{Protected lifts.} Every target chain has a simultaneous protected source lift
	\[ \gamma_1\prec\cdots\prec\gamma_a \longmapsto \widehat{\gamma}_1\prec\cdots\prec\widehat{\gamma}_a, \qquad \mathsf{thist}(\widehat{\gamma}_s)=\mathsf{shist}(\gamma_s) \quad(1\le s\le a). \]
	Its corrections are \eqref{eq:src-lft-ter}--\eqref{eq:src-lft-t1} and \eqref{eq:fml-prt-crr}, while \eqref{eq:prt-two-his} and Proposition~\ref{prop:fin-ext}\textup{(4)} give the history and simultaneous realization.

	\item\label{itm:ana-g3} \emph{Certified realization.} For each corner \(c\in\{00,11,01\}\) and seed level \(q\in\mathcal{C}_c^{\rm seed}\), every compatible ordered list of at most \(n_q\) separation certificates extends to a level-\(q\) outer history in its own weight pool.
	After a seed level \(q\) and a realized outer stem \(\pi\), the next inner level satisfies
	\begin{equation}\label{eq:syn-cod-grw}
		p=\sigma_c(q,\pi),\qquad
		p>\max\{q,R(\pi)\},\qquad
		m_p\ge2^{q+|\pi|+10}L_qn_qm_q^4.
	\end{equation}
	Here \(\sigma_c\) is injective, and admissible level-\(p\) inner continuations occur cofinally after every cutoff up to age \(n_p\).
	Let \(i\to j\) be the corner specified by \(c\). If \(\zeta_1\prec\cdots\prec\zeta_a\) is the outer chain on side \(j\) and \(\eta_s\) is the terminal atom of its \(s\)-th completed inner chain, let \(\langle\eta_s\rangle\) be the complete strong packet and put \(f_s=e_{\eta_s}^{j*}\), its induced functional on \(X_j\). The restrictions of the ambient corrections to \(X_j\) satisfy
	\[
		\begin{aligned}
			r_0&=\chi,& r_s&=\operatorname{rank}\zeta_s &&(1\le s\le a\le n_q),\\
			{\left.c_{\zeta_1}^{j*}(1)\right|_{X_j}}
			&={\frac{1}{m_q}f_1P_{(r_0,\infty)}^j},&
			{\left.c_{\zeta_s}^{j*}(1)\right|_{X_j}}
			&={e_{\zeta_{s-1}}^{j*}+\frac{1}{m_q}f_sP_{(r_{s-1},\infty)}^j} &&(2\le s\le a).
		\end{aligned}
	\]
	The complete strong packets are \(\langle\eta_s\rangle\), and interval restriction preserves their codes. The forward corner also has the protected lift from \hyperref[itm:ana-g2]{\textup{(G2)}}.
	These assertions follow from \eqref{eq:fml-sig}--\eqref{eq:qnt-cod-grw}, Lemma~\ref{lem:req-adm}{\textup{(4)}}, Proposition~\ref{prop:fin-ext}{\textup{(2)--(4)}}, and Remark~\ref{rem:int-rst}.

	\item\label{itm:ana-g4} \emph{Reverse weights.} For every \(r\in\mathcal{R}\), every rational source-only chain \(\xi_1\prec\cdots\prec\xi_a\), \(a\le n_r\), occurs cofinally, while no target chain has level \(r\).
	With \(p_0=\chi\) and \(p_s=\operatorname{rank}\xi_s\), let \(\mathbf{h}_s\) be the contractive represented source packet stored at \(\xi_s\), and put \(h_s=\left.B_{\mathbf{h}_s}^*(1)\right|_{X_0}\). The restrictions of its ambient corrections to \(X_0\) are
	\[ \left.c_{\xi_1}^{0*}(1)\right|_{X_0}=\frac{1}{m_r}h_1P_{(p_0,\infty)}^0,\qquad \left.c_{\xi_s}^{0*}(1)\right|_{X_0}=e_{\xi_{s-1}}^{0*}+\frac{1}{m_r}h_sP_{(p_{s-1},\infty)}^0 \quad(2\le s\le a), \]
	where \(\supp_{\rm raw}\mathbf{h}_s\subseteq(p_{s-1},p_s)\).
	This follows from Proposition~\ref{prop:fin-ext}{\textup{(1)}}, Lemma~\ref{lem:req-adm}{\textup{(5)}}, and \eqref{eq:ful-vec-ea}.

	\item\label{itm:ana-g5} \emph{Interval restriction.} If \(K_s=(p_{s-1},p_s)\) are the cells of a chain, restriction to an FDD interval \(I\) acts by
	\[ FP_{K_s}^i\longmapsto FP_{K_s\cap I}^i, \qquad \#\{s :  K_s\cap I\notin\{\varnothing,K_s\}\}\le2. \]
	Thus only the first and last cells meeting \(I\) may be partial.
	The stored history, predecessors, packet codes, and certificates are unchanged, and every certificate support remains in its original cell.
	These claims follow from Lemma~\ref{lem:fml-rst}, Remark~\ref{rem:int-rst}, and \eqref{eq:two-bnd-rst}.

	\item\label{itm:ana-g6} \emph{Probe terminals.} Probes have level zero and are never predecessors.
	For every \(r,N\in\N_+\), there is \(\gamma\in\mathcal{P}\) such that
	\[
		\begin{gathered}
			\operatorname{rank}\gamma>N,\qquad E_\gamma\supseteq V_r,\\
			T_vj_\gamma q=v(q)z_\gamma,\qquad
			\norm{j_\gamma q}=\norm{q},\qquad
			\norm{z_\gamma}=\norm{\phi_\gamma}=1,\\
			\phi_\gamma(z_\gamma)=1
			\quad(q\in E_\gamma,\ v\in V^*).
		\end{gathered}
	\]
	Step~1 of Construction~\ref{con:rnk-rec} and Proposition~\ref{prop:fml-prb} give cofinal probes, the shielding bounds \eqref{eq:fml-prb-shd}, and the one-probe, two-probe, and probe-added RIS selections.

	\item\label{itm:ana-g7} \emph{Weight-preserving predecessor chains.} Every type-one predecessor edge, including a protected edge, preserves the root level and hence the root weight. If \(\xi\) is the immediate predecessor of \(\alpha\), then \(\mathsf{shist}(\xi)\) is an initial segment of \(\mathsf{shist}(\alpha)\).
	For \(x=(x_\alpha)_{\alpha\in\Delta_n^i}\in F_n^i\) and \(A\subseteq\Delta_n^i\), the block structure gives
	\[ \norm{\bigl(\mathbf{1}_{\{\alpha\in A\}}x_\alpha\bigr)_{\alpha\in\Delta_n^i}}_{F_n^i} \le\norm{x}_{F_n^i}. \]
	Source-only terminal pieces retain the chain weight, while probes have formal level zero.
	The predecessor and history rules are \eqref{eq:atm-rec}, \eqref{eq:nsl-his}, \eqref{eq:fml-prt-crr}, and \eqref{eq:prt-two-his}, and the contractive coordinate deletion follows from \eqref{eq:fml-fn}.
\end{enumerate}

A realized inner or outer code records its stem, corner, vectors, certificates, orbit, semantic records, intervals, cutoffs, and copy index.
Packet translation preserves the semantic \(V\)-data but receives a new code.
The seed, coded, generic, and reverse weight pools are pairwise disjoint.

\subsubsection{Scalarized evaluation grammar}

Each atom \(\gamma\) has a level \(\operatorname{lev}(\gamma)\in\N\).
Level zero means weight zero, level \(h>0\) means weight \(m_h^{-1}\), and a protected atom has the level of its target.
For \(i\in\{0,1\}\), let
\[
	\begin{gathered}
		e_{\gamma,\varphi}^{i*}=\epsilon_\gamma^{i*}(\varphi)|_{X_i},\\
		\mathcal{K}_i
		=\left\{e_{\gamma,\varphi}^{i*}\,\middle|\,
		n\in\N_+,\ \gamma\in\Delta_n^i,
		\varphi\in B_{(G_\gamma^i)^*}\right\}.
	\end{gathered}
\]
Elements of the ambient predual act on \(X_i\subseteq\mathcal E_i\) via the canonical embedding.
When \(G_\gamma^i=\K\), the scalarization condition is \(|\varphi|\le1\).
The construction and Lemma~\ref{lem:uni-eva} give the following properties, denoted by
\hyperref[itm:ana-ea1]{\textup{(EA1)}}--\hyperref[itm:ana-ea4]{\textup{(EA4)}}.

\begin{enumerate}[label=\textup{(EA\arabic*)},leftmargin=32pt]
	\item\label{itm:ana-ea1} Level-zero atoms are probes.
	If \(\operatorname{lev}(\gamma)=h>0\) and \(\xi_1\prec\cdots\prec\xi_a=\gamma\), then
	\begin{equation}\label{eq:syn-vec-ea}
		e_{\gamma,\varphi}^{i*}
		=\sum_{s=1}^{a}d_{\xi_s,\varphi_s}^{i*}
		+\frac{1}{m_h}\sum_{s=1}^{a}
		b_{s,\varphi}^{i*}P_{(p_{s-1},\infty)}^i.
	\end{equation}
	We call the terms in the first sum of \eqref{eq:syn-vec-ea} the BD parts of the analysis.
	Here
	\[ a\le n_h,\qquad p_0=\chi,\qquad p_s=\operatorname{rank}\xi_s,\qquad {\supp_{\rm raw}\mathbf{b}_{\xi_s}}\subseteq(p_{s-1},p_s), \]
	and each \(b_{s,\varphi}^{i*}\) is an absolutely convex combination of members of \(\mathcal{K}_i\) in the same cell.
	Moreover,
	\[ \norm{\varphi_s}\le\norm{\varphi},\qquad \norm{b_{s,\varphi}^{i*}}\le\norm{\varphi} \quad(1\le s\le a). \]

	\item\label{itm:ana-ea2} Along every positive-level predecessor chain,
	\[ \operatorname{lev}(\xi_s)=h\le\operatorname{rank}\xi_s, \qquad \operatorname{age}(\xi_s)=s \quad(1\le s\le a\le n_h). \]
	Every generic, inner, outer, and reverse history has this form.

	\item\label{itm:ana-ea3} Protected lifts are stable under scalarization.
	For a target packet
	\(\mathbf{b}=((a_l,\eta_l,R_l))_{l=1}^s\), write
	\(R_l(\lambda)=\lambda r_l\) and
	\(\widehat R_l=r_lR_{\eta_l\gamma}\) as in \eqref{eq:fml-prt-crr}.
	For \(\varphi\in B_{E_\gamma^*}\), packet contractivity gives
	\[ \norm{\sum_{l=1}^s a_l \epsilon_{\widehat{\eta}_l}^{0*}(\widehat R_l\varphi)}_{\ell_1} \le\sum_{l=1}^s|a_l|\norm{\widehat R_l}\norm{\varphi} \le\norm{\varphi}\le1. \]
	Hence every lifted packet is an absolutely convex combination of scalarized source evaluations in the same cell.

	\item\label{itm:ana-ea4} For \(x=(x_\alpha)_{\alpha\in\Delta_n^i}\in F_n^i\) and \(L\subseteq\N\), coordinate deletion by level satisfies
	\[ \norm{\bigl(\mathbf{1}_{\{\operatorname{lev}(\alpha)\in L\}}x_\alpha\bigr)_{\alpha\in\Delta_n^i}}_{F_n^i} \le\norm{x}_{F_n^i}. \]
	Thus every scalarized evaluation either comes from a level-zero probe or has a positive-level predecessor chain and satisfies \eqref{eq:syn-vec-ea}.
\end{enumerate}

Equation~\eqref{eq:syn-vec-ea} is obtained by iterating \eqref{eq:blk-t0}--\eqref{eq:blk-t1} and telescoping.
Together with \eqref{eq:fml-fn}, this shows that Lemma~\ref{lem:blk-bd} applies to generic, inner, outer, reverse, and protected chains.

\part{Analytic estimates and operator classification}

Part~II uses the construction theorem and the analytic input from Part~I. It does not use the order in which the stage construction enumerates the requirements.
We first prove the basic inequality and the estimates for RIS averages.
And then we use exact pairs and dependent sequences for the self and forward corners,
and the reserved missing levels for the reverse corner.
The last section passes from the local orbit estimates to the global operator classification.

\section{Rapidly increasing sequences (RIS) and the basic inequality}

The first analytic step is a basic inequality for scalarized evaluation analyses.
We use the RIS and basic-inequality method for BD spaces,
in the form developed in the AH construction and adapted in subsequent work
\cite{ArgyrosHaydon2011,Tarbard2012,Tarbard2013,Zisimopoulou2014,MotakisPuglisiZisimopoulou2016,ManoussakisPelczarSwietek2017,ArgyrosMotakis2019,MotakisPuglisiTolias2020,Motakis2024,MotakisPelczar2025}.
The RIS method, the basic inequality, and tree estimates for auxiliary
averages are well-established tools.
For their background in recursive norming, weighted averages, and mixed
Tsirelson spaces, see
\cite{Schlumprecht1991,GowersMaurey1993,GowersMaurey1997,ArgyrosDeliyanni1997,ArgyrosDeliyanniKutzarovaManoussakis1998}.
A general HI method is developed in \cite{ArgyrosTolias2004}.
For the BD and AH variants and their subsequent development, see also
\cite{Tarbard2013,Zisimopoulou2014,ManoussakisPelczarSwietek2017,MotakisPuglisiTolias2020,MotakisPelczar2025}.

On the protected source side, the packets are scalarizations of finite-dimensional vector fibres.
The estimates below are uniform in the fibre dimensions and are used for RIS averages, shrinking,
and the compactness test.

\subsection{Auxiliary norming set}

Recall the constants \(\kappa,C_{\rm tr},B\) from \eqref{eq:cst-led}.
The FDD interval projections have norm at most \(\kappa\).
Let \((e_k)_{k\geq1}\) be the canonical unit vectors in \(c_{00}(\N_+)\),
and let \((e_k^*)_{k\geq1}\) be the associated coordinate functionals.
For nonzero \(f,g\in c_{00}(\N_+)^*\), write \(f<g\) when
\(\max\operatorname{supp}f<\min\operatorname{supp}g\).
Let \(W=W[(\mathcal{A}_{3n_h},m_h^{-1})_{h\geq1}]\) be the smallest subset of \(c_{00}(\N_+)^*\) containing every \(\zeta e_k^*\),
where \(|\zeta|=1\), and closed under
\[ f_1<\cdots<f_d,\quad d\leq3n_h \quad\Longrightarrow\quad \frac{1}{m_h}\sum_{r=1}^{d}f_r\in W. \]
If the outermost defining operation of \(g\in W\) is an \((\mathcal{A}_{3n_h},m_h^{-1})\)-operation, we call \(h\) the root level of \(g\).
The completion of \(c_{00}(\N_+)\) under the norm induced by \(W\) is the corresponding mixed Tsirelson space
\(T[(\mathcal{A}_{3n_h},m_h^{-1})_{h\geq1}]\) in the sense of Argyros and Deliyanni \cite{ArgyrosDeliyanni1997}.
We denote its norm by \(\|\cdot\|_T\) and use it as the auxiliary norm in the basic inequality and the estimates for RIS averages.

\subsection{Truncation estimates}

We use the RIS terminology of Definition~\ref{def:fml-ris}.

\begin{lemma}[Two truncation estimates]\label{lem:vec-ris-trn}
	Let \((x_k)\) be a \(C\)-RIS with RIS indices \((j_k)\), let \(s\in\N\),
	and let \(f=e_{\gamma,\varphi}^{i*}\in\mathcal{K}_i\) have level \(h>0\).
	Then
	\[
		|f(P_{(s,\infty)}x_k)|\leq
		\begin{cases}
			C_{\rm tr}C/m_h,&h<j_k,\\
			\kappa C/m_h,&h\geq j_{k+1}.
		\end{cases}
	\]
\end{lemma}

\begin{proof}
	Write \(f=e_a\) and use the notation \(p_r,e_r,t_r,r_w\) attached to
	\eqref{eq:rst-tel}--\eqref{eq:two-bnd-rst} for its stored chain.
	By \eqref{eq:tal-win-cns} and \hyperref[itm:ana-ea1]{\textup{(EA1)}},
	\(t_r=m_h^{-1}P_{(p_{r-1},p_r)}^{i*}b_{r,\varphi}^{i*}\), where \(\norm{b_{r,\varphi}^{i*}}\leq1\).
	Put \(\nu=\max\operatorname{ran}x_k\).
	The assertion is immediate when \(s\geq\nu\), so we assume that \(s<\nu\).
	And then \(P_{(s,\infty)}^ix_k=P_{(s,\nu]}^ix_k\).

	Suppose first that \(h<j_k\).
	Every nonzero \(e_{r_w}\) is an inherited level-\(h\) evaluation with scalarization in the appropriate dual unit ball.
	Hence the RIS condition gives
	\[ |e_{r_s}(x_k)|\leq\frac{C}{m_h} \quad\text{and}\quad |e_{r_\nu}(x_k)|\leq\frac{C}{m_h}, \]
	where \(e_0=0\).
	For \(w\in\{s,\nu\}\) and \(r_w<a\), the formula for \(t_{r_w+1}\) gives
	\begin{align*}
		\left|\bigl((P_{(w,\infty)}^{i*}-I)t_{r_w+1}\bigr)(x_k)\right|
		&=\frac{1}{m_h}\left|
		b_{r_w+1,\varphi}^{i*}
		\bigl(P_{(p_{r_w},p_{r_w+1})\cap(0,w]}^ix_k\bigr)\right|\\
		&\leq\frac{1}{m_h}
		\norm{b_{r_w+1,\varphi}^{i*}}
		\norm{P_{(p_{r_w},p_{r_w+1})\cap(0,w]}^i}
		\norm{x_k}
		\leq\frac{\kappa C}{m_h}.
	\end{align*}
	If \(r_w=a\), the same boundary term is zero.
	The displayed identity also covers \(r_w=0\) and \(w<p_0\), when the intersection is empty.
	Applying \eqref{eq:two-bnd-rst} now yields
	\begin{align*}
		\quad|f(P_{(s,\infty)}^ix_k)|&=|(P_{(s,\nu]}^{i*}e_a)(x_k)|\\
		&\leq |e_{r_\nu}(x_k)|+|e_{r_s}(x_k)|
		+\sum_{w\in\{s,\nu\}}
		\left|\bigl((P_{(w,\infty)}^{i*}-I)t_{r_w+1}\bigr)(x_k)\right|\\
		&\leq\frac{2C+2\kappa C}{m_h}
		=\frac{C_{\rm tr}C}{m_h}.
	\end{align*}

	Suppose now that \(h\geq j_{k+1}\).
	Since \(\alpha_1\) is the start atom of the level-\(h\) chain,
	the admissibility condition preceding Lemma~\ref{lem:req-adm} gives
	\(p_1=\operatorname{rank}\alpha_1\geq h\).
	The RIS range condition gives \(j_{k+1}>\nu\).
	Hence \(p_1\geq h\geq j_{k+1}>\nu\).
	Thus every \(d_{\alpha_r,\varphi_r}^{i*}\) is supported after \(x_k\),
	and for \(r\geq2\), the cell \((p_{r-1},p_r)\) is also supported after \(x_k\).
	Hence only the first packet can act, and the product of the two coordinate projections is the projection onto their intersection.
	\begin{align*}
		|f(P_{(s,\infty)}^ix_k)|
		&=\frac{1}{m_h}\left|
		b_{1,\varphi}^{i*}
		\bigl(P_{(p_0,p_1)\cap(s,\nu]}^ix_k\bigr)\right|\leq\frac{\kappa C}{m_h}.
	\end{align*}
\end{proof}

\subsection{Basic inequality}

\begin{lemma}[General vector-block basic inequality]
	\label{lem:vec-bas-ine}
	Let \(I\) be a finite integer interval and let \((x_k)_{k\in I}\) be the corresponding finite part of a \(C\)-RIS in either space,
	let \((\lambda_k)_{k\in I}\) be scalars, let \(s\in\N\), and let \(f\in\mathcal{K}_i\).
	There are \(k_0\in I\) and \(g\in W\cup\{0\}\) such that
	\begin{equation}\label{eq:gen-vec-bi}
		\left|f\left(P_{(s,\infty)}
		\sum_{k\in I}\lambda_kx_k\right)\right|
		\leq BC\left(\,|\lambda_{k_0}|+
		g\left(\sum_{k\in I}|\lambda_k|e_k\right)\right).
	\end{equation}
	The functional \(g\) has nonnegative coefficients.
	If \(g\ne0\), then \(f\) has a positive level \(h\), \(g\) has root level \(h\), and \(\operatorname{supp}g\subseteq\{k\in I:k>k_0\}\).
	In particular, \(B=4\kappa+4\) is independent of the dimensions and norms of the vector fibres.
\end{lemma}

\begin{proof}
	We induct on the rank of the ambient coordinate defining \(f\).
	A level-zero evaluation is one FDD coefficient functional.
	It meets at most one block \(x_k\), so \eqref{eq:gen-vec-bi} holds with \(g=0\).

	Let \(f\) have level \(h>0\), and set
	\[ l=\max\bigl(\{\min I\}\cup\{k\in I:j_k\leq h\}\bigr). \]
	Thus \(h<j_k\) for every \(k>l\).
	And whenever there is an index \(k<l\), we have \(h\geq j_{k+1}\).
	This convention also covers the cases \(h<j_{\min I}\) and \(h\geq j_{\max I}\).
	For \(k<l\), the second estimate in Lemma~\ref{lem:vec-ris-trn} and the strict increase of the \(j_k\)'s give
	\[ \sum_{k<l}|\lambda_k|\,|f(P_{(s,\infty)}x_k)| \leq \kappa C\sum_{k<l}\frac{|\lambda_k|}{m_{j_k}} \leq \kappa C\Big(\sum_{r\geq1}m_r^{-1}\Big) \max_{k<l}|\lambda_k|\leq \kappa C\max_{k<l}|\lambda_k|. \]

	For \(k=l\), the interval-projection bound gives
	\[ |\lambda_l|\,|f(P_{(s,\infty)}x_l)| \leq |\lambda_l|\,\|f\|\,\|P_{(s,\infty)}\|\,\|x_l\| \leq \kappa C|\lambda_l|. \]
	Choose
	\[ k_0\in I\cap(-\infty,l] \quad\text{with}\quad |\lambda_{k_0}|=\max_{k\in I,\,k\le l}|\lambda_k|. \]
	The early sum and the \(l\)-th term are then bounded separately by \(\kappa C|\lambda_{k_0}|\),
	so
	\begin{equation}\label{eq:gen-vec-bi-initial}
		\sum_{k\le l}|\lambda_k|\,|f(P_{(s,\infty)}x_k)|
		\leq 2\kappa C|\lambda_{k_0}|\leq BC|\lambda_{k_0}|.
	\end{equation}
	Notice that this choice also guarantees that every index used in the auxiliary tail below is strictly larger than \(k_0\).

	Put \(I'=\{k\in I:k>l\}\), and use the analysis \eqref{eq:ful-vec-ea}.
	Let \(I'_0\) be the set of those \(k\in I'\) whose range contains one of \(p_1,\ldots,p_a\).
	Then \(|I'_0|\leq a\), because the ranges of the \(x_k\)'s are disjoint and each \(p_r\) can belong to only one of them.
	After deleting these boundary blocks, every remaining tail interval
	\(I_k^{\rm tail}=\operatorname{ran}x_k\cap(s,\infty)\) meets at most one open cell
	\((p_{r-1},p_r)\) on which \(f\) can act.
	Put
	\[ I'_r=\{k\in I'\setminus I'_0:I_k^{\rm tail}\cap(p_{r-1},p_r)\ne\varnothing\}. \]
	After empty sets are omitted, the \(I'_r\)'s are successive intervals of indices.
	A block outside \(I'_0\cup\bigcup_r I'_r\) is disjoint from the support of \(f\) and contributes nothing.
	The blocks in \(I'_0\) are estimated as whole blocks by the first truncation estimate,
	while on \(I'_r\) all BD parts vanish and only the \(r\)-th packet can act.
	Therefore
	\begin{align*}
		\left|f\left(P_{(s,\infty)}
		\sum_{k\in I'}\lambda_kx_k\right)\right|  \leq \frac{C_{\rm tr}C}{m_h}\sum_{k\in I'_0}|\lambda_k|+\frac{1}{m_h}\sum_{r=1}^{a}
		\left|b_{r,\varphi}^{i*}
		P_{(s\vee p_{r-1},\infty)}
		\sum_{k\in I'_r}\lambda_kx_k\right|.
	\end{align*}
	Fix a nonempty \(I'_r\), and put
	\(z_r=P_{(s\vee p_{r-1},\infty)} \sum_{k\in I'_r}\lambda_kx_k.
	\)
	By the scalarization in the proof of Lemma~\ref{lem:uni-eva}, the packet has a finite expansion
	\[ b_{r,\varphi}^{i*} =\sum_{\ell\in L_r}c_\ell e_{\eta_\ell,\psi_\ell}^{i*}, \qquad c_\ell=a_\ell\norm{R_\ell\varphi_r}, \qquad \sum_{\ell\in L_r}|c_\ell|\leq1, \]
	where \(\psi_\ell\in B_{(G_{\eta_\ell}^i)^*}\) and
	\(p_{r-1}<\operatorname{rank}\eta_\ell<p_r\).
	If \(b_{r,\varphi}^{i*}=0\), the required packet estimate holds directly with
	\(k_r=\min I'_r\) and \(g_r=0\).
	Otherwise \(L_r\) is nonempty, and we choose \(\ell_r\in L_r\) such that
	\[ |e_{\eta_{\ell_r},\psi_{\ell_r}}^{i*}(z_r)| =\max_{\ell\in L_r}|e_{\eta_\ell,\psi_\ell}^{i*}(z_r)|, \qquad f_r=e_{\eta_{\ell_r},\psi_{\ell_r}}^{i*}\in\mathcal K_i. \]
	Then
	\[ |b_{r,\varphi}^{i*}(z_r)| \leq\sum_{\ell\in L_r}|c_\ell|\, |e_{\eta_\ell,\psi_\ell}^{i*}(z_r)| \leq |f_r(z_r)|. \]
	Since \(f=e_{\alpha,\varphi}^{i*}\), we have
	\(
		\operatorname{rank}\eta_{\ell_r}<p_r\leq p_a
		=\operatorname{rank}\alpha.
	\)
	Thus the ambient coordinate defining \(f_r\) has strictly smaller rank,
	and we may apply the inductive hypothesis on \(I'_r\) with cutoff \(s\vee p_{r-1}\).
	We obtain \(k_r\in I'_r\) and \(g_r\in W\cup\{0\}\) with
	\[ |b_{r,\varphi}^{i*}(z_r)|\leq |f_r(z_r)| \leq BC\left(|\lambda_{k_r}|+ g_r\left(\sum_{k\in I'_r}|\lambda_k|e_k\right)\right). \]
	Put \(\mathcal R=\{r\in\{1,\ldots,a\}:I'_r\ne\varnothing\}\).
	By the successivity of the \(I'_r\)'s and the inclusion
	\(\operatorname{supp}g_r\subseteq\{k\in I'_r:k>k_r\}\), we can list the functionals
	\(e_k^*\) for \(k\in I'_0\), the functionals \(e_{k_r}^*\) for \(r\in\mathcal R\),
	and the nonzero \(g_r\) for \(r\in\mathcal R\).
	The bounds \(|I'_0|\leq a\) and \(|\mathcal R|\leq a\) give
	\[ d\leq |I'_0|+2|\mathcal R| \leq a+2a \leq3n_h. \]
	Therefore
	\[ g=\frac{1}{m_h}\left( \sum_{k\in I'_0}e_k^*+ \sum_{r \in \mathcal{R}}(e_{k_r}^*+g_r)\right) \in W\cup\{0\} \]
	is supported after \(k_0\), and has root level \(h\) when it is nonzero.
	Let \(v=\sum_{k\in I}|\lambda_k|e_k\).
	The preceding estimates give
	\begin{align}\label{eq:gen-vec-bi-tail}
		\left|f\left(P_{(s,\infty)}\sum_{k\in I'}\lambda_k x_k\right)\right|\leq
		\frac{C_{\rm tr}C}{m_h}\sum_{k\in I'_0}|\lambda_k|
		+\frac{BC}{m_h}\sum_{r\in \mathcal{R}}
		\bigl(|\lambda_{k_r}|+g_r(v)\bigr) \leq BC g(v).
	\end{align}
	Adding \eqref{eq:gen-vec-bi-initial} and \eqref{eq:gen-vec-bi-tail}
	gives \eqref{eq:gen-vec-bi}.
	The induction uses only contractive restriction maps in the vector fibres.
	Hence the constant \(B\) does not depend on their dimensions.
\end{proof}

\subsection{Auxiliary and special averages}

Recall that \(A_1=1\) and \(A_h=4\max_{r<h}n_r\) for \(h\geq2\).

\begin{lemma}[Vanishing auxiliary averages]\label{lem:aux-avg}
	If
	\[ u_N=N^{-1}\sum_{k=1}^{N}e_k, \]
	then \(\|u_N\|_T\to 0\).
\end{lemma}

\begin{proof}
	Fix \(h_0,d\in\N_+\).
	Let \(g\in W\) be arbitrary and fix a tree analysis of \(g\).
	Every weighted node \(\alpha\) of level \(h\) has at most \(3n_h\) successors.
	For each \(q\in\{1,\ldots,N\}\cap\operatorname{supp}g\), the successivity of the successors gives
	a unique branch ending at a coordinate leaf \(\zeta_qe_q^*\), where \(|\zeta_q|=1\).
	If its weighted levels are \(h_1,\ldots,h_r\), then
	\[ |g(e_q)|=\prod_{\ell=1}^{r}\frac{1}{m_{h_\ell}} \qquad \left(\text{denote }\prod_{\ell=1}^{0}\frac{1}{m_{h_\ell}}=1\right). \]

	Stop each branch when it first reaches a coordinate leaf, a node of level greater than \(h_0\),
	or its \(d\)-th node of level at most \(h_0\).
	Let \(\mathcal L\), \(\mathcal H\), and \(\mathcal D\) be the three classes of indices
	\(q\in\{1,\ldots,N\}\cap\operatorname{supp}g\) determined by these stopping events.
	Since \(3n_h\leq A_{h_0+1}\) for \(h\leq h_0\), we have
	\[
		\begin{aligned}
			|\mathcal L|\leq\sum_{j=0}^{d-1}A_{h_0+1}^{j}\leq A_{h_0+1}^{d},\qquad
			\max\{|\mathcal H|,|\mathcal D|\}
			\leq\bigl|\{1,\ldots,N\}\cap\operatorname{supp}g\bigr|\leq N.
		\end{aligned}
	\]
	Moreover, we obtain
	\[
		|g(e_q)|\leq
		\begin{cases}
			1,&q\in \mathcal L,\\
			m_{h_0+1}^{-1},&q\in \mathcal H,\\
			m_1^{-d},&q\in \mathcal D.
		\end{cases}
	\]
	Hence the arbitrary \(g\in W\) satisfies
	\[ |g(u_N)| =\frac{1}{N}\left|\sum_{q=1}^N g(e_q)\right| \leq\frac{A_{h_0+1}^{d}}{N} +\frac{1}{m_{h_0+1}} +\frac{1}{m_1^{d}}. \]
	The right-hand side is independent of \(g\), so taking the supremum over \(W\) gives
	\[ \|u_N\|_T =\sup_{g\in W}|g(u_N)| \leq\frac{A_{h_0+1}^{d}}{N} +\frac{1}{m_{h_0+1}} +\frac{1}{m_1^{d}}. \]
	Therefore
	\[ \limsup_{N\to\infty}\|u_N\|_T \leq\frac{1}{m_{h_0+1}}+\frac{1}{m_1^d}. \]
	Since \(h_0\) and \(d\) are arbitrary, we know \(\|u_N\|_T\to 0\) as \(N \to \infty\).
\end{proof}
We next give a quantitative estimate for auxiliary special averages.

The stopping-event formulation of the proof below was prompted by a suggestion
from a generative-AI system.
Subsequent literature review showed that the underlying counting and
tree-truncation techniques are well established; see
\cite[Lemma~2.4 and Proposition~2.5]{ArgyrosHaydon2011},
\cite{ManoussakisPelczarSwietek2017},
and \cite[Proposition~11.7]{MotakisPelczar2025}.
Here we adapt these techniques to our parameters and explicitly record
the first stopping event on each branch.
For further details, see the \hyperref[sec:ai-assistance]{AI assistance statement}
at the end of the paper.

\begin{lemma}[Auxiliary special averages]
	\label{lem:qnt-aux-avg}
	Let
	\[ u_j=\frac{1}{n_j}\sum_{k=1}^{n_j}e_k. \]
	Then
	\begin{equation}\label{eq:qnt-aux-avg}
		\|u_j\|_T\leq \frac{2}{m_j}\qquad(j\geq1).
	\end{equation}
\end{lemma}

\begin{proof}
	First let \(j\geq2\).
	Let \(g\in W\) be arbitrary and fix a tree analysis of \(g\).
	As in the proof of Lemma~\ref{lem:aux-avg}, each
	\(q\in\{1,\ldots,n_j\}\cap\operatorname{supp}g\) determines a unique branch
	ending at a coordinate leaf \(\zeta_qe_q^*\), where \(|\zeta_q|=1\).

	Stop each branch when it first reaches a node of level \(h\geq j\),
	its \(D_j\)-th node of level \(h<j\), or a coordinate leaf.
	Let \(\mathcal H\), \(\mathcal D\), and \(\mathcal L\) be the corresponding classes of coordinate indices.
	Since every node of level \(h<j\) has at most \(A_j\) successors, put \(R_j^{\rm aux}=\sum_{s=0}^{D_j-1}A_j^s.\)
	Then
	\[ |\mathcal L|\leq R_j^{\rm aux}, \qquad \max\{|\mathcal H|,|\mathcal D|\} \leq\bigl|\{1,\ldots,n_j\}\cap\operatorname{supp}g\bigr| \leq n_j. \]
	The same branch-coefficient argument gives
	\[
		|g(e_q)|\leq
		\begin{cases}
			m_j^{-1},&q\in\mathcal H,\\
			m_1^{-D_j},&q\in\mathcal D,\\
			1,&q\in\mathcal L.
		\end{cases}
	\]
	Hence the arbitrary \(g\in W\) satisfies
	\[
		\begin{aligned}
			|g(u_j)| \leq\frac{|\mathcal L|}{n_j}
				+\frac{|\mathcal H|}{n_jm_j}
				+\frac{|\mathcal D|}{n_jm_1^{D_j}}\leq\frac{R_j^{\rm aux}}{n_j}
				+\frac{1}{m_j}
				+\frac{1}{m_1^{D_j}}.
		\end{aligned}
	\]
	The definition of \(D_j\), \eqref{eq:par-grw}, \eqref{eq:n-grw},
	and \eqref{eq:led-elm} give
	\[
		\begin{aligned}
			m_1^{-D_j}\leq2^{-D_j}\leq m_j^{-2},\quad
			R_j^{\rm aux}\leq\frac{L_j}{m_j},
			\quad n_j\geq2^{j+8}L_jm_j^4,\quad
			\frac{R_j^{\rm aux}}{n_j}
			\leq\frac{L_j/m_j}{2^{j+8}L_jm_j^4}
			=\frac{1}{2^{j+8}m_j^5}.
		\end{aligned}
	\]
	Therefore
	\[ |g(u_j)| \leq\frac{1}{m_j} \left(1+\frac{1}{m_j}+\frac{1}{2^{j+8}m_j^4}\right) <\frac{2}{m_j}, \]
	where the last inequality follows from \(m_j\geq m_1\geq2^8\).
	The right-hand side is independent of \(g\), so taking the supremum over \(W\) gives
	\[ \|u_j\|_T=\sup_{g\in W}|g(u_j)|\leq\frac{2}{m_j} \qquad(j\geq2). \]

	Now let \(j=1\) and let \(g\in W\) be arbitrary.
	A coordinate root and a weighted root give, respectively,
	\[
		|g(u_1)|\leq
		\begin{cases}
			n_1^{-1},&g=\zeta e_q^*,\\
			m_1^{-1},&g\text{ has a weighted root}.
		\end{cases}
	\]
	Since \(n_1\geq m_1\), taking the supremum over \(g\in W\) gives
	\[ \|u_1\|_T\leq m_1^{-1}\leq\frac{2}{m_1}. \]
\end{proof}

\begin{corollary}[Bounds for RIS special averages]
	\label{cor:ris-spc-avg}
	Let \(j\in\N_+\), and let
	\((x_k)_{k=1}^{n_j}\subseteq X_i ( i=0,1)\) be a \(C\)-RIS. Put
	\[ x=\frac{m_j}{n_j}\sum_{k=1}^{n_j}x_k. \]
	For every \(\theta\in\K\) with \(|\theta|\leq2\), we have
	\begin{equation}\label{eq:ris-avg-nrm}
		\|\theta x\|\leq3BC|\theta|\leq6BC.
	\end{equation}
	And for every \(\alpha\in\bigcup_{n\geq1}\Delta_n^i\),
	\begin{equation}\label{eq:ris-avg-dst}
		\|D_\alpha^i(\theta x)\|
		=\sup_{\varphi\in B_{(G_\alpha^i)^*}}
		|d_{\alpha,\varphi}^{i*}(\theta x)|
		\leq\frac{\kappa C|\theta|}{m_j}
		\leq\frac{2\kappa C}{m_j}.
	\end{equation}
\end{corollary}

\begin{proof}
	By homogeneity, it is enough to prove both estimates for \(\theta=1\).
	Apply the general basic inequality with \(s=0\) and \(\lambda_k=m_j/n_j\).
	For every scalarized ambient coordinate evaluation \(f\), it supplies \(k_0\) and \(g\in W\cup\{0\}\) such that
	\[ |f(x)| \leq BC\left(\frac{m_j}{n_j} +m_jg(u_j)\right) \leq BC\left(\frac{m_j}{n_j}+2\right). \]
	By \eqref{eq:led-elm}, \(n_j\geq m_j^2\), so the last expression is at most \(3BC\).
	The set \(\mathcal{K}_i\) defined in \eqref{eq:fml-ki} norms \(X_i\). Thus
	\[ \|x\|=\sup_{f\in\mathcal{K}_i}|f(x)|\leq3BC. \]
	For \(\alpha\in\bigcup_{n\geq1}\Delta_n^i\) and
	\(\varphi\in B_{(G_\alpha^i)^*}\),
	\(
		d_{\alpha,\varphi}^{i*}
		=P_{\{\operatorname{rank}\alpha\}}^{i*}e_{\alpha,\varphi}^{i*}
	\)
	gives \(\|d_{\alpha,\varphi}^{i*}\|\leq\kappa\).
	Since the \(x_k\)'s are successive, at most one of their FDD ranges contains
	\(\operatorname{rank}\alpha\). Therefore
	\[ |d_{\alpha,\varphi}^{i*}(x)| \leq\frac{m_j}{n_j}\,\kappa C \leq\frac{\kappa C}{m_j}. \]
	Taking the supremum over \(\varphi\) proves \eqref{eq:ris-avg-dst}.
\end{proof}

\begin{corollary}[RIS are weakly null]
	\label{cor:ris-wek-nul}
	Every RIS in either \(X_i\) is weakly null.
\end{corollary}

\begin{proof}
	Let \((x_k)\subseteq X_i\) be a \(C\)-RIS with RIS indices \((j_k)\).
	Suppose that \((x_k)\) is not weakly null.
	Then there are \(x^*\in X_i^*\), \(\varepsilon>0\), and a subsequence
	\((x_{k_r})\) such that
	\[ |x^*(x_{k_r})|\geq\varepsilon\qquad(r\in\N_+). \]
	Choose \(\theta_r\in\K\) with \(|\theta_r|=1\) such that \(x^*(\theta_rx_{k_r})=|x^*(x_{k_r})|\),
	and put \(y_r=\theta_rx_{k_r}\).
	Then \((y_r)\) is a \(C\)-RIS with RIS indices \((j_{k_r})\).
	For \(N\in\N_+\), put
	\[ z_N=\frac{1}{N}\sum_{r=1}^{N}y_r \qquad\text{and}\qquad u_N=\frac{1}{N}\sum_{r=1}^{N}e_r. \]
	By Lemma~\ref{lem:vec-bas-ine} with \(s=0\) and \(\lambda_r=1/N\),
	for every \(f\in\mathcal{K}_i\) there is \(g\in W\cup\{0\}\) such that
	\[ |f(z_N)| \leq BC\left(\frac{1}{N}+g(u_N)\right) \leq BC\left(\frac{1}{N}+\|u_N\|_T\right). \]
	By taking the supremum over \(f\in\mathcal{K}_i\) and using
	Lemma~\ref{lem:aux-avg}, we have
	\[ \|z_N\| \leq BC\left(\frac{1}{N}+\|u_N\|_T\right) \longrightarrow 0 \quad(N \to \infty). \]
	On the other hand, we have
	\[ \varepsilon \leq\frac{1}{N}\sum_{r=1}^{N}|x^*(x_{k_r})| =|x^*(z_N)| \leq\|x^*\|\,\|z_N\| \longrightarrow 0,  \quad(N \to \infty) \]
	which is a contradiction.
\end{proof}

\subsection{Shrinking and a compactness test}

The local-weight estimate below allows us to pass from RISs to arbitrary block sequences.
We then obtain shrinking of the FDDs and a compactness test.
For nonzero \(x\in X_i^{\rm alg}\), let \(q=\max\operatorname{ran}x\) and write \(x=i_q^i(u)\), where \(u=(u_\alpha)_{\alpha\in\bigcup_{r=1}^{q}\Delta_r^i}\in\bigoplus_{r\leq q}F_r^i\).
For \(\alpha\in\bigcup_{r=1}^{q}\Delta_r^i\) and
\(\psi\in B_{(G_\alpha^i)^*}\), the scalarized evaluation functional satisfies
\(e_{\alpha,\psi}^{i*}(x)=\psi(u_\alpha)\).
Define its local support by \(\supp_{\rm loc}x=\{\alpha\in\bigcup_{r=1}^{q}\Delta_r^i:u_\alpha\ne0\}\) and
\(\supp_{\rm loc}0=\varnothing\).

\begin{lemma}[Local-weight estimate]\label{lem:vec-loc-wgt}
	If \(x\in X_i^{\rm alg}\), \(f\in\mathcal{K}_i\) has level \(h>0\),
	and \(\operatorname{lev}(\alpha)\ne h\) for every
	\(\alpha\in\supp_{\rm loc}x\), then
	\[ |f(x)|\leq \kappa m_h^{-1}\|x\|. \]
\end{lemma}

\begin{proof}
	The assertion is immediate for \(x=0\).
	Assume that \(x\ne0\), and put
	\[ q=\max\operatorname{ran}x, \qquad x=i_q^i(u), \qquad u\in\bigoplus_{n\leq q}F_n^i. \]
	Write \(f=e_{\gamma,\varphi}^{i*}\), and let
	\[ \xi_1\prec\cdots\prec\xi_a=\gamma, \qquad p_0=\chi(\gamma), \qquad p_s=\operatorname{rank}\xi_s, \qquad e_0=0, \qquad e_s=e_{\xi_s,\varphi_s}^{i*}. \]
	By \eqref{eq:syn-vec-ea}, the BD evaluation analysis of \(f\) is
	\[ f=e_a =\sum_{s=1}^{a}d_{\xi_s,\varphi_s}^{i*} +\frac{1}{m_h}\sum_{s=1}^{a} b_{s,\varphi}^{i*}P_{(p_{s-1},\infty)}^i. \]
	Use \(t_s\) and \(r_w\) from \eqref{eq:rst-tel}, and put \(r=r_q\).
	By \hyperref[itm:ana-g7]{\textup{(G7)}} and the hypothesis, we have
	\[ \operatorname{lev}(\xi_s)=h \quad(1\leq s\leq a), \qquad \operatorname{lev}(\alpha)\ne h \quad(\alpha\in\supp_{\rm loc}x). \]
	If \(1\leq s\leq a\) and \(p_s\leq q\), then
	\(\xi_s\notin\supp_{\rm loc}x\), so \(u_{\xi_s}=0\).
	This gives
	\[ e_s(x)=\varphi_s(u_{\xi_s})=0. \]
	The convention \(e_0=0\) gives \(e_0(x)=0\).
	If \(r=a\), then \(f(x)=e_a(x)=0\).
	We may therefore assume that \(r<a\).
	Since \(P_{(q,\infty)}^ix=0\), \eqref{eq:rst-tel} gives
	\begin{align*}
		0 =\bigl(P_{(q,\infty)}^{i*}e_a\bigr)(x)=f(x)-e_r(x)
		+\bigl((P_{(q,\infty)}^{i*}-I)t_{r+1}\bigr)(x)=f(x)-t_{r+1}(x).
	\end{align*}
	By \hyperref[itm:ana-ea1]{\textup{(EA1)}}, \eqref{eq:fml-ext-bnd},
	and \eqref{eq:cst-led}, we have
	\begin{align*}
		|f(x)|
		=|t_{r+1}(x)|=\frac{1}{m_h}
		\left|b_{r+1,\varphi}^{i*}
		\bigl(P_{(p_r,\infty)}^ix\bigr)\right|\leq\frac{1+M}{m_h}\|x\|\leq\frac{\kappa}{m_h}\|x\|.
	\end{align*}
\end{proof}

The following is the vector-valued form of the RIS reduction used in the AH construction.

\begin{proposition}[RIS reduction and compactness test]
	\label{prop:ris-red}
	Let \(Y\) be a Banach space and \(R:X_i\to Y\) be bounded.
	If \(\|Rx_k\|\to 0\) for every RIS \((x_k)\), then
	\(\|Rx_k\|\to 0\) for every bounded block sequence.
	Thus the FDDs of \(X_0\) and \(X_1\) are shrinking.
	Under the same assumption, \(R\) is compact.
\end{proposition}

\begin{proof}
	Let \((x_k)\) be a bounded block sequence with \(\sup_k\norm{x_k}\le C_x\), and put \(s_k=\max\operatorname{ran}x_k\),
	and write \(x_k=i_{s_k}^i(u_k)\).
	For \(N\in\N_+\), write \(u_k=v_k^N+w_k^N\), where
	\(v_k^N\) is supported on atoms of level at most \(N\) and \(w_k^N\) is supported on atoms of level greater than \(N\).
	Put \(y_k^N=i_{s_k}^i(v_k^N)\) and \(z_k^N=i_{s_k}^i(w_k^N)\).
	By \hyperref[itm:ana-ea4]{\textup{(EA4)}} and the block extension estimate, we have
	\[ \|y_k^N\|,\ \|z_k^N\|\leq M\|x_k\|. \]
	Also, \(\supp_{\rm FDD}y_k^N\cup\supp_{\rm FDD}z_k^N\subseteq\operatorname{ran}x_k\).
	For fixed \(N\), Lemma~\ref{lem:vec-loc-wgt} shows that \((y_k^N)_k\) is a RIS.
	Indeed, choose RIS indices \(j_k\) increasing beyond the preceding block ranges and with \(j_k>N\).
	Let \(f=e_{\gamma,\varphi}^{i*}\in\mathcal{K}_i\) have level \(0<h<j_k\).
	If \(N<h<j_k\), the local support of \(y_k^N\) contains no level-\(h\) summand, and Lemma~\ref{lem:vec-loc-wgt} gives
	\[ |f(y_k^N)|\le\frac{\kappa\|y_k^N\|}{m_h}. \]
	If \(h\le N\), the trivial bound \(|f(y_k^N)|\le\|y_k^N\|\) is at most \(m_N\|y_k^N\|/m_h\).
	Thus the constant \( C_N=MC_x\max\{1,\kappa,m_N\} \)
	satisfies all RIS inequalities.
	Thus, for every fixed \(N\), the hypothesis on \(R\) gives \(\|Ry_k^N\|\to 0\). We may therefore choose recursively \(N_t,k_t\) so that
	\[ N_{t+1}>\max\{N_t,s_{k_t}\},\qquad k_{t+1}>k_t,\qquad \|Ry_{k_t}^{N_t}\|<2^{-t}.\]
	For the sequence \((z_{k_t}^{N_t})_t\), put \(j_t=N_t+1\). Then
	\[ j_{t+1}=N_{t+1}+1>s_{k_t}\geq\max\operatorname{ran}z_{k_t}^{N_t}, \qquad \operatorname{lev}(\alpha)\geq j_t \quad\bigl(\forall \alpha\in\supp_{\rm loc}z_{k_t}^{N_t}\bigr). \]
	By Lemma~\ref{lem:vec-loc-wgt}, we have
	\[ \left|e_{\gamma,\varphi}^{i*}\bigl(z_{k_t}^{N_t}\bigr)\right| \leq\frac{\kappa\|z_{k_t}^{N_t}\|}{m_h} \leq\frac{\kappa MC_x}{m_h} \]
	whenever \(0<\operatorname{lev}(\gamma)=h<j_t\).
	Thus \((z_{k_t}^{N_t})_t\) is a \(\kappa MC_x\)-RIS.
	Hence \(\|Rz_{k_t}^{N_t}\|\to 0\), and \(Rx_{k_t}=Ry_{k_t}^{N_t}+Rz_{k_t}^{N_t}\to 0\).
	Applying this argument to every subsequence, we have that \(Rx_k\to 0\) as \(k \to \infty\).

	By Corollary~\ref{cor:ris-wek-nul}, every scalar functional kills every RIS.
	The first part therefore says that every bounded block sequence is weakly null,
	which is equivalent to shrinking of an FDD.

	Finally, suppose that \(\|R(I_{X_i}-P_{[1,N]})\|\not\to 0\).
	Choose \(\delta_0>0\) and arbitrarily large \(N\) with \(\|R(I_{X_i}-P_{[1,N]})\|>\delta_0\).
	Since \(\|I_{X_i}-P_{[1,N]}\|\le1+M\le\kappa\), applying the tail projection to a unit vector and then normalizing its nonzero image gives a vector in the tail range on which \(R\) has norm at least \(\delta_0/(2\kappa)\).
	For each such \(N\), we can therefore choose
	\[ u\in\operatorname{ran}(I_{X_i}-P_{[1,N]}),\qquad \norm{u}\le1,\qquad \norm{Ru}>\frac{\delta_0}{2\kappa}. \]
	Starting with \(q_0=0\), make such a choice after \(q_{k-1}\), and then choose \(q_k>q_{k-1}\) so that
	\[ \widetilde{x}_k=P_{(q_{k-1},q_k]}u_k \quad\text{with}\quad \norm{R(u_k-\widetilde{x}_k)}<\frac{\delta_0}{4\kappa}. \]
	The interval projection bound gives \(\|\widetilde{x}_k\|\le\kappa\), the \(\widetilde{x}_k\)'s are successive, and \(\|R\widetilde{x}_k\|>(\delta_0/4\kappa)\).
	This contradicts the first part of the proposition.
	Therefore
	\[ \norm{R(I_{X_i}-P_{[1,N]})}\longrightarrow0. \]
	Thus \(R\) is compact.
\end{proof}

\begin{lemma}[Isometry of the protected module]
	\label{lem:tv-ess-low}
	For every \(v\in V^*\),
	\begin{equation}\label{eq:tv-ess-low}
		\norm{[T_v]}=\norm{v}.
	\end{equation}
\end{lemma}

\begin{proof}
	If \(V=\{0\}\), the assertion is immediate. Otherwise, fix \(q\in V_{\Q}\cap B_V\setminus\{0\}\).
	By \hyperref[prop:fml-prb]{Proposition~\ref*{prop:fml-prb}~\textup{(P1)}}, choose successively separated probes \(\gamma_k\) whose fibres contain \(q\).
	The uniform shielding estimate in \hyperref[prop:fml-prb]{Proposition~\ref*{prop:fml-prb}~\textup{(P2)}} makes \((j_{\gamma_k}q)\) a RIS,
	so it is weakly null by Corollary~\ref{cor:ris-wek-nul}.
	Hence every compact \(K:X_0\to X_1\) satisfies \(\norm{Kj_{\gamma_k}q}\to 0\).
	By the probe identities in \hyperref[itm:ana-g6]{\textup{(G6)}}, we have
	\(\phi_{\gamma_k}(T_vj_{\gamma_k}q)=v(q)\). Hence
	\begin{align*}
		\left|v(q)-\phi_{\gamma_k}(Kj_{\gamma_k}q)\right|
		=\left|\phi_{\gamma_k}\bigl((T_v-K)j_{\gamma_k}q\bigr)\right|
		\leq\norm{T_v-K}\norm{q}.
	\end{align*}
	Letting \(k\to\infty\) and taking the supremum over \(q\in V_{\Q}\cap B_V\) gives \(\norm{T_v-K}\ge\norm{v}\). Taking the infimum over compact \(K\) and using Lemma~\ref{lem:int} gives \eqref{eq:tv-ess-low}.
\end{proof}

\section{Exact pairs and dependent sequences}

The exact-pair and dependent-sequence scheme used in this section follows the AH construction \cite{ArgyrosHaydon2011}.
Rooted auxiliary-average estimates and off-weight (equivalently, off-level)
estimates are well-established tools in mixed Tsirelson and Argyros--Haydon
constructions.
For auxiliary averages, see \cite[Proposition~2.5]{ArgyrosHaydon2011}
and compare \cite[Proposition~11.7]{MotakisPelczar2025}.
For off-weight estimates, see
\cite[Definition~6.1 and the following remark]{ArgyrosHaydon2011}
and \cite[Definition~6.7]{Motakis2024}.
The broader exact-pair methods and their subsequent developments are treated in
\cite{ArgyrosTolias2004,Tarbard2013,Zisimopoulou2014,ManoussakisPelczarSwietek2017,MotakisPuglisiTolias2020}.
The proofs below give the versions needed for our parameters, interval
restrictions, and root perturbations.

This section turns a persistent failure of a  local orbit estimate into a dependent-sequence contradiction.
If an operator stays away from its orbit on a RIS,
finite-dimensional separation and rational approximation first give rational packet data with small annihilation error.
Operator-dependent exactification uses chains already present in the fixed datum to turn these data into an exact pair and a compatible terminal-coordinate certificate.
The finite-extension property then lets us select such pairs successively along a coded history.
On the resulting dependent sequence, the local operator lower bounds add up,
while the basic inequality gives a lower-order upper bound for the corresponding vector average.
This contradiction proves the required local orbit estimate.

Readers who wish to see the main dependent-sequence argument before the technical construction may go directly to Proposition~\ref{prop:crt-loc-orb}.

The annihilation error below treats the two self corners and the forward corner uniformly.

\subsection{\texorpdfstring{Annihilation and root perturbations}{Annihilation and root perturbations}}

Recall from Definition~\ref{def:fml-pkt} and \eqref{eq:pkt-val-sup}
that a represented operator packet \(\mathbf b\) determines an operator
\(B_{\mathbf b}^*:G^*\longrightarrow\mathcal E_{i,*}\), together with a raw rank support.
For each \(\varphi\in G^*\), \(B_{\mathbf b}^*\varphi\) is a scalar-valued ambient functional, acting on \(X_i\) via the canonical embedding.

As stipulated before \eqref{eq:sep-crt}, \(\norm{b}_{\ell_1}\) denotes the block \(\ell_1\)-sum norm. For a finite ambient functional
\(b=(b_n)_n\in\mathcal E_{i,*}\), with \(b_n=(\psi_\alpha)_{\alpha\in\Delta_n^i}\in(F_n^i)^*\), we have
\[ \norm{b}_{\ell_1} =\sum_n\norm{b_n}_{(F_n^i)^*} =\sum_n\sum_{\alpha\in\Delta_n^i}\norm{\psi_\alpha}_{(G_\alpha^i)^*}. \]
For a packet representation, entries at the same atom are combined when computing this norm.
In particular, a scalar packet \(\mathbf b=((a_l,\eta_l,R_l))_l\) satisfies
\(\norm{B_{\mathbf b}^*(1)}_{\ell_1}\le\sum_l|a_l|\norm{R_l}\).

For a rational finite ambient functional \(b\) with \(\norm{b}_{\ell_1}\le1\), a represented scalar packet \(\mathbf b\), with \(G=\K\), can be chosen so that \(b=B_{\mathbf b}^*(1)\).
Its entries may refer to atoms with vector fibres.
Whenever this fixed representation is admissible for the prescribed chain extension, Proposition~\ref{prop:fin-ext} guarantees cofinally many pre-existing realizations carrying it.

For simplicity of notation, we also use \(b\) to refer to its fixed represented scalar packet \(\mathbf b\).
In a general evaluation analysis, \(b\) denotes \(B_{\mathbf b}^*\varphi\) with the inherited scalarization \(\varphi\) which need not be rational.
In both cases, \(\supp_{\rm raw}b\) means \(\supp_{\rm raw}\mathbf b\) and depends on this stored representation.
Recall that \(U_h\) is defined by Equation \eqref{eq:uh}.

For a corner \(c=ij\in\{00,11,01\}\), define the \emph{annihilation error}
\begin{equation}\label{eq:ann-dfc}
	\Delta_c(h,x)=
	\begin{cases}
		h(x),&c=00\text{ or }c=11,\\
		U_hx,&c=01.
	\end{cases}
\end{equation}

\begin{lemma}[Error identity and continuity]
	\label{lem:dfc-cnt}
	For \(c=ij\in\{00,11,01\}\),
	\begin{equation}\label{eq:dfc-ann}
		h|_{\mathscr{M}_{ij}(x)}=0
		\quad\Longleftrightarrow\quad
		\Delta_c(h,x)=0.
	\end{equation}
	If \(h,k\) belong to a common finite ambient dual block, then
	\begin{equation}\label{eq:dfc-lip}
		\norm{\Delta_c(h,x)-\Delta_c(k,y)}
		\le \norm{h-k}_{\ell_1}\norm{x}
		+\norm{k}_{\ell_1}\norm{x-y}.
	\end{equation}
\end{lemma}

\begin{proof}
	Only the forward corner needs comment.
	Equation \eqref{eq:h-int} gives
	\[ h(T_vx)=v(U_hx)\qquad(v\in V^*). \]
	Since \(V^*\) separates points of \(V\), this proves \eqref{eq:dfc-ann}.
	Moreover \( \norm{U_hx} =\sup_{v\in B_{V^*}}|h(T_vx)| \le\norm{h}_{\ell_1}\norm{x}, \)
	which gives \eqref{eq:dfc-lip}. The other cases are its scalar specialization.
\end{proof}

\begin{lemma}[Rational approximate annihilator]
	\label{lem:rat-app-ann}
	Let \(c=ij\in\{00,11,01\}\), let \(S:X_i\to X_j\) be bounded, and let \((x_k)\) be a RIS such that
	\begin{equation}\label{eq:unf-orb-sep}
		\dist(Sx_k,\mathscr{M}_{ij}(x_k))\ge\varepsilon
		\qquad(k\ge1).
	\end{equation}
	For every cutoff \(a\), \(\delta>0\), and positive sequence \((\eta_k)\), there are
	\(k\), a rational block vector \(u\), and a rational packet \(b\), with \(u\) and \(b\) supported after \(a\), such that
	\begin{equation}\label{eq:app-ann-con}
		\begin{gathered}
			\ran u=\ran x_k,
			\quad \norm{u-x_k}<\eta_k,
			\quad \norm{b}_{\ell_1}\le1,\quad
			\operatorname{Re}b(Su)>\frac{\varepsilon}{2},
			\quad \norm{\Delta_c(b,u)}<\delta.
		\end{gathered}
	\end{equation}
\end{lemma}

\begin{proof}
	Since the FDDs are shrinking, \((Sx_k)\) is weakly null.
	We can choose \(k\) and \(r>\max\operatorname{ran}x_k\) such that
	\[ \min\operatorname{ran}x_k>a, \qquad \norm{P_{(0,a]}^jSx_k}+\norm{P_{(r,\infty)}^jSx_k}<\frac{\varepsilon}{8}. \]
	Put \(z=P_{(a,r]}^jSx_k\).
	By \eqref{eq:fou-orb} and \(P_{(a,r]}^1T_v=T_vP_{(a,r]}^0\), we have
	\[ \mathscr{M}_{ij}(x_k)\subseteq P_{(a,r]}^jX_j, \qquad \norm{Sx_k-z}<\frac{\varepsilon}{8}, \qquad \dist(z,\mathscr{M}_{ij}(x_k))>\frac{7\varepsilon}{8}. \]
	By the Hahn--Banach theorem, we choose \(x^*\in B_{X_j^*}\) such that
	\[ x^*|_{\mathscr{M}_{ij}(x_k)}=0, \qquad \operatorname{Re}x^*(z)>\frac{3\varepsilon}{4}. \]
	Since the scalarized ambient coordinates norm \(X_j\), the bipolar theorem \cite{Conway1990} gives finite ambient packets \((b_\alpha)\) with
	\[ \norm{b_\alpha}_{\ell_1}\le1, \qquad b_\alpha\xrightarrow{w^*}x^*. \]
	For \(i=j\), this gives \(b_\alpha(x_k)\to 0\).
	For \(ij=01\), by \eqref{eq:def-tv}, all \(U_{b_\alpha}x_k\) lie in the finite-dimensional space
	\(\operatorname{span}\{D_\gamma x_k:D_\gamma x_k\ne0\}\), and \eqref{eq:h-int} gives \(
	v(U_{b_\alpha}x_k)=b_\alpha(T_vx_k)\longrightarrow x^*(T_vx_k)=0\), for every \(v\in V^*\).
	Thus \(\norm{\Delta_c(b_\alpha,x_k)}\to 0\).
	Since \(z\) and \(\mathscr{M}_{ij}(x_k)\) are supported in \((a,r]\), deleting the raw ranks at most \(a\) preserves
	both \(b_\alpha(z)\) and \(\Delta_c(b_\alpha,x_k)\).
	Thus we may choose a finite packet \(b^0\) so that
	\[ \supp_{\rm raw}b^0\subseteq(a,\infty), \qquad \norm{b^0}_{\ell_1}\le1, \qquad \operatorname{Re}b^0(z)>\frac{5\varepsilon}{8}, \qquad \norm{\Delta_c(b^0,x_k)}<\frac{\delta}{4}. \]
	It follows that
	\[ \operatorname{Re}b^0(Sx_k) \ge \operatorname{Re}b^0(z)-\norm{Sx_k-z} >\frac{\varepsilon}{2}. \]
	Then there is \(\theta\) with \(0<\theta<1\) so that the packet \(b^1=(1-\theta)b^0\) satisfies
	\[ \norm{b^1}_{\ell_1}<1, \qquad \operatorname{Re}b^1(Sx_k)>\frac{\varepsilon}{2}, \qquad \norm{\Delta_c(b^1,x_k)}<\frac{\delta}{2}. \]
	By rational density in the same finite blocks and \eqref{eq:dfc-lip}, we choose rational \(u\) and \(b\) with
	\[
		\begin{gathered}
			\ran u=\ran x_k,
			\qquad \norm{u-x_k}<\eta_k,
			\qquad \supp_{\rm raw}b\subseteq(a,\infty),\\
			\norm{b}_{\ell_1}\le1,
			\qquad \operatorname{Re}b(Su)>\frac{\varepsilon}{2},
			\qquad \norm{\Delta_c(b,u)}<\delta.
		\end{gathered}
	\]
\end{proof}

\begin{corollary}[Simultaneous approximate-annihilator selection]
	\label{cor:sim-app-ann}
	Assume the hypotheses of Lemma~\ref{lem:rat-app-ann}.
	Let \(A:X_i\to F\) be bounded, where \(F\) is finite dimensional, and let
	\(\lambda_1,\ldots,\lambda_r\in X_i^*\).
	For every cutoff \(a\), positive \(\beta_0,\ldots,\beta_r,\tau\), and positive sequence \((\eta_k)\),
	there are \(k\), a rational block vector \(u\), and a rational packet \(b\) such that
	\[
		\begin{gathered}
			\ran u=\ran x_k,
			\quad \min\operatorname{ran}u>a,
			\quad \supp_{\rm raw}b\subseteq(a,\infty),
			\quad \norm{u-x_k}<\eta_k,\\
			\norm{Au}<\beta_0,
			\quad |\lambda_l(u)|<\beta_l\quad(1\le l\le r),
			\quad \norm{b}_{\ell_1}\le1,
			\quad \operatorname{Re}b(Su)>\frac{\varepsilon}{2},
			\quad \norm{\Delta_c(b,u)}<\tau.
		\end{gathered}
	\]
\end{corollary}

\begin{proof}
	By Corollary~\ref{cor:ris-wek-nul} and the finite dimensionality of \(F\),
	\[ \norm{Ax_k}\rightarrow0,\qquad \lambda_l(x_k)\rightarrow0\quad(1\le l\le r). \]
	Pass to a tail on which \(\norm{Ax_k}<\beta_0/2\) and
	\(|\lambda_l(x_k)|<\beta_l/2\) for \(1\le l\le r\), and put
	\[ \eta_k'=\min\left\{\eta_k, \frac{\beta_0}{2\max\{1,\norm A\}}, \min_{1\le l\le r}\frac{\beta_l}{2\max\{1,\norm{\lambda_l}\}}\right\}, \]
	where the last minimum is omitted when \(r=0\).
	By Lemma~\ref{lem:rat-app-ann}, choose \(k,u,b\) for this tail with cutoff \(a\), error bound \(\tau\), and radius \(\eta_k'\). Then
	\[ \norm{Au}\le\norm{Ax_k}+\norm A\eta_k'<\beta_0, \qquad |\lambda_l(u)|\le|\lambda_l(x_k)|+\norm{\lambda_l}\eta_k'<\beta_l \quad(1\le l\le r). \]
	The remaining assertions follow from \eqref{eq:app-ann-con}.
\end{proof}

\begin{definition}[Root-perturbed exact vector]
	\label{def:roo-exa}
	Fix a corner \(c=ij\in\{00,11,01\}\).
	Let \(u_1<\cdots<u_{n_p}\) belong to \(X_i^{\rm alg}\), and let
	\(\xi_1\prec\cdots\prec\xi_{n_p}=\eta\) be a complete level-\(p\) inner chain on side \(j\).
	Let \(\nu_0\) be its initial cutoff and put \(\nu_t=\operatorname{rank}\xi_t\) for \(1\le t\le n_p\).
	Assume
	\[ \nu_{t-1}<\min\operatorname{ran}u_t \le\max\operatorname{ran}u_t<\nu_t \qquad(1\le t\le n_p). \]
	Put \(f=e_\eta^{j*}\),
	and put
	\[ y^0=\frac{m_p}{n_p}\sum_{t=1}^{n_p}u_t. \]
	The root-perturbed vector certified by \(f\) is
	\begin{equation}\label{eq:roo-crr}
		y=
		\begin{cases}
			y^0-f(y^0)d_\eta^i,&i=j,\\
			y^0-d_{\widehat{\eta}}^0(U_fy^0),&ij=01.
		\end{cases}
	\end{equation}
\end{definition}

Throughout this section, \(y^0\) denotes the unperturbed vector and \(y\) the corresponding root-perturbed vector. The same convention applies to the sequence.

\begin{lemma}[Admissibility of the root perturbation]
	\label{lem:roo-crr-adm}
	Let \(C>0\), and let \(c,f,y^0,y\) be as in Definition~\ref{def:roo-exa}.
	If \(\norm{\Delta_c(f,y^0)}\le C/m_p\), then
	\[ \supp_{\rm FDD}(y-y^0)\subseteq\{\operatorname{rank}\eta\}, \qquad \norm{y-y^0}\le\frac{MC}{m_p}, \qquad \Delta_c(f,y)=0. \]
	The perturbation leaves all earlier \(d^*\)-coefficients and the inner-chain history unchanged.

	Suppose that the hypothesis holds termwise for a \(C_1\)-RIS \((y_p^0)_p\) with the level sequence \((p)\) as RIS indices,
	and let \(\eta_p\) be the corresponding terminal atom.
	For consecutive RIS indices \(p<p'\), assume \(\operatorname{rank}\eta_p<\min\bigl\{p',\min\operatorname{ran}y_{p'}^0\bigr\}\).
	Then the perturbed sequence \((y_p)_p\) is a \((C_1+MC)\)-RIS with RIS indices \((p)\).
\end{lemma}

\begin{proof}
	By Lemma~\ref{lem:fml-blk-bio} and the protected terminal identity, we have
	\[ e_\eta^{i*}(d_\eta^i)=1\quad(i=j), \qquad U_\eta d_{\widehat{\eta}}^0(q)=q\quad(ij=01). \]
	By \eqref{eq:fml-xi-di} and \eqref{eq:fml-ext-bnd}, for every atom \(\alpha\) on side \(i\) and \(g\in G_\alpha^i\),
	\[ \norm{d_\alpha^i(g)} =\norm{i_{\operatorname{rank}\alpha}^i(0,\ldots,g,\ldots,0)} \le M\norm{g}. \]
	Thus \eqref{eq:roo-crr} gives
	\[ \norm{y-y^0}\le M\norm{\Delta_c(f,y^0)}\le\frac{MC}{m_p}. \]
	The identities above and Lemma~\ref{lem:fml-blk-bio} give the support and annihilation assertions
	and show that all earlier \(d^*\)-coefficients are unchanged.
	Since the perturbation is supported at \(\operatorname{rank}\eta_p>\max\operatorname{ran}y_p^0\), we have \(y_p\ne0\) and \(\min\operatorname{ran}y_p=\min\operatorname{ran}y_p^0\).
	For consecutive RIS indices \(p<p'\),
	\[
		\begin{aligned}
			\max\operatorname{ran}y_p
			&\le\operatorname{rank}\eta_p<\min\operatorname{ran}y_{p'}^0
			 =\min\operatorname{ran}y_{p'},\\
			\max\operatorname{ran}y_p
			&\le\operatorname{rank}\eta_p<p'.
		\end{aligned}
	\]
	Since scalarized coordinates are contractive, for \(e_{\gamma,\varphi}^{i*}\in\mathcal K_i\) of level \(0<h<p\), we have
	\[ \bigl|e_{\gamma,\varphi}^{i*}(y_p)\bigr| \le\frac{C_1}{m_h}+\frac{MC}{m_p} \le\frac{C_1+MC}{m_h}, \qquad \norm{y_p}\le C_1+\frac{MC}{m_p}\le C_1+MC. \]
	By construction, the inner chain is fixed before the perturbation, so its history is unchanged.
\end{proof}

Recall that \(A_{\rm ex}=3B+M\) from \eqref{eq:cst-led}.

\begin{lemma}[Operator-dependent exactification]
	\label{lem:opr-exf}
	Let \(c=ij\in\{00,11,01\}\), let \(S\in\Bcal(X_i,X_j)\), and let \((x_k)\subseteq X_i\) be a \(C_0\)-RIS
	with RIS indices \((j_k)\), where \(C_0\ge1\). Suppose that, for some \(\varepsilon>0\),
	\[ \dist(Sx_k,\mathscr{M}_{ij}(x_k))\ge\varepsilon\qquad(k\ge1). \]
	Let \(q\in\mathcal{C}_c^{\rm seed}\), let \(\vartheta=(c,q,\chi,o)\) be a seed,
	and let \(\pi\) be a realized outer stem over \(\vartheta\) whose outer chain has age \(a<n_q\).
	Put \(p=\sigma_c(q,\pi)\) and \(C=2C_0\).
	For every cutoff \(R\in\N_+\), we can choose after \(R\) a complete level-\(p\) inner chain with terminal atom \(\eta\)
	and the corresponding rational root-perturbed vector \(y\). Put \(f=e_\eta^{j*}\). Then
	\begin{align}
		\norm{y}\le A_{\rm ex}C,&\qquad
		\sup_{\substack{\gamma\ \text{ an atom on side }i\\
		\varphi\in B_{(G_\gamma^i)^*}}}
		|d_{\gamma,\varphi}^{i*}(y)|
		\le \frac{A_{\rm ex}C}{m_p},
		\label{eq:roo-exa-est}\\
		\Delta_c(f,y)=0,&\qquad
		\operatorname{Re}f(Sy)\ge\frac{\varepsilon}{4}.
		\label{eq:roo-exa-sep}
	\end{align}
	Consequently, the construction  gives a rational certificate \(\mathfrak{c}\) compatible with \((\pi,\eta)\)
	and the pairs \((y,f)\) can be chosen successively.
\end{lemma}

\begin{proof}
	Recall from Definition~\ref{def:crt-stm} that \(\rho(\pi)\) is the endpoint rank of the realized outer stem \(\pi\), namely the seed birth rank when the stem is empty and the rank of its last outer atom otherwise.
	Fix \(R\in\N_+\). After passing to a tail and relabelling, we may assume for every \(k\) that
	\[ j_k\ge2, \qquad \min\operatorname{ran}x_k>\max\{R,\rho(\pi)\}. \]
	If \(u_k\) has the same FDD range as \(x_k\) and
	\(\|u_k-x_k\|\le C_0/m_{j_k-1}\), then, for \(0<h<j_k\), we have
	\[ \|u_k\|\le C_0+\frac{C_0}{m_{j_k-1}}\le C, \qquad |e_{\gamma,\varphi}^{i*}(u_k)| \le\frac{C_0}{m_h}+\frac{C_0}{m_{j_k-1}} \le\frac{C}{m_h}. \]
	The equality of the FDD ranges preserves the block order and
	\(j_{k+1}>\max\operatorname{ran}u_k\). Thus \((u_k)\) is a \(C\)-RIS.

	Choose
	\[ 0<\tau<\min\left\{ \frac{C}{m_p}, \frac{\varepsilon}{16M(1+\norm{S})} \right\}. \]
	Fix
	\[ \alpha_d=\frac{\varepsilon}{32m_pn_p},\qquad \alpha_\times=\frac{\varepsilon}{16n_p},\qquad \alpha_0=\frac{\varepsilon}{32},\qquad \delta_{\rm tail}=\frac{1}{\kappa}\min\{\alpha_d,\alpha_\times\}=\frac{\varepsilon}{32\kappa m_pn_p}. \]
	Put \(\nu_0=\rho(\pi)\), the initial cutoff stored by the attached inner chain.
	This is the lower endpoint used in its first type-zero correction.
	Then we recursively select actual occurrences from the fixed datum.
	Suppose that \((u_s,b_s,\xi_s,\nu_s)_{s<t}\), with \(\nu_s=\operatorname{rank}\xi_s\) have been selected.
	Put \(a_t=\max\{R,\nu_{t-1}\}\). Apply Corollary~\ref{cor:sim-app-ann}
	with the operator \(P_{[1,\nu_{t-1}]}^jS\), the functionals
	\(d_{\xi_s}^{j*}S\) and \(b_sP_{(\nu_{s-1},\infty)}^jS\) for \(s<t\),
	and the respective bounds \(\alpha_0,\alpha_d,\alpha_\times\).
	Use error bound \(\tau\) and \(\eta_k=C_0/m_{j_k-1}\). This gives \(u_t\) and a rational packet \(b_t\) such that
	\[
		\begin{gathered}
			|d_{\xi_s}^{j*}(Su_t)|<\alpha_d \quad (s<t), \qquad
			|b_sP_{(\nu_{s-1},\infty)}^j(Su_t)|<\alpha_\times\quad(s<t),\qquad \norm{P_{[1,\nu_{t-1}]}^jSu_t}<\alpha_0 \\
			\norm{b_t}_{\ell_1}\le1,\qquad \supp_{\rm raw}b_t\subseteq(a_t,\infty),\\
			\norm{\Delta_c(b_t,u_t)}<\tau,\qquad
			\operatorname{Re}b_t(Su_t)>\frac{\varepsilon}{2}.
		\end{gathered}
	\]
	Since the FDD tails of \(Su_1,\ldots,Su_t\) tend to zero, we can choose \(r_t\) so that
	\begin{equation}\label{eq:old-out-tal}
		\begin{gathered}
			r_t>\max\bigl\{\max\operatorname{supp}_{\rm raw}b_t,
			\max\operatorname{ran}u_t\bigr\},\quad
			\max_{1\le s\le t}\norm{P_{(r_t,\infty)}^jSu_s}<\delta_{\rm tail}.
		\end{gathered}
	\end{equation}
	Choose one of the cofinally pre-existing realizations guaranteed by Proposition~\ref{prop:fin-ext} as the next occurrence \(\xi_t\) of the prescribed level-\(p\) inner chain after rank \(r_t\), carrying the chosen representation of \(b_t\).
	When \(t<n_p\), choose \(u_{t+1}\) only after \(\xi_t\) has been selected.
	The fixed datum already contains this occurrence, and \(S\) is used only to select it.

	The terminal coordinate \(f=e_\eta^{j*}\) has the unique analysis
	\begin{equation}\label{eq:inn-exa-ana}
		f=\sum_{s=1}^{n_p}d_{\xi_s}^{j*}
		+\frac{1}{m_p}\sum_{s=1}^{n_p}
		b_sP_{(\nu_{s-1},\infty)}^j.
	\end{equation}
	The recursive order gives
	\[
		\begin{aligned}
			\supp_{\rm FDD}u_t\subseteq(\nu_{t-1},\nu_t),\quad
			\max\operatorname{supp}_{\rm raw}b_s<\min\operatorname{ran}u_t
			(s<t),\quad
			\max\operatorname{ran}u_t<\nu_{s-1}
			(s>t).
		\end{aligned}
	\]
	By Lemma~\ref{lem:fml-blk-bio}, the BD part and every off-diagonal packet term vanish on the corresponding \(u_t\).
	For \(c=01\), the same cancellations, \eqref{eq:h-int}, and \(P_I^1T_v=T_vP_I^0\) give, for every \(v\in B_{V^*}\),
	\[ v(U_fy^0)=f(T_vy^0) =\frac{1}{n_p}\sum_{t=1}^{n_p}b_t(T_vu_t) =\frac{1}{n_p}\sum_{t=1}^{n_p}v(U_{b_t}u_t). \]
	Taking the supremum over \(v\in B_{V^*}\) gives the same estimate as in the self corners.
	This gives
	\[ \norm{\Delta_c(f,y^0)} \le\frac{1}{n_p}\sum_{t=1}^{n_p}\norm{\Delta_c(b_t,u_t)}<\tau. \]

	By Lemma~\ref{lem:fml-blk-bio}, the selection inequalities, and \eqref{eq:old-out-tal}, we have
	\begin{align*}
		\left|\frac{m_p}{n_p}\sum_{s,t=1}^{n_p}d_{\xi_s}^{j*}(Su_t)\right|
		&\le\frac{m_p}{n_p}\left(
			\sum_{s<t}|d_{\xi_s}^{j*}(Su_t)|
			+\sum_{s\ge t}|d_{\xi_s}^{j*}(P_{(r_t,\infty)}^jSu_t)|\right)\\
		&<\frac{m_p}{n_p}\left(\sum_{s<t}\alpha_d+\sum_{s\ge t}\kappa\delta_{\rm tail}\right)
		\le m_pn_p\alpha_d=\frac{\varepsilon}{32}.
	\end{align*}
	Similarly, since \(\nu_{s-1}\ge\nu_t>r_t\) for \(s>t\), we have
	\begin{align*}
		\left|\frac{1}{n_p}\sum_{s\ne t}b_sP_{(\nu_{s-1},\infty)}^j(Su_t)\right|
		&\le\frac{1}{n_p}\left(
			\sum_{s<t}|b_sP_{(\nu_{s-1},\infty)}^j(Su_t)|
			+\sum_{s>t}|b_sP_{(\nu_{s-1},\infty)}^jP_{(r_t,\infty)}^j(Su_t)|\right)\\
		&<\frac{1}{n_p}\left(\sum_{s<t}\alpha_\times+\sum_{s>t}\kappa\delta_{\rm tail}\right)
		\le(n_p-1)\alpha_\times<\frac{\varepsilon}{16}.
	\end{align*}
	Finally, \(P_{(\nu_{t-1},\infty)}^j-I_{X_j}=-P_{[1,\nu_{t-1}]}^j\), so
	\begin{align*}
		\left|\frac{1}{n_p}\sum_t b_t(P_{(\nu_{t-1},\infty)}^j-I_{X_j})Su_t\right|
		&=\left|\frac{1}{n_p}\sum_t b_tP_{[1,\nu_{t-1}]}^jSu_t\right|\\
		&\le\frac{1}{n_p}\sum_t\norm{P_{[1,\nu_{t-1}]}^jSu_t}<\alpha_0=\frac{\varepsilon}{32}.
	\end{align*}
	By \eqref{eq:inn-exa-ana}, the definition of \(y^0\), and these three estimates, we have
	\begin{align*}
		&\left|f(Sy^0)-\frac{1}{n_p}\sum_{t=1}^{n_p}b_t(Su_t)\right|\\
		=&\left|\frac{m_p}{n_p}\sum_{s,t=1}^{n_p}d_{\xi_s}^{j*}(Su_t)
		+\frac{1}{n_p}\sum_{s\ne t}b_sP_{(\nu_{s-1},\infty)}^j(Su_t)
		+\frac{1}{n_p}\sum_{t=1}^{n_p}b_t\bigl(P_{(\nu_{t-1},\infty)}^j-I_{X_j}\bigr)Su_t\right|\\
		\le&\left|\frac{m_p}{n_p}\sum_{s,t=1}^{n_p}d_{\xi_s}^{j*}(Su_t)\right|
		+\left|\frac{1}{n_p}\sum_{s\ne t}b_sP_{(\nu_{s-1},\infty)}^j(Su_t)\right|+\left|\frac{1}{n_p}\sum_{t=1}^{n_p}b_t\bigl(P_{(\nu_{t-1},\infty)}^j-I_{X_j}\bigr)Su_t\right|\\
		<&\frac{\varepsilon}{32}+\frac{\varepsilon}{16}+\frac{\varepsilon}{32}
		=\frac{\varepsilon}{8}.
	\end{align*}
	It follows that
	\[ \operatorname{Re}f(Sy^0) >\frac{1}{n_p}\sum_{t=1}^{n_p}\operatorname{Re}b_t(Su_t)-\frac{\varepsilon}{8} >\frac{\varepsilon}{2}-\frac{\varepsilon}{8} =\frac{3\varepsilon}{8}. \]

	Lemma~\ref{lem:roo-crr-adm} gives
	\(\Delta_c(f,y)=0\) and \(\norm{y-y^0}<M\tau\). Since \(\norm{f}\le1\), we have
	\[ |f(S(y-y^0))| \le\norm{f}\norm{S}\norm{y-y^0} <M\norm{S}\tau <\frac{\varepsilon}{16}. \]
	It follows that
	\[ \operatorname{Re}f(Sy) >\frac{3\varepsilon}{8}-\frac{\varepsilon}{16} =\frac{5\varepsilon}{16} >\frac{\varepsilon}{4}. \]
	This proves \eqref{eq:roo-exa-sep}. By Corollary~\ref{cor:ris-spc-avg}, Lemma~\ref{lem:roo-crr-adm}, and \(A_{\rm ex}=3B+M\), we have
	\begin{align*}
		\norm{y}
		\le3BC+\frac{MC}{m_p}
		\le A_{\rm ex}C,\quad
		\sup_{\substack{\gamma\ \text{ an atom on side }i\\
		\varphi\in B_{(G_\gamma^i)^*}}}
		|d_{\gamma,\varphi}^{i*}(y)|
		\le\frac{\kappa C}{m_p}+\frac{C}{m_p}
		\le\frac{A_{\rm ex}C}{m_p}.
	\end{align*}
	This proves \eqref{eq:roo-exa-est}. By \eqref{eq:inn-exa-ana}, the cell order,
	and \(P_J^1T_v=T_vP_J^0\), where \(J=(\rho(\pi),\operatorname{rank}\eta]\), we have
	\[ P_J^{j*}f=f, \qquad \supp_{\rm FDD}y\cup\supp_{\rm FDD}\mathscr{M}_{ij}(y)\subseteq J, \qquad f(Sy)=f(P_J^jSy). \]
	Choose a rational finite vector \(z\) such that
	\[ \supp_{\rm FDD}z\subseteq J, \qquad \norm{z-P_J^jSy}<\frac{\varepsilon}{8}. \]
	Then \(\operatorname{Re}f(z)>\varepsilon/8\).
	Choose \(\varepsilon_{\rm cert}\in\Q_+\) with \(\varepsilon_{\rm cert}<\varepsilon/8\),
	and choose a finite output interval \(I\subseteq(\rho(\pi),\operatorname{rank}\eta]\) such that
	\[ \supp_{\rm FDD}y\cup\supp_{\rm FDD}z \cup\supp_{\rm FDD}\mathscr{M}_{ij}(y)\subseteq I, \qquad \operatorname{rank}\eta\in I. \]
	Since \(f=e_\eta^{j*}\) is the ambient coordinate indexed by \(\eta\),
	the last condition also gives \(f\in\mathcal{E}_{j,*}(I)\).
	By \(\Delta_c(f,y)=0\) and Lemma~\ref{lem:dfc-cnt}, \(f|_{\mathscr{M}_{ij}(y)}=0\).
	Hence
	\[ \mathfrak{c}=(i,j,y,z,f,I,\mathscr{M}_{ij}(y),\varepsilon_{\rm cert}) \]
	is a finite rational separation certificate in the sense of \eqref{eq:sep-crt} for the next outer stem.
	It is worth noting that the operator \(S\)  is not part of its code.
	In the forward case \(\Delta_c(f,y^0)=U_fy^0\) already belongs to the fixed terminal fibre \(E_\eta\),
	and it belongs to its rational structure because every packet, restriction map, and vector used above is rational.
	Moreover, rationality of \(y\) makes the finite matrix of \(v\mapsto T_vy\) rational in the protected block coordinates.
	Hence \(\mathscr{M}_{01}(y)\) is a rational finite-dimensional output subspace.
	In a self corner, the orbit is \(\K y\), so it is rational as well.
	Repeating the construction with each new cutoff \(R\) larger than the preceding output interval gives successive pairs \((y,f)\).
\end{proof}

\begin{remark}[Quantifier pattern and the fixed datum]
	\label{rem:opr-exf-qnt}
A specific feature of Lemma~\ref{lem:opr-exf} is the uniform,
certificate-compatible order of choices described below.
It uses a single oracle-relative program which is uniform.

	Let \(V^*\) be a dual space with separable predual \(V\).
	By Remark~\ref{rem:orc-rec}, there is a single program \(\mathfrak M\), independent of \(V^*\) and \(V\).
	Let \(\mathcal D(V)\) be the fixed datum produced by \(\mathfrak M^{\mathcal O_{V}}\).
	For fixed \(c,q,\pi\) in \(\mathcal D(V)\), the choices in the lemma have the order
	\[
		\begin{aligned}
			\exists\mathfrak M \forall V \forall V^* \Bigl((V)^*=V^*\Rightarrow
			\exists\mathcal D(V) \bigl(\forall S \forall R
			\exists(u_1,b_1)\exists\xi_1\in\mathcal D(V)  \cdots
			\exists(u_{n_p},b_{n_p}) \exists\xi_{n_p}\in\mathcal D(V)\bigr)\Bigr),
		\end{aligned}
	\]
	and they satisfy the conclusions of Lemma~\ref{lem:opr-exf}.
	At step \(t\), the rational pair \((u_t,b_t)\) is chosen first.
	We then choose one of the cofinally pre-existing realizations guaranteed by Proposition~\ref{prop:fin-ext} as \(\xi_t\).
	The next pair is chosen only after \(\xi_t\) has been selected.
	Thus \(S\) selects occurrences but never extends the datum.
	The certificate code contains no \(S\), and the last paragraph of the proof shows that the forward orbit \(\mathscr M_{01}(y)\) is rational.
\end{remark}

The next lemma follows directly from Lemma~\ref{lem:opr-exf} and Proposition~\ref{prop:ris-red}.

\begin{lemma}[Weakly null exact-pair sequence]
	\label{lem:exa-par-res}
	Let \(c=ij\in\{00,11,01\}\), let \(S\in\Bcal(X_i,X_j)\), and let \((x_k)\subseteq X_i\) be a \(C_0\)-RIS,
	where \(C_0\ge1\), with RIS indices \((j_k)\), such that, for some \(\varepsilon>0\),
	\[ \dist(Sx_k,\mathscr{M}_{ij}(x_k))\ge\varepsilon\qquad(k\ge1). \]
	Let \(q\in\mathcal{C}_c^{\rm seed}\), let \(\vartheta=(c,q,\chi,o)\) be a seed,
	and let \(\pi\) be a realized outer stem over \(\vartheta\) whose outer chain has age \(a<n_q\).
	Put \(p=\sigma_c(q,\pi)\) and \(C=2C_0\).
	After every cutoff there is a sequence of pairs \((y_r,f_r)_{r}\) satisfying Lemma~\ref{lem:opr-exf},
	with their complete inner chains and the supports of the \(y_r\)'s successive.
	\((y_r)\) is uniformly bounded by \(A_{\rm ex}C\),
	and both \((y_r)\) and \((Sy_r)\) are weakly null.
\end{lemma}

\subsection{Collisions and dependent sequences}

Dependent sequences play the same role as in the AH method.
They produce a small-norm average while the image under the operator has large norm \cite{ArgyrosHaydon2011}.
Here an outer history contains many coded inner levels.
The estimates below control collisions between these levels.

\subsubsection{Quantitative collision estimate}

For an outer level \(q\), recall \(S_q\) and \(\Lambda_q\) from \eqref{eq:par-led-def}.
Since \(L_q\ge m_qS_q\) and \(m_q\ge2^8\), \eqref{eq:qnt-cod-grw} implies that every inner level
\(k=\sigma_c(q,\pi)\) in a level-\(q\) history of corner \(c\) satisfies
\begin{equation}\label{eq:cod-lam}
	m_k\ge\Lambda_q.
\end{equation}

By the root-level convention for the auxiliary norming set,
the following estimate applies when the root level \(h\) differs from the averaging level \(k\).

\begin{lemma}[Rooted auxiliary averages]
	\label{lem:rtd-aux-avg}
	Let \(h,k\in\N_+\) with \(h\ne k\), let \(g\in W\) have root level \(h\), and put \(u_k=n_k^{-1}\sum_{s=1}^{n_k}e_s\).
	Then
	\[
		|g(u_k)|\le
		\begin{cases}
			2/(m_hm_k),&h<k,\\
			1/m_h,&h>k.
		\end{cases}
	\]
	The same estimates hold after translating the support of \(u_k\).
\end{lemma}

\begin{proof}
	Suppose first that \(h<k\). Write
	\[ g=\frac{1}{m_h}\sum_{r=1}^df_r,\qquad f_1<\cdots<f_d,\qquad d\leq3n_h\leq A_k. \]
	Apply the stopping argument from the proof of Lemma~\ref{lem:qnt-aux-avg}
	to each subtree rooted at \(f_r\), counting the \(D_k\) low-level nodes within that subtree.
	Let \(\mathcal H,\mathcal D,\mathcal L\) be the unions of the corresponding classes of coordinate indices in \(\{1,\ldots,n_k\}\).
	Since the supports of the \(f_r\)'s are disjoint and every node of level less than \(k\) has at most \(A_k\) successors,
	\[ \max\{|\mathcal H|,|\mathcal D|\}\leq n_k, \qquad |\mathcal L|\leq d\sum_{s=0}^{D_k-1}A_k^s \leq\sum_{s=1}^{D_k}A_k^s\leq\frac{L_k}{m_k}. \]
	Retaining the root factor \(m_h^{-1}\), the branch-coefficient estimates give
	\begin{align*}
		|g(u_k)|
		&\leq\frac{1}{m_h}\left(
			\frac{|\mathcal H|}{n_km_k}
			+\frac{|\mathcal D|}{n_km_1^{D_k}}
			+\frac{|\mathcal L|}{n_k}\right)\\
		&\leq\frac{1}{m_h}\left(
			\frac{1}{m_k}+\frac{1}{m_1^{D_k}}+\frac{L_k}{m_kn_k}\right)\\
		&\leq\frac{1}{m_h}\left(
			\frac{1}{m_k}+\frac{1}{m_k^2}+\frac{1}{2^{k+8}m_k^5}\right)\\
		&=\frac{1}{m_hm_k}\left(1+\frac{1}{m_k}+\frac{1}{2^{k+8}m_k^4}\right)
		<\frac{2}{m_hm_k}.
	\end{align*}
	Suppose now that \(h>k\).
	The coordinate coefficients of the \(f_r\)'s have modulus at most one, and their supports are disjoint. Hence
	\[ |g(u_k)| \le\frac{1}{m_hn_k} \sum_{r=1}^d|\operatorname{supp}f_r\cap\{1,\ldots,n_k\}| \le\frac{1}{m_h}. \]
	The same calculation applies to every translate of \(u_k\).
\end{proof}

\begin{lemma}[Off-weight estimate for root-perturbed exact vectors]
	\label{lem:off-wgt-exa}
	Let \(C\ge1\), \(k\in\N_+\), and
	\[ y=y^0+r,\qquad y^0=\frac{m_k}{n_k}\sum_{s=1}^{n_k}u_s, \]
	where \((u_s)_{s=1}^{n_k}\) is a \(C\)-RIS and \(\norm{r}\le C/m_k\).
	If \(f\) is a scalarized coordinate evaluation of level \(h\in\N_+\) with \(h\ne k\), then,
	for every FDD interval \(I\), we have
	\begin{equation}\label{eq:off-wgt-exa}
		|f(P_Iy)|\le \frac{C_{\rm off}C}{m_{h \wedge k}}.
	\end{equation}
\end{lemma}

\begin{proof}
	Write \(I=(a,b]\), with the evident modifications for an initial or final interval.
	Since \(P_I=P_{(a,\infty)}-P_{(b,\infty)}\), apply
	Lemma~\ref{lem:vec-bas-ine} to both tails,
	with every coefficient equal to \(m_k/n_k\).
	Consequently
	\[ |f(P_Iy^0)| \le2BC\left(\frac{m_k}{n_k} +m_k\sup\{|g(u_k)|:g\in W \text{ has root level }h\}\right). \]
	The estimate is also valid when one of the tails is absent.
	If \(h<k\), Lemma~\ref{lem:rtd-aux-avg} gives
	\[ |f(P_Iy^0)|\le2BC\left(\frac{m_k}{n_k}+\frac{2}{m_h}\right)
	\le2BC\left(m_k^{-1}+2m_h^{-1}\right)\le\frac{6BC}{m_h}. \]
	If \(h>k\), the same lemma gives
	\[ |f(P_Iy^0)|\le2BC\left(\frac{m_k}{n_k}+\frac{m_k}{m_h}\right)
	\le2BC\left(m_k^{-1}+m_k^{-4}\right)\le\frac{4BC}{m_k}. \]
	Finally, \(\|P_I\|\le\kappa\), so we have
	\[ |f(P_Ir)|\le\kappa\norm{r}\le\kappa C/m_k. \]
	If \(h<k\), then \(m_k^{-1}\le m_h^{-1}\).
	If \(h>k\), we already have the desired inequality.
	Since \(B,\kappa\ge1\), in both cases \(6B+\kappa\le8B\kappa=C_{\rm off}\).
	This proves \eqref{eq:off-wgt-exa}.
\end{proof}

Recall from Requirement~\hyperref[sec:req-cat]{\textup{(R4)}} that an outer continuation of a stem \(\pi\) in corner \(c=ab\)
uses the one-entry strong packet \(\langle\eta\rangle\), where \(\eta\) is the terminal atom of the complete inner chain attached to \(\pi\), and requires a compatible certificate.
In the notation of Definition~\ref{def:fml-pkt},
\[ \langle\eta\rangle=((1,\eta,\operatorname{id}_{\K})), \qquad B_{\langle\eta\rangle}^*(1)=\epsilon_\eta^{b*}(1). \]
Writing \(f=e_\eta^{b*}\), the associated outer packet functional is \(fP_{(\rho(\pi),\infty)}^b\).

\begin{definition}[Dependent sequence]
	\label{def:zer-dep}
	Let \(c=ab\in\{00,11,01\}\), let \(q\in\mathcal{C}_c^{\rm seed}\), and let \(C\ge1\).
	A \(C\)-bounded \(q\)-dependent sequence in \(X_a\) is a successive family of root-perturbed exact pairs \((y_i,f_i)_{i=1}^{n_q}\) satisfying the following requirements.
	\begin{enumerate}
		\item \emph{History and placement.}
		There are a seed \(\vartheta=(c,q,\chi,o)\) and successive realized outer stems \((\pi_i)_{i=0}^{n_q}\) over \(\vartheta\).
		Put
		\[ \rho_i=\rho(\pi_i)\quad(0\le i\le n_q), \qquad k_i=\sigma_c(q,\pi_{i-1})\quad(1\le i\le n_q). \]
		For \(1\le i\le n_q\), the pair \((y_i,f_i)\) has inner level \(k_i\) and lies in the corresponding cell, so
		\[ \rho_{i-1}<\min\operatorname{ran}y_i \le\max\operatorname{ran}y_i<\rho_i, \]
		and the matching outer packet is \(f_iP_{(\rho_{i-1},\infty)}^b\).

		\item \emph{Exact-vector structure.}
		For \(1\le i\le n_q\),
		\[ y_i=y_i^0+r_i, \qquad y_i^0=\frac{m_{k_i}}{n_{k_i}}\sum_{t=1}^{n_{k_i}}u_{i,t}, \qquad \norm{r_i}\le\frac{C}{m_{k_i}}, \]
		where \(r_i\) is the root perturbation from Definition~\ref{def:roo-exa} and \((u_{i,t})_{t=1}^{n_{k_i}}\) is a \(C\)-RIS.

		\item \emph{Uniform bounds.}
		For \(1\le i\le n_q\),
		\[
			\norm{y_i}\le C,
			\qquad
			\sup_{\substack{\gamma\ \text{ an atom on side }a\\
			\varphi\in B_{(G_\gamma^a)^*}}}
			|d_{\gamma,\varphi}^{a*}(y_i)|
			\le \frac{C}{m_{k_i}}.
		\]

		\item \emph{Annihilation.}
		\[ \Delta_c(f_i,y_i)=0\qquad(1\le i\le n_q). \]
	\end{enumerate}
\end{definition}

Figure~\ref{fig:inner-outer-averages} separates the two averaging scales for the corner \(c=ab\).
Each inner average is perturbed before the outer average is formed.
The inner level \(k_s\) is determined by the preceding stem; the outer level \(q\) is fixed.

\begin{figure}[htbp]
	\centering
\begin{tikzpicture}[x=1cm,y=1cm,>=stealth,
	every node/.style={font=\small},
	innerbox/.style={draw=blue!55!black,rounded corners=2pt,
		fill=white,align=center,minimum height=1.20cm,inner sep=4pt},
	outerbox/.style={draw=green!40!black,rounded corners=2pt,
		fill=white,align=center,minimum height=1.05cm,inner sep=4pt}]
	\draw[draw=blue!55!black,fill=blue!4,rounded corners=3pt]
		(0,2.70) rectangle (12.20,5.10);
	\node at (6.10,4.82)
		{Inner level \(k_s=\sigma_c(q,\pi_{s-1})\)};

	\node[innerbox,text width=2.20cm] (ris) at (1.50,3.55)
		{RIS block\\[2pt]
		 \(u_{s,1},\ldots,u_{s,n_{k_s}}\)};
	\node[innerbox,text width=3.35cm] (average) at (5.50,3.55)
		{Unperturbed vector\\[2pt]
		 \(y_s^0=\frac{m_{k_s}}{n_{k_s}}
			 \sum_{t=1}^{n_{k_s}}u_{s,t}\)};
	\node[innerbox,text width=3.35cm] (pair) at (10.15,3.55)
		{\(y_s=y_s^0+r_s\), \quad \(f_s=e_{\eta_s}^{b*}\)\\[3pt]
		 \(\eta_s\): inner-chain terminal\\[2pt]
		 exact pair \((y_s,f_s)\)};
	\draw[->] (ris.east) -- (average.west);
	\draw[->] (average.east) -- (pair.west);
	\node[font=\scriptsize] at (3.38,4.45) {inner average};
	\node[font=\scriptsize] at (7.82,4.45) {root perturbation};

	\draw[->,black!65] (6.10,2.70) -- (6.10,1.80);
	\node[anchor=west,font=\scriptsize,inner sep=3pt] at (6.18,2.25)
		{repeat for \(s=1,\ldots,n_q\)};

	\draw[draw=green!40!black,fill=green!4,rounded corners=3pt]
		(0,-0.30) rectangle (12.20,1.80);
	\node at (6.10,1.48)
		{Outer level \(q\): successive pairs with coded inner levels \(k_s\)};
	\node[outerbox,text width=6.05cm] (pairs) at (3.50,0.48)
		{\((y_1,f_1),\ \ldots,\ (y_s,f_s),\ \ldots,\
			 (y_{n_q},f_{n_q})\)\\[3pt]
		 \(n_q\) pairs in successive outer cells};
	\node[outerbox,text width=3.30cm] (outeraverage) at (10.15,0.48)
		{Outer average\\[2pt]
		 \(\bar y=\frac{1}{n_q}\sum_{s=1}^{n_q}y_s\)};
	\draw[->] (pairs.east) -- node[above,font=\scriptsize]
		{average} (outeraverage.west);
\end{tikzpicture}
	\caption{Inner and outer averages.}
	\label{fig:inner-outer-averages}
\end{figure}
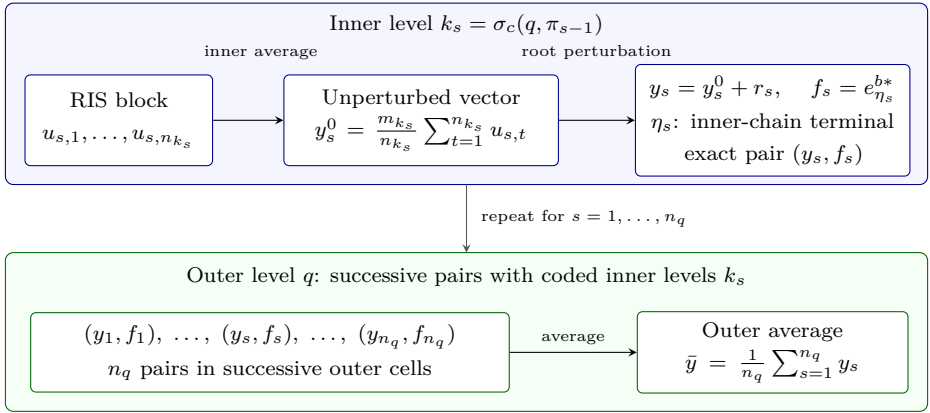

The stems \((\pi_i)_{i=0}^{n_q}\) have strictly increasing lengths, so they are distinct.
By injectivity of \(\sigma_c\), the levels \((k_i)_{i=1}^{n_q}\) are pairwise distinct.

Every \(C\)-bounded dependent sequence is a \(C_{\rm off}C\)-RIS with RIS indices \((k_i)\).
Indeed the next realized stem contains the outer predecessor rank \(\rho_i\), so \eqref{eq:qnt-cod-grw} gives
\[ k_{i+1}>\rho_i>\max\operatorname{ran}y_i\qquad(1\le i<n_q). \]
If \(0<\operatorname{lev}(\gamma)=h<k_i\) and
\(\varphi\in B_{(G_\gamma^a)^*}\), then Lemma~\ref{lem:off-wgt-exa} gives
\[ |e_{\gamma,\varphi}^{a*}(y_i)|\le\frac{C_{\rm off}C}{m_h}. \]

The common-initial-segment comparison in \cite[Lemma~6.5]{ArgyrosHaydon2011}
is used below with explicit interval restrictions to count the surviving
same-level incidences.

\begin{lemma}[Stem collision]
	\label{lem:crt-stm-col}
	Let \(c=ab\in\{00,11,01\}\), let \(q\in\mathcal{C}_c^{\rm seed}\), let \(C\ge1\),
	and let \((y_i,f_i)_{i=1}^{n_q}\) be a \(C\)-bounded \(q\)-dependent sequence in \(X_a\) for the corner \(c\).
	Let \(Q_0\) be a scalarized outer coordinate evaluation on side \(a\),
	of the same corner \(c\) and outer level \(q\).
	Let \(J_0\) be an FDD interval and put \(I=\bigl\{i\in\{1,\ldots,n_q\}:\operatorname{ran}y_i\subseteq J_0\bigr\}.\)
	Assume that \(I\ne\varnothing\).
	Then
	\begin{equation}\label{eq:q-col-bnd}
		\left|Q_0P_{J_0}^a\left(\frac{1}{n_q}\sum_{i\in I}y_i\right)\right|
		\le
		\frac{3\kappa C}{n_qm_q}
		+\frac{\kappa C}{\Lambda_q}
		+\frac{3C_{\rm off}C}{m_q\Lambda_q}.
	\end{equation}
\end{lemma}

\begin{proof}
	If \(Q_0=e_{\alpha,\varphi}^{a*}\), with \(\varphi\in B_{(G_\alpha^a)^*}\),
	recall from \eqref{eq:nsl-his}--\eqref{eq:two-his-fld} that its outer coding history
	\(\mathcal H=\mathsf{thist}(\alpha)\). And when \(\alpha=\widehat\gamma\) is a protected mirror, we have \(\mathcal H=\mathsf{shist}(\gamma)\).
	Let \(d\le n_q\) be the length of this history, equivalently the outer age of \(\alpha\).
	For \(1\le s\le d\), let \(K_s\) be its \(s\)-th packet cell, let \(\tau_{s-1}\) be the realized stem before this cell, and let \(\eta_s\) be the terminal atom of its strong packet.
	For \(c=00,11\), let \(\lambda=\varphi(1)\).
	Put
	\[
		h_s=\operatorname{lev}(\eta_s)=\sigma_c(q,\tau_{s-1}),
		\qquad
		F_s=
		\begin{cases}
			\lambda e_{\eta_s}^{a*},&c=00,11,\\
			e_{\widehat{\eta_s},\psi_s}^{0*},&c=01,
		\end{cases}
		\qquad
		|\lambda|\le1,\quad \psi_s\in B_{E_{\eta_s}^*}.
	\]
	Thus \(\|F_s\|\le1\).
	By Lemma~\ref{lem:uni-eva} and Remark~\ref{rem:int-rst}, the restricted evaluation analysis has the form
	\[ Q_0P_{J_0}^a =\sum_{s=1}^d d_{\alpha_s,\varphi_s}^{a*}P_{J_0}^a +\frac{1}{m_q}\sum_{s=1}^d F_sP_{K_s\cap J_0}^a. \]
	Put
	\[ \bar y_I=\frac{1}{n_q}\sum_{i\in I}y_i, \qquad L_{s,i}=K_s\cap\operatorname{ran}y_i, \]
	and
	\(
		\mathcal{B}
		=\bigl\{(s,i)\in\{1,\ldots,d\}\times I: d_{\alpha_s,\varphi_s}^{a*}(y_i)\ne0\bigr\}\),
		\[ \mathcal{A}=\bigl\{(s,i)\in\{1,\ldots,d\}\times I: F_s(P_{L_{s,i}}^ay_i)\ne0\bigr\},
		\qquad
		\mathcal{A}_{=}
		=\bigl\{(s,i)\in\mathcal{A} :  h_s=k_i\bigr\}.
	\]
	Since \(P_{J_0}^ay_i=y_i\) for \(i\in I\), we have
	\begin{align*}
		&|Q_0P_{J_0}^a(\bar y_I)|
		=\frac{1}{n_q}\left|\sum_{i\in I}Q_0(y_i)\right|\\
		\le&\frac{1}{n_q}
		\sum_{(s,i)\in\mathcal{B}} |d_{\alpha_s,\varphi_s}^{a*}(y_i)|\quad+\frac{1}{n_qm_q}\sum_{(s,i)\in\mathcal{A}_{=}} |F_s(P_{L_{s,i}}^ay_i)|+\frac{1}{n_qm_q}
		\sum_{(s,i)\in\mathcal{A}\setminus\mathcal{A}_{=}} |F_s(P_{L_{s,i}}^ay_i)|.
	\end{align*}
	Since \((\operatorname{ran}y_i)_{i\in I}\) are successive,
	\[ \#\bigl\{i\in I : \operatorname{rank}\alpha_s\in\operatorname{ran}y_i\bigr\}\le1 \quad(1\le s\le d), \qquad |\mathcal{B}|\le d\le n_q. \]
	By Definition~\ref{def:zer-dep} and \eqref{eq:cod-lam}, we have
	\[ |d_{\alpha_s,\varphi_s}^{a*}(y_i)| \le\frac{C}{m_{k_i}} \le\frac{C}{\Lambda_q} \le\frac{\kappa C}{\Lambda_q} \quad((s,i)\in\mathcal{B}). \]
	Hence
	\[ \frac{1}{n_q} \sum_{(s,i)\in\mathcal{B}} |d_{\alpha_s,\varphi_s}^{a*}(y_i)| \le\frac{\kappa C}{\Lambda_q}. \]

	Let \(\mathcal{D}\) be the history determined by \((\pi_i)_{i=0}^{n_q}\) in Definition~\ref{def:zer-dep}.
	Suppose first that the seed records of \(\mathcal{H}\) and \(\mathcal{D}\) agree, and let \(\tau\) be their longest common realized stem.
	By Lemma~\ref{lem:com-stm}, the labelled packets agree along \(\tau\). For every pair \((s,i)\) with \(h_s=k_i\) whose two packets occur on \(\tau\) and for which \(K_s\cap J_0=K_s\), the placement in Definition~\ref{def:zer-dep} and \(\Delta_c(f_i,y_i)=0\) give
	\[
		P_{L_{s,i}}^ay_i=y_i,
		\qquad
		F_s(P_{L_{s,i}}^ay_i)=
		\begin{cases}
			\lambda f_i(y_i),&c=00,11,\\
			\psi_s(U_{f_i}y_i),&c=01,
		\end{cases}
		=0.
	\]
	After \(\tau\), Lemma~\ref{lem:com-stm} allows only the two packets immediately following \(\tau\) to have the same level, and these form at most one pair \((s,i)\).
	If the seed records differ, Lemma~\ref{lem:com-stm} gives \(\mathcal{A}_{=}=\varnothing\).
	In both cases, we have
	\[ \#\bigl\{(s,i)\in\mathcal{A}_{=} :  K_s\cap J_0=K_s\bigr\}\le1. \]
	By Remark~\ref{rem:int-rst} and the pairwise distinctness of \((k_i)\),
	\[ \#\bigl\{s : \varnothing\ne K_s\cap J_0\ne K_s\bigr\}\le2, \qquad \#\bigl\{i\in I :  h_s=k_i\bigr\}\le1. \]
	Therefore
	\[
		\begin{aligned}
			|\mathcal{A}_{=}|
			&\le
			\#\bigl\{(s,i)\in\mathcal{A}_{=} :  K_s\cap J_0=K_s\bigr\}+
			\sum_{\substack{1\le s\le d\\
			\varnothing\ne K_s\cap J_0\ne K_s}}
			\#\bigl\{i\in I : (s,i)\in\mathcal{A}_{=}\bigr\} \le1+2=3.
		\end{aligned}
	\]
	Since \(\|P_{L_{s,i}}^a\|\le\kappa\) and \(\|y_i\|\le C\), we have
	\[ |F_s(P_{L_{s,i}}^ay_i)| \le\|F_s\|\|P_{L_{s,i}}^a\|\|y_i\| \le\kappa C \quad((s,i)\in\mathcal{A}_{=}). \]
	Hence
	\[ \frac{1}{n_qm_q} \sum_{(s,i)\in\mathcal{A}_{=}} |F_s(P_{L_{s,i}}^ay_i)| \le\frac{3\kappa C}{n_qm_q}. \]
	By \eqref{eq:cod-lam} and Lemma~\ref{lem:off-wgt-exa},
	\[ |F_s(P_{L_{s,i}}^ay_i)| \le\frac{C_{\rm off}C}{m_{\min\{h_s,k_i\}}} \le\frac{C_{\rm off}C}{\Lambda_q} \qquad((s,i)\in\mathcal{A}\setminus\mathcal{A}_{=}). \]
	The two families \((K_s)_{s=1}^d\) and \((\operatorname{ran}y_i)_{i\in I}\) are successive.
	Hence
	\[ |\mathcal{A}\setminus\mathcal{A}_{=}| \le\#\bigl\{(s,i)\in\{1,\ldots,d\}\times I :  K_s\cap\operatorname{ran}y_i\ne\varnothing\bigr\} \le d+|I|-1 \le2n_q-1<3n_q. \]
	Hence
	\[ \frac{1}{n_qm_q} \sum_{(s,i)\in\mathcal{A}\setminus\mathcal{A}_{=}} |F_s(P_{L_{s,i}}^ay_i)| \le\frac{3C_{\rm off}C}{m_q\Lambda_q}. \]
	Thus \eqref{eq:q-col-bnd} follows.
\end{proof}

The preceding lemma handles a terminal at outer level \(q\).
For a coordinate with lower root level, we stop its evaluation tree at level \(q\) and keep the carrier intervals.

A generative-AI system supplied the carrier-labelled stopping-tree mechanism
and the associated counting proof used in
Lemmas~\ref{lem:car-dom-tre}--\ref{lem:pai-bi}.
We rewrote the proof in our notation, verified it in the setting of our
construction, and checked its technical antecedents in the literature.
This review confirmed that the underlying tree and counting techniques are
well established; the relevant sources are cited below.
For further details, see the \hyperref[sec:ai-assistance]{AI assistance statement}
at the end of the paper.

For trees, branches and asymptotic games, see Odell and Schlumprecht~\cite{OdellSchlumprecht2002}.
Interval-labelled tree analyses of Bourgain--Delbaen evaluations appear in \cite[Section~2]{ManoussakisPelczarSwietek2017}.
For weight-threshold stopping and depth truncation in BD evaluations, see \cite[Section~3.2]{ManoussakisPelczarSwietek2017}.
For a related decomposition of auxiliary averages by node weights and branch lengths, see \cite[Proposition~11.7]{MotakisPelczar2025}.
Here the tree is selected for the fixed vector \(w\).
Each branch stops at zero contribution or at the first occurrence of a BD part, a coordinate of level zero or at least \(q\), or the depth cutoff \(D_q+2\).
Retaining the carrier intervals gives disjoint carriers at incomparable nodes, so the terminal contributions can be grouped by these stopping events.
For related estimates on repeated averages using tree analyses and a refined basic inequality, compare \cite[Propositions~11.8 and~13.6, Corollary~13.7]{MotakisPelczar2025}.

\begin{lemma}[Carrier-labelled tree]
	\label{lem:car-dom-tre}
	Let \(F\) be a scalarized coordinate evaluation, let \(J_\varnothing\) be an FDD interval, let \(w\) have finite FDD support, and fix \(q\in\N_+\).
	There are a finite rooted tree \(\mathcal{T}\), nonempty carrier intervals \(J_t\subseteq J_\varnothing\), coefficients \(0\le a_t\le1\), and terminal functionals \(\Phi_t\), where \(t\in\mathcal{T}_{\rm term}\), such that
	\begin{equation}\label{eq:car-dom}
		|F(P_{J_\varnothing}w)|
		\le \sum_{t\in\mathcal{T}_{\rm term}}
		a_t|\Phi_t(P_{J_t}w)|.
	\end{equation}
	Carriers of incomparable nodes are successive and disjoint.
	\begin{equation}\label{eq:car-tre-siz}
		|\mathcal{T}|\le1+A_q+\cdots+A_q^{D_q+2}=S_q.
	\end{equation}
	If \((y_i)\) is a finite successive block sequence, at most \(2|\mathcal{T}_{\rm term}|\) block ranges contain an endpoint of a terminal carrier.
	After these blocks are removed, if a remaining block range meets \(J_t\) for \(t\in\mathcal{T}_{\rm term}\), then it is contained in \(J_t\).
\end{lemma}

\begin{proof}
	By \hyperref[itm:ana-ea1]{\textup{(EA1)}}, every positive-level coordinate satisfies \eqref{eq:syn-vec-ea}.
	Every packet in the analysis of \(e_{\gamma,\varphi}^*\) is a finite absolutely convex combination of scalarized coordinates \(e_{\delta,\psi}^*\).
	Each constituent satisfies \(\operatorname{rank}\delta<\operatorname{rank}\gamma\).
	Hence the same formula can be applied recursively.
	Since \(w\) has finite FDD support, choose a finite nonempty FDD interval
	\(J\subseteq J_\varnothing\) such that \(P_Jw=P_{J_\varnothing}w\), and replace \(J_\varnothing\) by \(J\).
	Give the root carrier \(J_\varnothing\), initial coefficient \(a_\varnothing=1\), and functional \(\Phi_\varnothing=F\).
	At every node \(t\), including the root, if \(a_t|\Phi_t(P_{J_t}w)|=0\), set \(a_t=0\) and stop.
	Otherwise, stop each branch when it first reaches a BD part or a coordinate of level zero or at least \(q\).
	If neither occurs, stop the branch at depth \(D_q+2\).
	Every stopped functional is retained as a terminal functional.

	Let \(t\) be a coordinate node at which the recursion has not stopped.
	Write \(\Phi_t=e_{\gamma_t,\varphi_t}^*\) and put \(\ell_t=\operatorname{lev}(\gamma_t)\).
	Then \(0<\ell_t<q\).
	Let \(d_t\le n_{\ell_t}\) be the age of its evaluation analysis.
	For \(1\le r\le d_t\), let \(p_{t,r}=\operatorname{rank}\xi_{t,r}\) and let \(K_{t,r}\) be the corresponding packet cell.
	Put
	\[ J_{t,r}^{\rm BD}=J_t\cap\{p_{t,r}\}, \qquad J_{t,r}^{\rm pkt}=J_t\cap K_{t,r}. \]
	By \eqref{eq:syn-vec-ea},
	\[ \Phi_tP_{J_t} =\sum_{r=1}^{d_t}d_{\xi_{t,r},\varphi_{t,r}}^*P_{J_{t,r}^{\rm BD}} +\frac{1}{m_{\ell_t}}\sum_{r=1}^{d_t}b_{t,r}P_{J_{t,r}^{\rm pkt}}. \]
	If the represented packet for \(b_{t,r}\) is empty, its term is zero and gives no child.
	For every other nonempty packet carrier, write
	\( b_{t,r}=\sum_{l=1}^{s_{t,r}}c_{t,r,l}f_{t,r,l}\), where
	\( \sum_{l=1}^{s_{t,r}}|c_{t,r,l}|\le1\).
	For such a carrier, choose \(l(t,r)\) such that
	\(
		|f_{t,r,l(t,r)}(P_{J_{t,r}^{\rm pkt}}w)|
		=\max_{1\le l\le s_{t,r}}|f_{t,r,l}(P_{J_{t,r}^{\rm pkt}}w)|.
	\)
	Hence
	\begin{equation}\label{eq:num-pkt-dom}
		\begin{aligned}
			|b_{t,r}(P_{J_{t,r}^{\rm pkt}}w)|
			&\le\sum_{l=1}^{s_{t,r}}|c_{t,r,l}|
			|f_{t,r,l}(P_{J_{t,r}^{\rm pkt}}w)|\\
			&\le\left(\sum_{l=1}^{s_{t,r}}|c_{t,r,l}|\right)
			|f_{t,r,l(t,r)}(P_{J_{t,r}^{\rm pkt}}w)|
			\le |f_{t,r,l(t,r)}(P_{J_{t,r}^{\rm pkt}}w)|.
		\end{aligned}
	\end{equation}
	For every nonempty BD carrier and every packet carrier retained above, initialize the corresponding child \(u\) by
	\[
		(\Phi_u,J_u,a_u)=
		\begin{cases}
			(d_{\xi_{t,r},\varphi_{t,r}}^*,J_{t,r}^{\rm BD},a_t),
			&u\text{ is a BD child},\\
			(f_{t,r,l(t,r)},J_{t,r}^{\rm pkt},a_t/m_{\ell_t}),
			&u\text{ is a packet child}.
		\end{cases}
	\]
	Apply the same stopping rules to every child.
	The evaluation analysis and \eqref{eq:num-pkt-dom} give
	\[ a_t|\Phi_t(P_{J_t}w)| \le\sum_{u\in\operatorname{succ}(t)}a_u|\Phi_u(P_{J_u}w)|. \]
	Iterating this inequality from the root to the terminal nodes gives \eqref{eq:car-dom}.

	The packet cells and the singleton BD carriers are successive and disjoint.
	Hence, for incomparable nodes \(u\) and \(v\), we have
	\(J_u\cap J_v=\varnothing\).
	Moreover, for every nonterminal coordinate node \(t\),
	\(
		|\operatorname{succ}(t)|
		\le2d_t
		\le2n_{\ell_t}
		\le A_q.
	\)
	Thus \(\mathcal{T}\) has depth at most \(D_q+2\), and
	\[ |\mathcal{T}| \le\sum_{r=0}^{D_q+2}A_q^r =S_q, \]
	which is \eqref{eq:car-tre-siz}.

	For the last assertion, write the successive block sequence as \((y_i)_{i=1}^N\) and put
	\(
		I_{\rm end}
		=\left\{i:\operatorname{ran}y_i\cap\{\min J_t,\max J_t\}\ne\varnothing
		\text{ for some }t\in\mathcal{T}_{\rm term}\right\}.
	\)
	Since the block ranges are successive,
	\[ |I_{\rm end}| \le\sum_{t\in\mathcal{T}_{\rm term}} \sum_{\rho\in\{\min J_t,\max J_t\}} \#\{i:\rho\in\operatorname{ran}y_i\} \le2|\mathcal{T}_{\rm term}|. \]
	For \(t\in\mathcal{T}_{\rm term}\), put
	\(
		I_t=\{i\notin I_{\rm end}:\operatorname{ran}y_i\cap J_t\ne\varnothing\}.
	\)
	By the interval property, for \(i\in I_t\) we have
	\(
		\operatorname{ran}y_i\subseteq J_t.
	\)
	If \(\max J_t<\min J_u\) and \(i\in I_t\cap I_u\), then
	\(
		\max J_t,\ \min J_u\in\operatorname{ran}y_i,
	\)
	which contradicts \(i\notin I_{\rm end}\).
	Therefore
	\(
		I_t\cap I_u=\varnothing
	\), if \(t \ne u\).
\end{proof}

Figure~\ref{fig:stopping-tree} summarizes the stopping rules, with representative children.
At each expansion, one constituent is selected from each retained packet for the fixed vector \(w\).
Packet carriers are intersections with the parent carrier. The root has depth zero.
Zero-contribution nodes are stopped with coefficient zero.

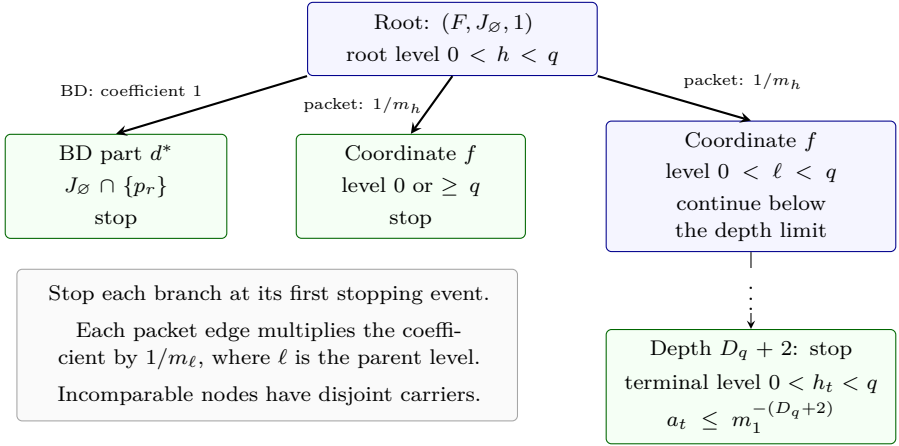
\begin{figure}[htbp]
	\centering
\begin{tikzpicture}[x=1cm,y=1cm,>=stealth,
	every node/.style={font=\small,align=center},
	coord/.style={draw=blue!55!black,fill=blue!4,rounded corners=2pt,
		minimum height=0.95cm,inner sep=4pt},
	terminal/.style={draw=green!40!black,fill=green!4,rounded corners=2pt,
		minimum height=1cm,inner sep=4pt}]
	\node[coord,text width=3.55cm] (root) at (6.05,5.70)
		{Root: \((F,J_\varnothing,1)\)\\[2pt]
		 root level \(0<h<q\)};
	\node[terminal,text width=2.65cm] (bd) at (1.60,3.75)
		{BD part \(d^*\)\\[2pt]
		 \(J_\varnothing\cap\{p_r\}\)\\[2pt]
		 stop};
	\node[terminal,text width=2.75cm] (level) at (5.50,3.75)
		{Coordinate \(f\)\\[2pt]
		 level \(0\) or \(\ge q\)\\[2pt]
		 stop};
	\node[coord,text width=3.55cm] (low) at (10.00,3.75)
		{Coordinate \(f\)\\[2pt]
		 level \(0<\ell<q\)\\[2pt]
		 continue below the depth limit};
	\draw[->,thick] (root.south west) --
		node[above left,font=\scriptsize] {BD: coefficient \(1\)} (bd.north);
	\draw[->,thick] (root.south) --
		node[left,font=\scriptsize] {packet: \(1/m_h\)} (level.north);
	\draw[->,thick] (root.south east) --
		node[above right,font=\scriptsize] {packet: \(1/m_h\)} (low.north);

	\node[terminal,text width=3.55cm] (deep) at (10.00,1.10)
		{Depth \(D_q+2\): stop\\[3pt]
		 terminal level \(0<h_t<q\)\\[2pt]
		 \(a_t\le m_1^{-(D_q+2)}\)};
	\draw[->] (low.south) -- node[midway,fill=white,inner sep=1pt] {\(\vdots\)} (deep.north);

	\node[draw=black!35,fill=black!2,rounded corners=2pt,
		text width=6.20cm,inner sep=6pt] at (3.60,1.65)
		{Stop each branch at its first stopping event.\\[4pt]
		 Each packet edge multiplies the coefficient by
		 \(1/m_{\ell}\), where \(\ell\) is the parent level.\\[4pt]
		 Incomparable nodes have disjoint carriers.};
\end{tikzpicture}
	\caption{A carrier-labelled stopping tree.}
	\label{fig:stopping-tree}
\end{figure}

The separation of boundary blocks from blocks contained in successive
intervals already appears in Gowers and Maurey~\cite{GowersMaurey1997}.
Summing the depth remainder over disjoint intervals also appears in
\cite[Section~3.2]{ManoussakisPelczarSwietek2017}.
The proof below applies these estimates to the terminal carriers of the stopped tree.
In the notation of the proof, the bad set consists of terminal--block incidences \((t,i)\):
each carrier contributes at most two boundary incidences, giving
\mbox{\(|\mathcal E_\partial|\le2|\mathcal T_{\rm term}|\le2S_q\)}.
For the contained blocks, disjointness gives \(\sum_t|I_t|\le n_q\).
Thus the boundary estimate includes \(S_q\), while the depth contribution
is bounded by \(Cm_1^{-(D_q+2)}\) without an additional factor \(S_q\).

\begin{lemma}[Paired mean inequality]\label{lem:pai-bi}
	Let \(c=ab\in\{00,11,01\}\).
	Then every \(C\)-bounded \(q\)-dependent sequence in \(X_a\) satisfies
	\begin{equation}\label{eq:zer-dep}
		\norm{\frac{1}{n_q}\sum_{i=1}^{n_q}y_i}
		\le \frac{K_{\rm dep}C}{m_q^2}.
	\end{equation}
\end{lemma}

\begin{proof}
	Put \(\bar y=n_q^{-1}\sum_{i=1}^{n_q}y_i\), let \(F=e_{\gamma,\varphi}^{a*}\in\mathcal{K}_a\), and put \(h=\operatorname{lev}(\gamma)\).
	Choose an FDD interval \(J_\varnothing\) containing every \(\operatorname{ran}y_i\), so \(P_{J_\varnothing}^a\bar y=\bar y\).

	We first consider the root level of \(F\).
	If \(h=0\), then \hyperref[itm:ana-ea1]{\textup{(EA1)}} gives
	\(F=d_{\gamma,\varphi}^{a*}\), and this coordinate meets at most one \(y_i\).
	Hence \[|F(\bar y)|\le \frac{\kappa C}{n_q} \le  \frac{\kappa C}{m_q^2}\le \frac{K_{\rm dep}C}{m_q^2}.\]
	If \(h>q\), the pairwise distinctness of \((k_i)\), \eqref{eq:qnt-cod-grw}, and Lemma~\ref{lem:off-wgt-exa} give
	\[ \#\{i:k_i=h\}\le1,\qquad |F(y_i)|\le\frac{C_{\rm off}C}{m_{h\wedge k_i}}\le\frac{C_{\rm off}C}{m_{q+1}}\quad(k_i\ne h). \]
	By \eqref{eq:led-elm} and \eqref{eq:led-dpt-tal}, we have
	\[ |F(\bar y)|\le\frac{1}{n_q}\left(\kappa C+\frac{n_qC_{\rm off}C}{m_{q+1}}\right)=\frac{\kappa C}{n_q}+\frac{C_{\rm off}C}{m_{q+1}}\le\frac{K_{\rm dep}C}{m_q^2}. \]
	If \(h=q\), disjointness of the weight pools makes \(F\) a scalarized outer evaluation for the corner \(c\).
	By \eqref{eq:q-col-bnd} on \(J_\varnothing\), \eqref{eq:led-rat}, and \eqref{eq:led-dpt-tal},
	\[ |F(\bar y)|\le\frac{3\kappa C}{n_qm_q}+\frac{\kappa C}{\Lambda_q}+\frac{3C_{\rm off}C}{m_q\Lambda_q}\le(4\kappa+3C_{\rm off})\frac{C}{m_q^2}\le\frac{K_{\rm dep}C}{m_q^2}. \]

	Now suppose that \(0<h<q\), and apply Lemma~\ref{lem:car-dom-tre} to
	\(F\), \(J_\varnothing\), and \(\bar y\).
	For every coordinate terminal \(t\), write \(\Phi_t=e_{\gamma_t,\varphi_t}^{a*}\) and \(h_t=\operatorname{lev}(\gamma_t)\).
	We separate blocks crossing a terminal endpoint from blocks contained in one carrier.
	Define \[\mathcal E_\partial=\{(t,i)\in\mathcal T_{\rm term}\times\{1,\ldots,n_q\}:\operatorname{ran}y_i\cap J_t\ne\varnothing,\ \operatorname{ran}y_i\nsubseteq J_t\}.\]
	For every terminal \(t\), put \(I_t=\{1\le i\le n_q:\operatorname{ran}y_i\subseteq J_t\}\) and write \(\mathbf{1}_{I_t}\bar y=n_q^{-1}\sum_{i\in I_t}y_i\).
	By Lemma~\ref{lem:car-dom-tre}, the terminal carriers are successive and disjoint and \(|\mathcal T_{\rm term}|\le S_q\).
	Every \((t,i)\in\mathcal E_\partial\) contains an endpoint of \(J_t\).
	Hence
	\[
			|\mathcal E_\partial|\le\sum_{t\in\mathcal T_{\rm term}}\sum_{r\in\{\min J_t,\max J_t\}}\#\{i:r\in\operatorname{ran}y_i\}\le2|\mathcal T_{\rm term}|\le2S_q,\]
	and
	\[I_t\cap I_u=\varnothing\quad(t\ne u),\qquad \sum_{t\in\mathcal T_{\rm term}}|I_t|\le n_q. \]
	Hence \eqref{eq:car-dom} gives
	\[
		\begin{aligned}
			|F(\bar y)|&\le\sum_{t\in\mathcal T_{\rm term}}a_t|\Phi_t(P_{J_t}^a\bar y)|\\
			&\le\frac{1}{n_q}\sum_{(t,i)\in\mathcal E_\partial}a_t|\Phi_t(P_{J_t}^ay_i)|+\sum_{t\in\mathcal T_{\rm term}}a_t|\Phi_tP_{J_t}^a(\mathbf{1}_{I_t}\bar y)|.
		\end{aligned}
	\]

	Since \(a_t\le1\), for a BD terminal \(\Phi_t=d_{\xi_t,\psi_t}^{a*}\), its singleton carrier gives
	\(\Phi_t(P_{J_t}^ay_i)=\Phi_t(y_i)\), and Definition~\ref{def:zer-dep} with \eqref{eq:cod-lam} gives
	\(a_t|\Phi_t(y_i)|\le C/m_{k_i}\le C/\Lambda_q\le C\le\kappa C\).
	For a coordinate terminal, we have
	\(a_t|\Phi_t(P_{J_t}^ay_i)|\le\|\Phi_t\|\,\|P_{J_t}^a\|\,\|y_i\|\le\kappa C\).
	Hence
	\begin{equation}\label{eq:pai-bou}
		\frac{1}{n_q}\sum_{(t,i)\in\mathcal E_\partial}a_t|\Phi_t(P_{J_t}^ay_i)|\le\frac{\kappa C}{n_q}|\mathcal E_\partial|\le\frac{2\kappa CS_q}{n_q}.
	\end{equation}
	If \(\Phi_t\) is a coordinate terminal with \(h_t=0\), then
	\hyperref[itm:ana-ea1]{\textup{(EA1)}} gives the same bound as for a BD terminal.
	Each BD terminal and each such coordinate terminal meets at most one \(y_i\). Hence
	\begin{equation}\label{eq:pai-bd0}
		\begin{aligned}
			\sum_{\substack{t\in\mathcal T_{\rm term}\\
				\Phi_t\text{ a BD terminal or }h_t=0}}
			a_t|\Phi_tP_{J_t}^a(\mathbf{1}_{I_t}\bar y)|
			&\le\frac{C}{n_q}
			\sum_{\substack{t\in\mathcal T_{\rm term}\\
				\Phi_t\text{ a BD terminal or }h_t=0}}
			\sum_{\substack{i\in I_t\\\Phi_t(y_i)\ne0}}\frac{1}{m_{k_i}}&\\
			&\quad\le\frac{C}{n_q\Lambda_q}|\mathcal T_{\rm term}|
			\le\frac{CS_q}{n_q\Lambda_q}.
		\end{aligned}
	\end{equation}
	It remains to estimate the terms in the second sum for which \(t\) is a coordinate with \(h_t>0\).
	For \(h_t=q\), \(\Phi_t\) is a scalarized outer evaluation of corner \(c\), so \eqref{eq:q-col-bnd} applies on \(J_t\) whenever \(I_t\ne\varnothing\).
	For each terminal \(t\) with \(h_t>q\), the pairwise distinctness of \((k_i)\) gives
	\(\#\{i\in I_t:k_i=h_t\}\le1\).
	Using also the disjointness of the \(I_t\)'s, we obtain
	\[
		\begin{aligned}
			\#\{(t,i):h_t>q,\ i\in I_t,\ k_i=h_t\}
			&\le\sum_{h_t>q}1\le|\mathcal T_{\rm term}|\le S_q,\\
			\#\{(t,i):h_t>q,\ i\in I_t,\ k_i\ne h_t\}
			&\le\sum_{h_t>q}|I_t|\le\sum_{t\in\mathcal T_{\rm term}}|I_t|\le n_q.
		\end{aligned}
	\]
	Moreover, \eqref{eq:qnt-cod-grw} gives \(k_i=\sigma_c(q,\pi_{i-1})>q\).
	Hence \(\min\{h_t,k_i\}\ge q+1\) whenever \(h_t>q\), and Lemma~\ref{lem:off-wgt-exa} gives
	\[ |\Phi_t(y_i)| \le\frac{C_{\rm off}C}{m_{\min\{h_t,k_i\}}} \le\frac{C_{\rm off}C}{m_{q+1}} \qquad(h_t>q,\ i\in I_t,\ k_i\ne h_t). \]
	For \(k_i=h_t\), we use \(|\Phi_t(y_i)|\le\|\Phi_t\|\,\|y_i\|\le C\le\kappa C\).
	For \(0<h_t<q\), either \(a_t=0\) or \(t\) has depth \(D_q+2\); hence \(a_t\le m_1^{-(D_q+2)}\).
	Hence
	\begin{equation}\label{eq:pai-pos}
		\begin{aligned}
			\sum_{h_t=q}a_t|\Phi_tP_{J_t}^a(\mathbf{1}_{I_t}\bar y)|&\le\frac{3\kappa CS_q}{n_qm_q}+\frac{\kappa CS_q}{\Lambda_q}+\frac{3C_{\rm off}CS_q}{m_q\Lambda_q},\\
			\sum_{h_t>q}a_t|\Phi_tP_{J_t}^a(\mathbf{1}_{I_t}\bar y)|&\le\frac{1}{n_q}\left(\kappa CS_q+\frac{C_{\rm off}C}{m_{q+1}}n_q\right)=\frac{\kappa CS_q}{n_q}+\frac{C_{\rm off}C}{m_{q+1}},\\
			\sum_{0<h_t<q}a_t|\Phi_tP_{J_t}^a(\mathbf{1}_{I_t}\bar y)|&\le\frac{C}{n_q}\sum_{0<h_t<q}a_t|I_t|\le\frac{C}{n_q}m_1^{-(D_q+2)}\sum_{0<h_t<q}|I_t|\le C m_1^{-(D_q+2)}.
		\end{aligned}
	\end{equation}

	Adding \eqref{eq:pai-bou}, \eqref{eq:pai-bd0}, and \eqref{eq:pai-pos}, and then using \eqref{eq:led-rat} and \eqref{eq:led-dpt-tal}, we have
		\begin{align*}
			|F(\bar y)|
			\le&\frac{2\kappa CS_q}{n_q}
			+\frac{CS_q}{n_q\Lambda_q}
			+\frac{3\kappa CS_q}{n_qm_q}
			+\frac{\kappa CS_q}{\Lambda_q}
			+\frac{3C_{\rm off}CS_q}{m_q\Lambda_q}+\frac{\kappa CS_q}{n_q}
			+\frac{C_{\rm off}C}{m_{q+1}}
			+C m_1^{-(D_q+2)}\\
			\le&8\kappa(C_{\rm off}+1)C\left(\frac{S_q}{n_q}+\frac{S_q}{\Lambda_q}+m_1^{-(D_q+2)}+\frac{1}{m_{q+1}}\right)\\
			\le&8\kappa(C_{\rm off}+1)C\left(2^{-q-8}m_q^{-5}+2^{-q-12}m_q^{-4}+m_q^{-2}+m_q^{-5}\right)\\
			<&32\kappa(C_{\rm off}+1)\frac{C}{m_q^2}\\
			=&\frac{K_{\rm dep}C}{m_q^2}.
		\end{align*}
	Taking the supremum over \(F\in\mathcal{K}_a\) gives \eqref{eq:zer-dep}.
\end{proof}

Figure~\ref{fig:carrier-boundary-count} illustrates the two counts used above.
Boundary terms are counted by pairs \((t,i)\), whereas the contained blocks
are counted only once in \(\sum_t|I_t|\). This is why the deep remainder has
no additional factor \(S_q\).
In the figure, orange blocks cross carrier endpoints; blue blocks are contained in one carrier.
A single block, such as \(y_3\), may give more than one boundary incidence.

\begin{figure}[htbp]
	\centering
\begin{tikzpicture}[x=1cm,y=1cm,>=stealth,
	every node/.style={font=\small,align=center}]
	\node[anchor=west,font=\scriptsize] at (0,3.20) {Terminal carriers};
	\foreach \a/\b/\t in {0.80/3.75/1,4.35/7.10/2,7.80/11.45/3}{
		\draw[draw=black!65,fill=black!6,rounded corners=1pt]
			(\a,2.50) rectangle (\b,2.90);
		\node at ({(\a+\b)/2},2.70) {\(J_{t_{\t}}\)};
		\draw[densely dashed,black!40] (\a,2.50) -- (\a,1.30);
		\draw[densely dashed,black!40] (\b,2.50) -- (\b,1.30);
	}
	\node[anchor=west,font=\scriptsize,fill=white,inner sep=2pt] at (0,2.03) {Block ranges \(\operatorname{ran}y_i\)};
	\foreach \a/\b/\i in {0.15/1.30/1,3.25/4.90/3,6.65/8.35/5,10.90/12.05/7}{
		\draw[draw=orange!75!black,fill=orange!15,rounded corners=1pt]
			(\a,1.12) rectangle (\b,1.52);
		\node[below=3pt] at ({(\a+\b)/2},1.12) {\(y_{\i}\)};
	}
	\foreach \a/\b/\i in {1.85/2.75/2,5.35/6.10/4,8.90/10.20/6}{
		\draw[draw=blue!55!black,fill=blue!12,rounded corners=1pt]
			(\a,1.12) rectangle (\b,1.52);
		\node[below=3pt] at ({(\a+\b)/2},1.12) {\(y_{\i}\)};
	}
	\draw[->,black!65] (0,0.46) -- (12.30,0.46)
		node[below left,font=\scriptsize] {rank};
	\node[text=orange!75!black,anchor=west,font=\scriptsize] at (0,-0.12)
		{Boundary: \((t_1,3),(t_2,3)\in\mathcal E_\partial\)};
	\node[text=blue!55!black,anchor=west,font=\scriptsize] at (6.15,-0.12)
		{Contained: \(2\in I_{t_1},\ 4\in I_{t_2},\ 6\in I_{t_3}\)};
	\node at (6.15,-0.78)
		{\(\displaystyle |\mathcal E_\partial|\le2|\mathcal T_{\rm term}|,
			\qquad \sum_{t\in\mathcal T_{\rm term}}|I_t|\le n_q\)};
\end{tikzpicture}
	\caption{Boundary incidences and contained blocks.}
	\label{fig:carrier-boundary-count}
\end{figure}
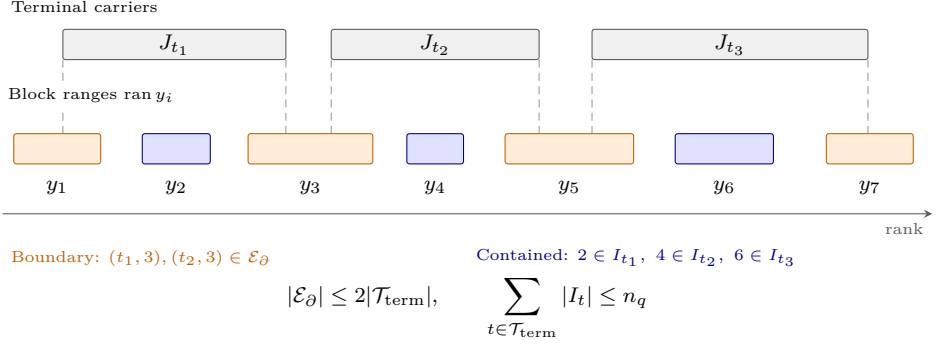

\subsubsection{Operator-dependent sequence}

\begin{lemma}[Operator-dependent extraction]
	\label{lem:opr-crt-ext}
	Let \(c=ij\in\{00,11,01\}\), let \(S\in\Bcal(X_i,X_j)\), and let \((x_k)\subseteq X_i\) be a \(C_0\)-RIS with RIS indices \((j_k)\), where \(C_0\ge1\).
	Suppose that for some \(\varepsilon>0\), we have
	\[ \dist(Sx_k,\mathscr{M}_{ij}(x_k))\ge\varepsilon\qquad(k\ge1). \]
	For every sufficiently large \(q\in\mathcal{C}_c^{\rm seed}\), there is a \(C_{\rm dep}\)-bounded \(q\)-dependent sequence \((y_s,f_s)_{s=1}^{n_q}\) where \(C_{\rm dep}=2A_{\rm ex}C_0\) such that
	\[ \operatorname{Re}f_s(Sy_s)\ge\frac{\varepsilon}{4}\qquad(1\le s\le n_q). \]
	Moreover, for \(\bar y=n_q^{-1}\sum_{s=1}^{n_q}y_s\),
	\begin{equation}\label{eq:out-low}
		\norm{S\bar y}\ge\frac{\varepsilon}{8m_q}.
	\end{equation}
\end{lemma}

\begin{proof}
	Put \(C=2C_0\) and \(C_{\rm dep}=A_{\rm ex}C\).
	Fix a seed \(\vartheta\) of level \(q\), with initial cutoff \(r_0\), and set
	\[ \beta_d=\frac{\varepsilon}{64m_qn_q},\qquad \beta_\times=\frac{\varepsilon}{64n_q},\qquad \beta_0=\frac{\varepsilon}{64},\qquad \theta=\frac{1}{\kappa}\min\{\beta_d,\beta_\times\}=\frac{\varepsilon}{64\kappa m_qn_q}. \]

	Suppose that the first \(s-1\) outer cells have been selected and that the last predecessor rank is \(r_{s-1}\).
	The realized stem \(\pi_{s-1}\) determines \(k_s=\sigma_c(q,\pi_{s-1})\).
	By Lemma~\ref{lem:exa-par-res}, choose \((y_s,f_s)\) sufficiently far out in the corresponding weakly null exact-pair sequence so that
	\begin{align}
		|d_{\zeta_l}^{j*}(Sy_s)|&<\beta_d \quad(l<s),
		\label{eq:out-crb-sml}\\
		|\widetilde f_l(Sy_s)|&<\beta_\times \quad (l<s),
		\label{eq:out-pkt-sml}\\
		\norm{P_{[1,r_{s-1}]}^jSy_s}&<\beta_0.
		\label{eq:out-dia-sml}
	\end{align}
	Take the compatible rational certificate from Lemma~\ref{lem:opr-exf} and choose \(R_s\) beyond its support so that
	\begin{equation}\label{eq:out-tal-sml}
		\norm{P_{(R_s,\infty)}^jSy_t}<\theta\qquad(1\le t\le s).
	\end{equation}
	Choose one of the cofinally pre-existing realizations guaranteed by Proposition~\ref{prop:fin-ext} as an outer successor \(\zeta_s\) after \(R_s\), carrying this certificate and the strong packet \(f_s\).
	Put
	\( r_s=\operatorname{rank}\zeta_s\), and put \(\widetilde f_s=f_sP_{(r_{s-1},\infty)}^j. \)
	If \(s<n_q\), choose the next pair only after this occurrence.

	After \(n_q\) steps, the terminal outer coordinate has the analysis
	\begin{equation}\label{eq:out-crt-ana}
		G_q=\sum_{s=1}^{n_q}d_{\zeta_s}^{j*}+\frac{1}{m_q}\sum_{s=1}^{n_q}\widetilde f_s.
	\end{equation}
	The realized stems and outer cells give the history and placement in Definition~\ref{def:zer-dep}.
	Lemma~\ref{lem:opr-exf} gives the root-perturbed exact-pair structure and
	\[
		\begin{gathered}
			\norm{y_s}\le C_{\rm dep},\qquad \norm{y_s-y_s^0}\le\frac{MC}{m_{k_s}}\le\frac{C_{\rm dep}}{m_{k_s}},\\
			\sup_{\substack{\gamma\ \text{an atom on side }i\\ \varphi\in B_{(G_\gamma^i)^*}}}|d_{\gamma,\varphi}^{i*}(y_s)|\le\frac{C_{\rm dep}}{m_{k_s}},\qquad
			\Delta_c(f_s,y_s)=0,\qquad \operatorname{Re}f_s(Sy_s)\ge\frac{\varepsilon}{4}.
		\end{gathered}
	\]
	Its underlying \(C\)-RIS is a \(C_{\rm dep}\)-RIS, so \((y_s,f_s)_{s=1}^{n_q}\) is a \(C_{\rm dep}\)-bounded \(q\)-dependent sequence.

	The recursive order gives \(r_l>R_s\) for \(l\ge s\) and \(r_{l-1}>R_s\) for \(l>s\).
	By \eqref{eq:out-crb-sml}, Lemma~\ref{lem:fml-blk-bio}, \eqref{eq:fml-ext-bnd}, and \eqref{eq:out-tal-sml},
	\begin{align}
		\left|\sum_{l,s=1}^{n_q}d_{\zeta_l}^{j*}(Sy_s)\right|
		&\le\sum_{l<s}|d_{\zeta_l}^{j*}(Sy_s)|+\sum_{l\ge s}|d_{\zeta_l}^{j*}(P_{(R_s,\infty)}^jSy_s)|\notag\\
		&<\sum_{l<s}\beta_d+\sum_{l\ge s}\kappa\theta\le n_q^2\beta_d<\frac{\varepsilon n_q}{32m_q}.
		\label{eq:out-crb-err}
	\end{align}
	Since \(f_s\) is contractive and \(P_{(r_{s-1},\infty)}^j-I=-P_{[1,r_{s-1}]}^j\), \eqref{eq:out-dia-sml} gives
	\begin{align}
		\left|\sum_{s=1}^{n_q}(\widetilde f_s-f_s)(Sy_s)\right|
		&=\left|\sum_{s=1}^{n_q}f_sP_{[1,r_{s-1}]}^jSy_s\right|\le\sum_{s=1}^{n_q}\norm{P_{[1,r_{s-1}]}^jSy_s}<n_q\beta_0<\frac{\varepsilon n_q}{32}.
		\label{eq:out-hed-err}
	\end{align}
	Since \(\norm{\widetilde f_l}\le\kappa\), \eqref{eq:out-pkt-sml} and \eqref{eq:out-tal-sml} give
	\begin{align}
		\left|\sum_{l\ne s}\widetilde f_l(Sy_s)\right|
		&\le\sum_{l<s}|\widetilde f_l(Sy_s)|+\sum_{l>s}|\widetilde f_l(P_{(R_s,\infty)}^jSy_s)|\notag\\
		&<\sum_{l<s}\beta_\times+\sum_{l>s}\kappa\theta\le n_q(n_q-1)\beta_\times<\frac{\varepsilon n_q}{32}.
		\label{eq:out-crs-err}
	\end{align}
	By \eqref{eq:out-crt-ana} and \eqref{eq:out-crb-err}--\eqref{eq:out-crs-err}, we have
	\begin{align*}
		&\left|G_q(S\bar y)-\frac{1}{m_qn_q}\sum_sf_s(Sy_s)\right|\notag\\
		=&\left|\frac{1}{n_q}\sum_{l,s}d_{\zeta_l}^{j*}(Sy_s)+\frac{1}{m_qn_q}\left(\sum_{l,s}\widetilde f_l(Sy_s)-\sum_sf_s(Sy_s)\right)\right|\notag\\
		\le&\frac{1}{n_q}\left|\sum_{l,s}d_{\zeta_l}^{j*}(Sy_s)\right|+\frac{1}{m_qn_q}\left|\sum_{l\ne s}\widetilde f_l(Sy_s)\right|+\frac{1}{m_qn_q}\left|\sum_s(\widetilde f_s-f_s)(Sy_s)\right|\notag\\
		<&\frac{3\varepsilon}{32m_q}<\frac{\varepsilon}{8m_q}.
	\end{align*}
	It follows that
	\[ \operatorname{Re}G_q(S\bar y) \ge\frac{1}{m_qn_q}\sum_{s=1}^{n_q}\operatorname{Re}f_s(Sy_s)-\frac{\varepsilon}{8m_q} \ge\frac{\varepsilon}{4m_q}-\frac{\varepsilon}{8m_q} =\frac{\varepsilon}{8m_q}. \]
	Since \(\norm{G_q}\le1\), this proves \eqref{eq:out-low}.
\end{proof}

For readers who do not wish to follow the technical construction in this section, we now give a short proof. It recalls the main idea and the inputs from the protected module, the finite-extension property, and the RIS estimates.

\begin{proposition}[The three certified local orbits]
	\label{prop:crt-loc-orb}
	Let \(c=ij\in\{00,11,01\}\), let \(S\in\Bcal(X_i,X_j)\), and let \((x_k)\subseteq X_i\) be a RIS.
	Then
	\[ \dist(Sx_k,\mathscr{M}_{ij}(x_k))\longrightarrow0. \]
\end{proposition}

\begin{proof}
	Otherwise we can pass to a subsequence for which there is \(\varepsilon>0\) such that
	\[ \dist(Sx_k,\mathscr{M}_{ij}(x_k))\ge\varepsilon\qquad(k\ge1), \]
	and choose \(C_0\ge1\) and RIS indices \((j_k)\) so that this subsequence is a \(C_0\)-RIS.
	We recall the three earlier inputs. The protected-module construction gives the local orbits in \eqref{eq:fou-orb}. Corollary~\ref{cor:ris-spc-avg} controls the inner RIS special averages. Corollary~\ref{cor:ris-wek-nul} makes the original RIS weakly null, and the shrinking conclusion of Proposition~\ref{prop:ris-red} makes successive bounded exact vectors weakly null. Fix a sufficiently large \(q\in\mathcal{C}_c^{\rm seed}\). At the \(s\)-th outer step, the current stem \(\pi_{s-1}\) fixes the inner level \(k_s\). Lemma~\ref{lem:opr-exf} recursively selects rational perturbations \(u_{s,1}<\cdots<u_{s,n_{k_s}}\) of terms from the remaining RIS tail and chooses cofinally pre-existing occurrences carrying their separating packets in a complete level-\(k_s\) inner chain with terminal atom \(\eta_s\). Put
	\[ k_s=\sigma_c(q,\pi_{s-1}), \qquad y_s^0=\frac{m_{k_s}}{n_{k_s}}\sum_{t=1}^{n_{k_s}}u_{s,t}, \qquad f_s=e_{\eta_s}^{j*}. \]
	Definition~\ref{def:roo-exa} gives \(y_s\) as the root perturbation of \(y_s^0\) determined by \(f_s\), and Lemma~\ref{lem:exa-par-res} lets these pairs be chosen successively and weakly null. Lemmas~\ref{lem:dfc-cnt} and~\ref{lem:opr-exf} give their two relevant properties
	\[ f_s|_{\mathscr{M}_{ij}(y_s)}=0, \qquad \operatorname{Re}f_s(Sy_s)\ge\frac{\varepsilon}{4}. \]
	Thus the inner chain converts one block of the original RIS into one exact pair. The outer chain has a different role. At each outer step, choose one of the cofinally pre-existing realizations guaranteed by Proposition~\ref{prop:fin-ext} that carries the next certificate. Its terminal evaluation combines the functionals \(f_s\) with the outer weight \(m_q^{-1}\), while the weak-null choices control the cross terms. Lemma~\ref{lem:opr-crt-ext} packages this recursion. The uniform separation above is exactly its hypothesis. Hence it gives a \(C_{\rm dep}\)-bounded \(q\)-dependent sequence in the sense of Definition~\ref{def:zer-dep}, where \(C_{\rm dep}=2A_{\rm ex}C_0\), and
	\[ \bar y=\frac{1}{n_q}\sum_{s=1}^{n_q}y_s, \qquad \norm{S\bar y}\ge\frac{\varepsilon}{8m_q}. \]
	This is precisely the class of sequences covered by Lemma~\ref{lem:pai-bi}. Its proof combines the RIS basic inequality in Lemma~\ref{lem:vec-bas-ine} with Lemmas~\ref{lem:crt-stm-col} and~\ref{lem:car-dom-tre}. Therefore
	\[ \norm{\bar y}\le\frac{K_{\rm dep}C_{\rm dep}}{m_q^2}. \]
	By Subsection~\ref{sec:req-cat} and Proposition~\ref{prop:fin-ext}\textup{(2)}, the allowed seed levels are unbounded and every such level has a copy after every cutoff. Thus \(q\) can be arbitrarily large, and a copy of that level can be chosen arbitrarily far out. The two estimates give
	\[ \frac{\varepsilon}{8m_q} \le\norm{S}\norm{\bar y} \le\frac{K_{\rm dep}C_{\rm dep}\norm{S}}{m_q^2}, \]
	which is impossible for large \(q\).
\end{proof}

\section{Reverse corner and missing levels}
\label{sec:rev}

The reverse corner has no protected target orbit, so the collision argument above is not available.
We use the reserved reverse levels instead.
A level \(r\in\mathcal{R}\) occurs in a source coordinate and is absent from every target evaluation tree.
This lets a source coordinate norm the image of the average,
while Lemma~\ref{lem:mis-wgt} gives an \(O(m_r^{-2})\) upper bound for the norm of the average in \(X_1\).

The stronger auxiliary-average estimate obtained by omitting a weight is a
well-established tool in AH-type constructions; see
\cite[Proposition~2.5]{ArgyrosHaydon2011}
and \cite[Proposition~11.7, \textup{(11.3)}]{MotakisPelczar2025}.
The proof below adapts this method to the reserved reverse levels of our
construction and transfers the estimate to RIS averages through the basic inequality.

For \(r\in\mathcal{R}\), let \(W^{(r)}\) be the smallest auxiliary norming set containing every \(\zeta e_k^*\),
where \(k\in\N_+\) and \(|\zeta|=1\), and closed, for \(h\ne r\), under
\[ f_1<\cdots<f_d,\quad d\le3n_h \quad\longmapsto\quad \frac{1}{m_h}\sum_{t=1}^df_t. \]
Write \(\norm{\cdot}_{T,r}\) for its norm.

\begin{remark}[Auxiliary-level inheritance]
	\label{rem:mis-lev-inh}
	In the induction in Lemma~\ref{lem:vec-bas-ine}, a level-zero evaluation is absorbed by the exceptional coordinate term.
	At a positive-level evaluation of level \(h\), the child estimates are combined by one
	\((\mathcal{A}_{3n_h},m_h^{-1})\)-operation. Passing from an absolutely convex packet to one scalarized coordinate adds no auxiliary operation.
	Hence, if no coordinate evaluation used in the induction has level \(r\), then the resulting auxiliary functional belongs to \(W^{(r)}\cup\{0\}\).
\end{remark}

\begin{lemma}[Missing-level auxiliary average]
	\label{lem:mis-wgt}
	For every \(r\in\mathcal{R}\),
	\begin{equation}\label{eq:mis-aux-avg}
		\left\|\frac{1}{n_r}\sum_{k=1}^{n_r}e_k\right\|_{T,r}
		\le\frac{2}{m_r^2}.
	\end{equation}
	Consequently every \(C\)-RIS \((x_k)_{k=1}^{n_r}\subseteq X_1\) satisfies
	\begin{equation}\label{eq:mis-tgt-avg}
		\norm{\frac{1}{n_r}\sum_{k=1}^{n_r}x_k}_{X_1}
		\le\frac{4BC}{m_r^2}.
	\end{equation}
\end{lemma}

\begin{proof}
	By the stopping-tree argument in Lemma~\ref{lem:aux-avg}, with \(h_0=r\) and \(d=D_r\), it is enough to stop at a coordinate leaf, at the first level larger than \(r\), or after \(D_r\) lower-level nodes. Since level \(r\) is absent, the first higher level is at least \(r+1\), and the lower-level branching is bounded by \(A_r\). By \eqref{eq:par-grw}, \eqref{eq:n-grw}, and the definitions of \(D_r\) and \(L_r\), every \(g\in W^{(r)}\) satisfies
	\[
		\begin{aligned}
			\left|g\left(\frac{1}{n_r}\sum_{k=1}^{n_r}e_k\right)\right|
			&\le\frac{1}{m_{r+1}}+m_1^{-D_r}
			+\frac{1+A_r+\cdots+A_r^{D_r-1}}{n_r}\\
			&\le\frac{1}{m_r^5}+\frac{1}{m_r^2}
			+\frac{1}{2^{r+8}m_r^5}<\frac{2}{m_r^2}.
		\end{aligned}
	\]
	Then \eqref{eq:mis-aux-avg} follows.

	Fix \(f\in\mathcal{K}_1\) and apply Lemma~\ref{lem:vec-bas-ine} with \(s=0\).
	By \hyperref[itm:ana-g4]{\textup{(G4)}}, every target coordinate evaluation used in the induction has level different from \(r\).
	Remark~\ref{rem:mis-lev-inh} gives \(g\in W^{(r)}\cup\{0\}\), while \(e_{k_0}^*\in W^{(r)}\) by definition. Hence
	\[
		\begin{aligned}
			\left|f\left(\sum_k\lambda_kx_k\right)\right|
			&\le BC\left(
			e_{k_0}^*\left(\sum_k|\lambda_k|e_k\right)
			+g\left(\sum_k|\lambda_k|e_k\right)\right)\le 2BC\norm{\sum_k|\lambda_k|e_k}_{T,r}.
		\end{aligned}
	\]
	Taking the supremum over \(f\in\mathcal{K}_1\), putting \(\lambda_k=1/n_r\), and using \eqref{eq:mis-aux-avg} gives
	\[ \norm{\frac{1}{n_r}\sum_{k=1}^{n_r}x_k} \le2BC\cdot\frac{2}{m_r^2} =\frac{4BC}{m_r^2}, \]
	which is \eqref{eq:mis-tgt-avg}.
\end{proof}

\begin{proposition}[Reverse local orbit]
	\label{prop:rev-loc-orb}
	If \(S:X_1\to X_0\) is bounded and \((x_k)\subseteq X_1\) is a bounded RIS, then \(\norm{Sx_k}\to 0\).
\end{proposition}

\begin{proof}
	Suppose not. After passing to a subsequence, choose \(C\ge1\) and \(\varepsilon>0\) such that
	\((x_k)\) is a \(C\)-RIS and \(\norm{Sx_k}\ge\varepsilon\). Fix an arbitrarily large \(r\in\mathcal{R}\).
	By Corollary~\ref{cor:ris-wek-nul}, \((x_k)\) is weakly null, so \((Sx_k)\) is weakly null.

	Fix an initial cutoff \(\chi\) before the first selected block and put \(p_0=\chi\).
	By the weak-null truncation and rational approximation used in the proof of Lemma~\ref{lem:rat-app-ann}, with
	\(\mathscr{M}_{10}(x)=\{0\}\) from \eqref{eq:fou-orb}, choose recursively a subsequence, denoted again by \((x_i)\),
	finite intervals \(I_i\), and contractive represented rational source packets \(\mathbf h_i\).
	Put \(h_i=\left.B_{\mathbf h_i}^*(1)\right|_{X_0}\). After \(I_i\) and \(\mathbf h_i\) are chosen,
	choose one of the cofinally pre-existing realizations guaranteed by Proposition~\ref{prop:fin-ext}\textup{(1)} as the next occurrence \(\xi_i\) of a level-\(r\) reverse chain, carrying the packet \(\mathbf h_i\),
	at a rank \(p_i=\operatorname{rank}\xi_i\) beyond \(I_i\cup\supp_{\rm raw}\mathbf h_i\). Thus
	\begin{equation}\label{eq:rev-win-dat}
		\begin{gathered}
			\xi_1\prec\cdots\prec\xi_{n_r},
			\qquad I_i\cup\supp_{\rm raw}\mathbf h_i\subseteq(p_{i-1},p_i),\\
			\operatorname{Re}h_iP_{I_i}^0Sx_i>\frac{\varepsilon}{2},
			\qquad \norm{Sx_i-P_{I_i}^0Sx_i}<\frac{\varepsilon}{8m_r}
			\quad(1\le i\le n_r).
		\end{gathered}
	\end{equation}

	By \hyperref[itm:ana-g4]{\textup{(G4)}} and the evaluation analysis \eqref{eq:ful-vec-ea},
	the terminal source coordinate has the following expansion. Coordinate contractivity gives \(\norm{f_r}\le1\).
	\begin{equation}\label{eq:act-rev-coo}
		f_r=e_{\xi_{n_r}}^{0*}
		=\sum_{i=1}^{n_r}d_{\xi_i}^{0*}
		+\frac{1}{m_r}\sum_{i=1}^{n_r}
		h_iP_{(p_{i-1},\infty)}^0,
		\qquad \norm{f_r}\le1.
	\end{equation}
	Put \(\bar{x}=n_r^{-1}\sum_ix_i\) and \(\bar{z}=n_r^{-1}\sum_iP_{I_i}^0Sx_i\).
	By FDD biorthogonality and \eqref{eq:rev-win-dat},
	\[
		d_{\xi_i}^{0*}P_{I_j}^0=0,
		\qquad
		h_iP_{(p_{i-1},\infty)}^0P_{I_j}^0=
		\begin{cases}
			h_iP_{I_i}^0,&j=i,\\
			0,&j\ne i.
		\end{cases}
	\]
	By \eqref{eq:act-rev-coo} and \eqref{eq:rev-win-dat}, we have
	\[ \operatorname{Re}f_r(\bar z) =\frac{1}{m_rn_r}\sum_i\operatorname{Re}h_iP_{I_i}^0Sx_i >\frac{\varepsilon}{2m_r}, \qquad \norm{S\bar x-\bar z} \le\frac{1}{n_r}\sum_i\norm{Sx_i-P_{I_i}^0Sx_i} <\frac{\varepsilon}{8m_r}. \]
	Since \(\norm{f_r}\le1\), it follows that
	\[ \norm{S\bar{x}}\ge\operatorname{Re}f_r(S\bar{x}) \ge\operatorname{Re}f_r(\bar z)-\norm{S\bar x-\bar z} >\frac{\varepsilon}{2m_r}-\frac{\varepsilon}{8m_r} >\frac{\varepsilon}{4m_r}. \]
	On the other hand, Lemma~\ref{lem:mis-wgt} gives \(\norm{\bar{x}}\le4BC/m_r^2\).
	Hence
	\[ \frac{\varepsilon}{4m_r}<\norm{S\bar{x}}\le\frac{4BC\norm{S}}{m_r^2}, \]
	which is impossible for arbitrarily large \(r\in\mathcal{R}\).
\end{proof}

\section{Operator classification and consequences}

\subsection{Local orbit theorem}

By Proposition~\ref{prop:crt-loc-orb} and Proposition~\ref{prop:rev-loc-orb}, we have the following four local orbit estimates.

\begin{theorem}[Four local orbit estimates]\label{thm:loc-orb}
	Let \(S:X_i\to X_j\) be bounded and let \((x_k)\) be a bounded RIS in \(X_i\).
	Then
	\begin{equation}\label{eq:loc-orb}
		\dist\bigl(Sx_k,\mathscr{M}_{ij}(x_k)\bigr)\longrightarrow0.
	\end{equation}
\end{theorem}

\subsection{Global classification and the Calkin algebra}

The passage from local orbit estimates to global operator classification will use compatibility and limiting arguments for the local coefficients.
In each self corner, synchronization shows that the local scalar coefficients obtained from different RISs have the same limit.
In the forward corner, probes read finite-dimensional coefficient functionals,
and two-probe estimates show that their restrictions are asymptotically compatible.
Since these functionals are uniformly bounded and the coefficient fibres have dense union in \(V\), their compatible restrictions determine only one \(v\in V^*\).
After subtracting the corresponding scalar operator or \(T_v\), the residual operator kills every bounded RIS and is therefore compact. We adapt the paired-RIS argument in \cite[proof of Theorem~7.4]{ArgyrosHaydon2011}.

\begin{lemma}[Synchronizing RISs]\label{lem:ris-syn}
	Two normalized RISs in the same space have subsequences \((x_k)\), \((y_k)\) such that
	\[ x_k<y_k<x_{k+1} \]
	and \((x_k+y_k)\) is a bounded RIS.
	Moreover the same holds for two subsequences of a single RIS.
\end{lemma}

\begin{proof}
	The proof is almost the same as that of
	Proposition~\ref{prop:fml-prb}\textup{(P2)(c)}.
	After passing to subsequences and relabelling, we have
	\[ x_k<y_k<x_{k+1},\qquad \ell_k=\min\{j_k^x,j_k^y\},\qquad \ell_{k+1}>\max\operatorname{ran}(x_k+y_k), \]
	where \((j_k^x)\) and \((j_k^y)\) are the corresponding RIS indices.
	If \(0<\operatorname{lev}(\gamma)=h<\ell_k\), then
	\[ \norm{x_k+y_k}\le C_x+C_y,\qquad |e_{\gamma,\varphi}^*(x_k+y_k)| \le\frac{C_x+C_y}{m_h}. \]
	Hence \((x_k+y_k)\) is a \((C_x+C_y)\)-RIS.
	The same argument applies to two subsequences of one RIS.
\end{proof}

\begin{lemma}[Globalizing scalar local orbits]\label{lem:slf-glo}
	Let \(X_i\) be \(X_0\) or \(X_1\), and let \(S\in\Bcal({X_i})\).
	If
	\[ \dist(Sx_k,\K x_k)\longrightarrow0 \]
	for every bounded RIS \((x_k)\), then there is a unique \(\lambda\in\K\) such that \(S-\lambda I_{X_i}\) is compact.
\end{lemma}

\begin{proof}
	For a normalized RIS \((x_k)\), choose \(a_k\in\K\) with
	\[ \varepsilon_k:=\norm{Sx_k-a_kx_k}\longrightarrow0. \]
	Since \(\|x_k\|=1\), we have
	\[ |a_k|=\|a_kx_k\|\le\|Sx_k\|+\varepsilon_k \le\|S\|+\varepsilon_k, \]
	so \((a_k)\) is bounded.
	If two of its subsequences converge to \(a\) and \(b\), write them as \(u_k=x_{p_k}\) and \(v_k=x_{q_k}\),
	and synchronize them so that \(u_k<v_k<u_{k+1}\).
	Lemma~\ref{lem:ris-syn} makes \(w_k=u_k+v_k\) a bounded RIS,
	so there are \(c_k\in\K\) with \(\eta_k=\|Sw_k-c_kw_k\|\to 0\).
	And we have
	\[ (c_k-a_{p_k})u_k+(c_k-a_{q_k})v_k =(Su_k-a_{p_k}u_k)+(Sv_k-a_{q_k}v_k) -(Sw_k-c_kw_k). \]
	The interval projections onto \(\ran u_k\) and \(\ran v_k\) have norm at most \(2M\), so
	\[ \max\{|c_k-a_{p_k}|,|c_k-a_{q_k}|\} \le2M\norm{(c_k-a_{p_k})u_k+(c_k-a_{q_k})v_k} \le2M(\varepsilon_{p_k}+\varepsilon_{q_k}+\eta_k)\longrightarrow0. \]
	Thus \(a=b\).
	So \((a_k)\) converges. And we
	denote its limit by \(\lambda_S\).
	For any other normalized RIS, the same argument gives a coefficient limit. After synchronizing subsequences of the two RISs, Lemma~\ref{lem:ris-syn} makes their sums a bounded RIS, and the preceding projection estimate shows that the two limits agree. Thus \(\lambda_S\) is independent of the normalized RIS.
	Hence, for every bounded RIS, we have \((S-\lambda_S I_{X_i})x_k\to 0\).
	Indeed, put \(R=S-\lambda_S I_{X_i}\).
	If a RIS had a subsequence with \(\|Rx_k\|\ge\varepsilon\),
	then \(\|x_k\|\ge\varepsilon/\|R\|\) on that subsequence.
	The normalized vectors \(u_k=x_k/\|x_k\|\) are a RIS,
	so the definition of \(\lambda_S\) gives \(\|Ru_k\|\to 0\) and then \(\|Rx_k\|\to 0\),
	which is a contradiction.
	Proposition~\ref{prop:ris-red} makes \(S-\lambda_S I_{X_i}\) compact.
	Uniqueness follows because \(I_{X_i}\) is not compact.
\end{proof}

Hence, for the self corners, we obtain the Argyros--Haydon scalar-plus-compact property.
Every operator in the reverse corner is compact.

\begin{corollary}\label{cor:slf-rev}
	\[ \Bcal(X_i)/\Kcal(X_i)=\K[I_{X_i}]\quad(i=0,1), \qquad \Bcal(X_1,X_0)=\Kcal(X_1,X_0). \]
\end{corollary}

It remains to treat the forward corner.

\begin{theorem}[Globalizing the forward orbit]\label{thm:fwd-glo}
	For every \(S\in\Bcal(X_0,X_1)\) there is a unique \(v\in V^*\) such that \(S-T_v\) is compact.
	Moreover
	\( \norm{v}\le M^2\norm{S}. \)
\end{theorem}

\begin{proof}
	By \eqref{eq:ter-par}, for a probe \(\gamma\in\mathcal P\) and \(a_\gamma=\phi_\gamma Sj_\gamma\in E_\gamma^*\), we have
		\[ \norm{a_\gamma}\le M^2\norm{S}. \]
	For \(N\in\N_+\) and \(z_\gamma=d_\gamma^1(1)\) defined in Subsection~\ref{sec:fre-prb}, put
	\[
		\varepsilon_N=\sup_{\substack{\gamma\in\mathcal P,\ \operatorname{rank}\gamma\ge N\\
			u\in B_{E_\gamma}}}
			\dist(Sj_\gamma u,\K z_\gamma).
	\]
	We first prove that \(\varepsilon_N\to 0\).
	Suppose not. By the proof of Proposition~\ref{prop:fml-prb}\textup{(P2)(a)},
	the violating probes can be chosen so that, for some \(\delta>0\),
	\[ u_k\in B_{E_{\gamma_k}}, \qquad (j_{\gamma_k}u_k)_k\text{ is a RIS}, \qquad \dist(Sj_{\gamma_k}u_k,\K z_{\gamma_k})>\delta. \]
	And it follows that \(u_k\ne0\). By the probe identity in
	\hyperref[itm:ana-g6]{\textup{(G6)}},
	\[ \mathscr{M}_{01}(j_{\gamma_k}u_k) =\{w(u_k)z_{\gamma_k}:w\in V^*\} =\K z_{\gamma_k}. \]
	This contradicts Theorem~\ref{thm:loc-orb}. Hence \(\varepsilon_N\to 0\).

	If \(\operatorname{rank}\gamma\ge N\) and \(u\in B_{E_\gamma}\), then
	\begin{equation}\label{eq:ter-err}
		\begin{aligned}
			\norm{Sj_\gamma u-a_\gamma(u)z_\gamma}
			&\le\inf_{c\in\K}\bigl(\norm{Sj_\gamma u-cz_\gamma}
				+|\phi_\gamma(Sj_\gamma u-cz_\gamma)|\norm{z_\gamma}\bigr)\\
			&\le(1+M^2)\dist(Sj_\gamma u,\K z_\gamma)
			\le(1+M^2)\varepsilon_N.
		\end{aligned}
	\end{equation}

	For \(N\in\N_+\), put
	\[
		\omega_N=\sup_{\substack{\gamma,\delta\in\mathcal P,\
			N\leq\operatorname{rank}\gamma<\operatorname{rank}\delta\\
			u\in B_{E_\gamma\cap E_\delta}}}
			|a_\gamma(u)-a_\delta(u)|.
	\]
	We next prove that \(\omega_N\to 0\).
	Suppose not. By the proof of Proposition~\ref{prop:fml-prb}\textup{(P2)(b)},
	there are \(\tau>0\), probes \(\gamma_k<\delta_k<\gamma_{k+1}\), and
	unit vectors \(u_k\in E_{\gamma_k}\cap E_{\delta_k}\) such that
	\[ |a_{\gamma_k}(u_k)-a_{\delta_k}(u_k)|>\tau, \qquad (j_{\gamma_k}u_k+j_{\delta_k}u_k)_k\text{ is a RIS}. \]
	By \hyperref[itm:ana-g6]{\textup{(G6)}}, we have
	\[ \mathscr{M}_{01}(j_{\gamma_k}u_k+j_{\delta_k}u_k) =\K(z_{\gamma_k}+z_{\delta_k}). \]
	Theorem~\ref{thm:loc-orb} gives \(c_k\in\K\) such that
	\[ \norm{S(j_{\gamma_k}u_k+j_{\delta_k}u_k) -c_k(z_{\gamma_k}+z_{\delta_k})}\longrightarrow0. \]
	By Lemma~\ref{lem:fml-blk-bio} and \eqref{eq:ter-err},
	\begin{align*}
		\tau&<|a_{\gamma_k}(u_k)-a_{\delta_k}(u_k)|\le |a_{\gamma_k}(u_k)-c_k|+|a_{\delta_k}(u_k)-c_k|\\
		&\le2M\Bigl(
		\norm{S(j_{\gamma_k}u_k+j_{\delta_k}u_k)
		-c_k(z_{\gamma_k}+z_{\delta_k})}\\
		&\qquad \qquad +\norm{Sj_{\gamma_k}u_k-a_{\gamma_k}(u_k)z_{\gamma_k}}
		+\norm{Sj_{\delta_k}u_k-a_{\delta_k}(u_k)z_{\delta_k}}
		\Bigr)\longrightarrow0,
	\end{align*}
	which is a contradiction. Hence \(\omega_N\to 0\).

	For \(u\in V_{\Q}\) and probes \(\gamma<\delta\) whose fibres contain \(u\), we have
	\[ |a_\gamma(u)-a_\delta(u)| \le\omega_{\operatorname{rank}\gamma}\norm{u}. \]
	Proposition~\ref{prop:fml-prb}\textup{(P1)} gives cofinally many such probes,
	so we may define
	\[
		v(u)=\lim_{\substack{\gamma\to\infty\\E_\gamma\ni u}}a_\gamma(u)
		\qquad(u\in V_{\Q}).
	\]
	By applying Proposition~\ref{prop:fml-prb}\textup{(P1)} to finitely many vectors and passing to the limit, we obtain
	\[ v(\alpha u+\beta w)=\alpha v(u)+\beta v(w), \qquad |v(u)|\le M^2\norm{S}\norm{u} \]
	for \(u,w\in V_{\Q}\) and rational scalars \(\alpha,\beta\).
	The density of \(V_{\Q}\) and of the rational scalars gives a unique
	\(v\in V^*\) with \(\norm{v}\le M^2\norm{S}\).

	Let \(\operatorname{rank}\gamma\ge N\) and \(u\in B_{E_\gamma}\cap V_{\Q}\).
	By Proposition~\ref{prop:fml-prb}\textup{(P1)}, probes \(\delta>\gamma\) with \(E_\delta\supseteq E_\gamma\)
	occur arbitrarily late. Hence
	\[
		|a_\gamma(u)-v(u)|
		=\lim_{\substack{\delta\to\infty\\E_\delta\supseteq E_\gamma}}
			|a_\gamma(u)-a_\delta(u)|
		\le\omega_N.
	\]
	The same estimate holds on \(B_{E_\gamma}\) by density. Therefore
	\begin{equation}\label{eq:mov-col}
		\sup_{\substack{\gamma\in\mathcal P,\ \operatorname{rank}\gamma\ge N\\
			u\in B_{E_\gamma}}}
			|a_\gamma(u)-v(u)|
		\le\omega_N\longrightarrow0.
	\end{equation}

	Put \(R=S-T_v\). By \hyperref[itm:ana-g6]{\textup{(G6)}}, \eqref{eq:ter-err} and \eqref{eq:mov-col} give
	\begin{equation}\label{eq:prb-r}
		\sup_{\substack{\gamma\in\mathcal P,\ \operatorname{rank}\gamma\ge N\\
			u\in B_{E_\gamma}}}
			\norm{Rj_\gamma u}
		\le(1+M^2)\varepsilon_N+M\omega_N\longrightarrow0.
	\end{equation}

	We now prove that \(R\) kills every  RIS.
	Suppose that a  \(C_x\)-RIS \((x_k)\), with \(C_x\ge1\) and RIS indices
	\((\ell_k)\), satisfies \(\norm{Rx_k}\ge\varepsilon>0\).
	Choose \(\eta_k\) with \(|e_{\eta_k}^{1*}(Rx_k)|\ge\varepsilon/2\). And we may assume
	\[ \operatorname{Re}e_{\eta_k}^{1*}(Rx_k)\ge\frac{\varepsilon}{2}. \]
	Put \(u_k=U_{\eta_k}x_k\).
	If \(u_k=0\) on a subsequence, then \eqref{eq:int} gives
	\[ e_{\eta_k}^{1*}(T_wx_k)=w(u_k)=0, \qquad \dist(Rx_k,\mathscr{M}_{01}(x_k)) \ge |e_{\eta_k}^{1*}(Rx_k)| \ge\frac{\varepsilon}{2}, \]
	which contradicts Theorem~\ref{thm:loc-orb}.
	So we pass to a subsequence on which \(u_k\ne0\).
	Since \(\phi_\gamma=d_{\gamma,1}^{1*}\) is defined in Subsection~\ref{sec:fre-prb}, Proposition~\ref{prop:fml-prb}\textup{(P1)},  the FDD tails of
	\(Rx_k\), and the cofinal RIS indices allow us to choose probes \(\gamma_k\) such that
	\[
		\begin{gathered}
			\max\{\max\operatorname{ran}x_k,\operatorname{rank}\eta_k\}
			<\operatorname{rank}\gamma_k
			<\min\{\min\operatorname{ran}x_{k+1},\ell_{k+1}\},\\
			E_{\eta_k}\subseteq E_{\gamma_k},
			\qquad |\phi_{\gamma_k}(Rx_k)|<2^{-k}.
		\end{gathered}
	\]
	Then put
	\[ y_k=x_k+j_{\gamma_k}u_k, \qquad g_k^*=e_{\eta_k}^{1*}-\phi_{\gamma_k}. \]
	By \eqref{eq:u-bnd}, \(\norm{u_k}\le C_x\).
	The proof of Proposition~\ref{prop:fml-prb}\textup{(P2)(c)}, applied to
	\(x_k/C_x\) and \(u_k/C_x\), shows that \((y_k)\) is a
	\(C_x(1+M)\)-RIS. Moreover, \eqref{eq:prb-r} gives
	\[
		\norm{Rj_{\gamma_k}u_k}
		\le C_x\sup_{\substack{\gamma\in\mathcal P,\
			\operatorname{rank}\gamma\ge\operatorname{rank}\gamma_k\\
			u\in B_{E_\gamma}}}
			\norm{Rj_\gamma u}
		\longrightarrow0.
	\]

	By the rank conditions, Lemma~\ref{lem:fml-blk-bio},
	\hyperref[itm:ana-g6]{\textup{(G6)}}, and \eqref{eq:int}, for every \(w\in V^*\), we have
	\[
		\begin{gathered}
			e_{\eta_k}^{1*}(T_wj_{\gamma_k}u_k)=0,\quad \phi_{\gamma_k}(T_wx_k)=0,\quad
			e_{\eta_k}^{1*}(T_wx_k)
			=w(u_k)=\phi_{\gamma_k}(T_wj_{\gamma_k}u_k).
		\end{gathered}
	\]
	Hence
	\[
		\begin{aligned}
			g_k^*(T_wy_k)=e_{\eta_k}^{1*}(T_wx_k)
				+e_{\eta_k}^{1*}(T_wj_{\gamma_k}u_k)
				-\phi_{\gamma_k}(T_wx_k)
				-\phi_{\gamma_k}(T_wj_{\gamma_k}u_k)=w(u_k)-w(u_k)=0.
		\end{aligned}
	\]
	Thus \(g_k^*\) annihilates \(\mathscr{M}_{01}(y_k)\). Also
	\begin{align*}
		\operatorname{Re}g_k^*(Ry_k)
		&=\operatorname{Re}e_{\eta_k}^{1*}(Rx_k)
			+\operatorname{Re}g_k^*(Rj_{\gamma_k}u_k)
			-\operatorname{Re}\phi_{\gamma_k}(Rx_k)\\
		&\ge\frac{\varepsilon}{2}
			-(1+M)\norm{Rj_{\gamma_k}u_k}-2^{-k}
			\ge\frac{\varepsilon}{3}
	\end{align*}
	for all sufficiently large \(k\). Since \(\norm{g_k^*}\le1+M\),
	\[ \dist(Ry_k,\mathscr{M}_{01}(y_k)) \ge\frac{\operatorname{Re}g_k^*(Ry_k)}{\norm{g_k^*}} \ge\frac{\varepsilon}{3(1+M)}, \]
	which contradicts Theorem~\ref{thm:loc-orb}.
	Therefore \(R\) kills every bounded RIS, and Proposition~\ref{prop:ris-red} makes
	\(R\) compact.

	If \(S-T_w\) is also compact, then \eqref{eq:tv-ess-low} gives
	\[ 0=\norm{[T_{v-w}]}=\norm{v-w}, \]
	so \(v=w\).
\end{proof}

\begin{proof}[Proof of Theorem~\ref{thm:tri-cal}]
	By Theorem~\ref{thm:dat-cns}, there are separable spaces \(X_0,X_1\) with FDDs.
	Both are infinite dimensional, because the target probes give a nonzero target vector at every rank by \eqref{eq:ter-par},
	while Proposition~\ref{prop:fin-ext} meets scalar source-only generic requirements cofinally.
	Take \(X_V=(X_0\oplus X_1)_{\ell_\infty}\).
	Corollary~\ref{cor:slf-rev} and Theorem~\ref{thm:fwd-glo}
	give the three types of corner assertions, and Proposition~\ref{prop:cor-alg}
	gives the Banach-algebra isomorphism.
	Lemma~\ref{lem:tv-ess-low} gives \(\norm{[T_v]}=\norm{v}\). The block embedding and the lower-left corner projection induce inverse contractions between the forward-corner quotient and the lower-left radical of \(\Cal(X_V)\).
	Under this isomorphism the Jacobson radical is
	\[
		\left\{\begin{pmatrix}0&0\\v&0\end{pmatrix}:v\in V^*\right\}.
	\]
\end{proof}

\subsection{Structural consequences}\label{sec:str-con}

The remaining consequences use the main construction and elementary Banach-algebra arguments.
For a Banach space \(G\), put
\[
	\mathfrak{A}(G)=
	\left\{\begin{pmatrix}\lambda&0\\g&\mu\end{pmatrix}:
	\lambda,\mu\in\K,\ g\in G\right\},
\]
with the multiplication in \eqref{eq:tri-prd}.

\begin{proof}[Proof of Theorem~\ref{thm:dua-cal}]
	Suppose first that \(\dim V\ge2\).
	Choose \(v_0,v_1\in V\) and \(\phi_0,\phi_1\in V^*\) such that \(\phi_i(v_j)=\delta_{ij}\), and put
	\[ F=\operatorname{span}\{v_0,v_1\},\qquad Pf=f-f(v_0)\phi_0-f(v_1)\phi_1. \]
	Let
	\[ Q:V\longrightarrow W:=V/F \]
	be the quotient map.
	Since \(Q^*\) is an isometric isomorphism from \(W^*\) onto \(F^\perp\),
	and since the biorthogonal system gives the topological direct sum, we have
	\[ V^*=\K\phi_0\oplus F^\perp\oplus\K\phi_1. \]
	Hence
	\[ \Theta:V^*\longrightarrow\mathfrak{A}(W^*),\qquad \Theta(f)= \bigl(f(v_0),(Q^*)^{-1}Pf,f(v_1)\bigr) \]
	is a Banach-space isomorphism.
	Define
	\begin{align}
		\norm{f}_\diamond={}&|f(v_0)|+\norm{Pf}+|f(v_1)|,
		\label{eq:dia-nrm}\\
		f\diamond g={}&f(v_0)g(v_0)\phi_0+g(v_0)Pf
		+f(v_1)Pg+f(v_1)g(v_1)\phi_1.
		\label{eq:dia-prd}
	\end{align}
	By the definition of \(\Theta\), the norm in \eqref{eq:dia-nrm} is the pullback of the \(\ell_1\)-sum norm on \(\mathfrak{A}(W^*)\),
	so it is equivalent to the original dual norm.
	Since
	\[ (\lambda,h,\mu)(\lambda',k,\mu') =(\lambda\lambda',h\lambda'+\mu k,\mu\mu'), \]
	the multiplication in \eqref{eq:dia-prd} is the pullback of the triangular multiplication.
	Thus \(\norm{\cdot}_\diamond\) is submultiplicative, the unit is \(\phi_0+\phi_1\),
	and the radical is precisely \(F^\perp\), with square zero.

	The space \(W\) is separable.
	Applying Theorem~\ref{thm:tri-cal} to \(W\) gives a separable space \(X_W\) and a Banach-algebra isomorphism
	\[ \Cal(X_W)\simeq\mathfrak{A}(W^*). \]
	Taking \(Z_V=X_W\) and composing with \(\Theta\) proves the assertion.

	If \(\dim V=1\), apply Theorem~\ref{thm:tri-cal} to the zero coefficient space.
	Its first self-corner \(X_0(0)\) satisfies \(\Cal(X_0(0))\simeq\K\).
	Choose \(\phi\in S_{V^*}\), and define
	\[ (a\phi)\diamond(b\phi)=ab\phi, \qquad \norm{a\phi}_\diamond=|a|. \]
	Then \((V^*,\norm{\cdot}_\diamond,\diamond)\simeq\K\), and we take \(Z_V=X_0(0)\).
\end{proof}

\begin{proof}[Proof of Corollary~\ref{cor:ban-uni}]
	The first assertion follows from Theorem~\ref{thm:dua-cal} by forgetting the multiplication.
	If \(E\ne\{0\}\) is separable and reflexive, then \(E^*\) is separable.
	Apply the first assertion to \(V=E^*\) and use the canonical isomorphism \(E\simeq E^{**}\).
	Thus \(E\simeq\Cal(Z_{E^*})\), and this Calkin algebra is reflexive.
\end{proof}

\begin{proof}[Proof of Corollary~\ref{cor:ref-rad}]
	Let \(E\) be separable and reflexive.
	Then \(E^*\) is separable, so the triangular theorem applies to \(V=E^*\).
	Put \(Y_E=X_{E^*}\).
	Reflexivity gives \(V^*=E^{**}\simeq E\), and hence
	\[
		\Cal(Y_E)\simeq
		\left\{\begin{pmatrix}\lambda&0\\e&\mu\end{pmatrix}:
		\lambda,\mu\in\K,\ e\in E\right\}.
	\]
	As a Banach space the algebra on the right is \(\K\oplus E\oplus\K\), and is therefore reflexive.
	Its radical and semisimple quotient are given by \eqref{eq:tgt-rad}.
\end{proof}

\begin{proof}[Proof of Corollary~\ref{cor:ell-rad}]
	Apply Theorem~\ref{thm:tri-cal} with \(V=\ell_1\) and use \(\ell_1^*=\ell_\infty\).
\end{proof}

Let \(q_V:\Bcal(X_V)\to\Cal(X_V)\) be the quotient map, and let
\(\mathcal{E}(X_V)\) be the closed ideal of inessential operators.
By Atkinson's theorem and the invertibility characterization of the Jacobson
radical \cite{Atkinson1951,Pietsch1980,Dales2000},
\(S\in\mathcal{E}(X_V)\) if and only if
\(q_V(S)\in\operatorname{rad}\Cal(X_V)\).

\begin{corollary}[Inessential operators and commutators]
	\label{cor:ine-com}
	The quotient \(\mathcal{E}(X_V)/\Kcal(X_V)\) is isometric to \(V^*\), and
	\(\mathcal{E}(X_V)^2\subseteq\Kcal(X_V)\). Moreover,
	\(\operatorname{rad}\Cal(X_V)=\{ab-ba:a,b\in\Cal(X_V)\}\).
	Consequently, every \(S\in\mathcal{E}(X_V)\) can be written as
	\(S=AB-BA+K\), where \(A,B\in\Bcal(X_V)\) and \(K\in\Kcal(X_V)\).
\end{corollary}

\begin{proof}
	The quotient map identifies \(\mathcal{E}(X_V)/\Kcal(X_V)\) isometrically with \(\operatorname{rad}\Cal(X_V)\). By Theorem~\ref{thm:tri-cal}, this radical is isometric to \(V^*\) and has
	square zero. Hence \(\mathcal{E}(X_V)^2\subseteq\Kcal(X_V)\).
	Since \(\mathfrak{A}/\operatorname{rad}(\mathfrak{A})\simeq\K\oplus\K\)
	is commutative, every commutator in \(\mathfrak{A}\) belongs to its radical.
	Conversely, \eqref{eq:tgt-rad} and direct multiplication give
	\[
		\left[
			\begin{pmatrix}0&0\\v&0\end{pmatrix},
			\begin{pmatrix}1&0\\0&0\end{pmatrix}
		\right]
		=\begin{pmatrix}0&0\\v&0\end{pmatrix}.
	\]
	Hence Theorem~\ref{thm:tri-cal} gives
	\(\operatorname{rad}\Cal(X_V)=\{ab-ba:a,b\in\Cal(X_V)\}\).
	If \(S\in\mathcal{E}(X_V)\), choose \(a,b\in\Cal(X_V)\) such that
	\(q_V(S)=ab-ba\), and let \(A,B\in\Bcal(X_V)\) be lifts of \(a,b\).
	Then \(K:=S-(AB-BA)\in\Kcal(X_V)\).
\end{proof}

The construction uses the rational information recorded by \(\mathcal O_V\) in Remark~\ref{rem:orc-rec}.
All rational threshold comparisons determine the norm on the dense rational core \(V_{\Q}\),
and its completion is \(V\).
By Theorem~\ref{thm:tri-cal}, the resulting Calkin algebra recovers \(V^*\) as its Jacobson radical.
Thus the construction gives a \emph{faithful coding} of the dual Banach-space isomorphism type, in the following sense.

\begin{corollary}[Faithful dual coding]\label{cor:dua-cod}
	For separable Banach spaces \(U\) and \(V\),
	\[ \Cal(X_U)\simeq\Cal(X_V)\ \text{as Banach algebras} \quad\Longleftrightarrow\quad U^*\simeq V^*\ \text{as Banach spaces}. \]
\end{corollary}

\begin{proof}
	Every Banach-algebra isomorphism \(\mathfrak{A}(E)\to\mathfrak{A}(F)\) maps the Jacobson radical onto the Jacobson radical.
	Its restriction therefore gives a Banach-space isomorphism \(E\to F\).
	Conversely, every Banach-space isomorphism \(R:E\to F\) induces the Banach-algebra isomorphism
	\[
		\begin{pmatrix}\lambda&0\\e&\mu\end{pmatrix}
		\longmapsto
		\begin{pmatrix}\lambda&0\\Re&\mu\end{pmatrix}.
	\]
	Apply this observation with \(E=U^*\), \(F=V^*\), and use Theorem~\ref{thm:tri-cal}.
\end{proof}

\begin{proof}[Proof of Theorem~\ref{cor:ide-lat}]
	Take \(V=\ell_1\) and \(X=X_{\ell_1}\). By
	Theorem~\ref{thm:tri-cal} and Corollary~\ref{cor:ine-com},
	the quotient map identifies \(\mathcal{E}(X)/\Kcal(X)\) isometrically with \(\operatorname{rad}\Cal(X)\), which is isometric to \(\ell_\infty\).
	The triangular multiplication gives
	\[ (\lambda,v,\mu)(0,w,0)=(0,\mu w,0), \qquad (0,w,0)(\lambda,v,\mu)=(0,w\lambda,0). \]
	Thus the image of every closed subspace of \(\ell_\infty\) is a closed
	two-sided ideal in \(\Cal(X)\). Taking preimages under the quotient map gives
	the stated lattice isomorphism and quotient identification.
	Writing \(\mathcal{J}_W\) for the preimage associated with a closed subspace
	\(W\subseteq\ell_\infty\), we have
	\[ W_1\subseteq W_2 \quad\Longleftrightarrow\quad \mathcal{J}_{W_1}\subseteq\mathcal{J}_{W_2}. \]

	If \(Y\) is separable, then \(B_{Y^*}\) is weak-star metrizable. Choose a
	weak-star dense sequence \((y_n^*)\) in \(B_{Y^*}\). The map
	\(y\mapsto(y_n^*(y))_n\) embeds \(Y\) isometrically into \(\ell_\infty\), because
	\(\sup_n|y_n^*(y)|=\norm{y}\). Its image corresponds to an ideal \(\mathcal{J}\), and the quotient identification gives an isometry from \(Y\) onto \(\mathcal{J}/\Kcal(X)\).
\end{proof}

\begin{remark}[The obstruction example is not a dual space]
	\label{rem:acu-not-dual}
	Let \(A_D=\mathbb{C} I_{E_D}\oplus J_D\) be the algebra constructed in \cite[Section~5]{AcuavivaAcuaviva2026},
	where \(E_D=\ell_1(\N,D)\) and \(J_D\) is the closure of the finite-support \(D\)-valued matrices.
	Choose a state \(\varphi\in D^*\).
	For a finite matrix \(a=(d_{ij})\), put
	\[ R(a)=\sum_i\varphi(d_{ii})\varepsilon_{ii}. \]
	Then
	\[ \norm{R(a)}=\sup_i|\varphi(d_{ii})| \leq\sup_i\norm{d_{ii}}\leq\norm{a}. \]
	Thus \(R\) extends to a projection from \(J_D\) onto the closure of the finite scalar diagonal matrices,
	which is isometric to \(c_0\).
	Since the displayed sum defining \(A_D\) is topological, \(c_0\) is complemented in \(A_D\).

	If \(A_D\) were isomorphic to a dual space, then it would be complemented in its bidual.
	A complemented subspace of such a space is complemented in its own bidual: if \(j:Z\to A_D\), \(q:A_D\to Z\),
	and \(P:A_D^{**}\to A_D\) are the corresponding maps, then \(qPj^{**}:Z^{**}\to Z\) restricts to the identity on \(Z\).
	Taking \(Z=c_0\) would make \(c_0\) complemented in \(c_0^{**}=\ell_\infty\),
	contrary to the Phillips--Sobczyk theorem \cite{Phillips1940,Sobczyk1941}.
	Hence \(A_D\) is not isomorphic to any dual Banach space.
\end{remark}

\section{Discussion}
\label{sec:dis}

Part~I packages the construction into a countable language of finite-extension requirements,
and Theorem~\ref{thm:dat-cns} records the analytic input used in Part~II.
We isolate the finite-certificate and oracle-relative viewpoints, and then discuss two possible extensions.
The multi-object and injury constructions are only directions for future work and require new estimates.

\subsection{Finite-certificate principle}

The spaces are fixed before any bounded operator is chosen.
Hence the construction enumerates finite rational requirements, not operators.
For a rational finitely supported vector, the relevant  orbit is finite dimensional,
so a persistent local error can be recorded by finite rational data.
In the two self corners and the forward corner, \hyperref[itm:ana-g3]{\textup{(G3)}} and Proposition~\ref{prop:fin-ext} allow us to choose one of the cofinally pre-existing realizations carrying a matching certified continuation.
The reverse corner instead uses \hyperref[itm:ana-g4]{\textup{(G4)}} and the missing-level estimates of Section~\ref{sec:rev}.
Thus none of the later operator arguments changes the datum.

\subsection{Oracle-relative finite extensions}

The rank recursion is relative to the oracle \(\mathcal{O}_V\) of Remark~\ref{rem:orc-rec},
supplied by the local structure of \(V\).
It provides the finite norm and admissibility data used by the fixed program \(\mathfrak{M}\).
The finite admissibility, acyclicity, and fairness checks are proved in Remark~\ref{rem:orc-rec},
Lemma~\ref{lem:req-adm}, and Lemma~\ref{lem:fai-per}; Proposition~\ref{prop:fin-ext} gives the finite-extension property.
The analytic estimates remain separate.

\subsection{Finite families and operator corners}

The present construction has two spaces.
Its four corners are
\[
	\begin{aligned}
		\Bcal(X_i)/\Kcal(X_i)&=\K[I_{X_i}] &&(i=0,1),\\
		\Bcal(X_1,X_0)/\Kcal(X_1,X_0)&=0,\\
		\Bcal(X_0,X_1)/\Kcal(X_0,X_1)&\simeq V^*.
	\end{aligned}
\]
The two sorts therefore act as two objects, and the protected module controls one nonzero quotient corner.

A natural next problem is to use finitely many sorts \((X_i)_{i=1}^d\) and prescribe several quotient corners.
A first target is
\[ \K^d\ltimes R, \qquad R=\bigoplus_{(i,j)\in E}V_{ij}^*, \qquad R^2=\{0\}, \]
where \(E\) is a finite directed graph and every \(V_{ij}\) is separable.
At the formal level, finitely many sorts and corner types would still give a countable requirement language.
This is only a bookkeeping observation and does not give the needed BD estimates.
Separate protected modules and weight pools could record the finite data for the different corners.

The new difficulty is analytic.
In the square-zero case, every composition of two radical corner operators must be compact.
Separate local orbit estimates for the two factors do not prove this by themselves.
One would also need compatible collision, interval-restriction, and mixed-corner estimates.
If nonzero products along paths are prescribed, the protected modules must respect composition.
For a path
\[ X_i\longrightarrow X_j\longrightarrow X_k, \]
the composition of the two protected families would have to agree, modulo compact operators,
with the multiplication prescribed in the \(i\)-to-\(k\) corner.
This is not part of the present construction and would require new composition-compatible modules.

\subsection{Possible injury constructions}

Lemma~\ref{lem:fai-per} ensures persistence without injury in the present stage construction.
In other problems, the finite requirements may not remain jointly active under monotone end extensions.
This suggests a version in which the current realization of a requirement may later be abandoned and restarted at a larger rank.

An old BD atom cannot be removed after it has been constructed.
Therefore atoms belonging to an injured realization would remain in the space even when they are no longer used by the active coding history.
A finite-injury construction would need an analytic estimate of the form
\[ \text{abandoned finite realization} \longrightarrow \text{compact contribution}. \]
For example, the abandoned pieces might have to be confined to finite intervals,
or their tail effects might have to tend to zero uniformly.
It is not enough that each individual requirement is injured only finitely many times,
because there may still be infinitely many abandoned pieces.
A new compactness theorem would be needed to control their total effect.
An infinite-injury version would require still stronger convergence estimates and is further from the present argument.

\subsection{Reusing the analytic input}

Theorem~\ref{thm:dat-cns}, summarized by
\hyperref[itm:ana-g1]{\textup{(G1)}}--\hyperref[itm:ana-g7]{\textup{(G7)}}
and \hyperref[itm:ana-ea1]{\textup{(EA1)}}--\hyperref[itm:ana-ea4]{\textup{(EA4)}},
provides the analytic input for the operator proof.
A different requirement language or protected module can reuse Part~II only after these properties have been verified.
This is not automatic: new corner products require composition-compatible modules and mixed-corner estimates,
whereas an injury rule requires the compactness control described above.

\section*{Acknowledgements}
Rui Liu and Jie Shen were partially supported by the National Natural Science Foundation of China (Grant Nos. 12471131 and 12071230). Rui Liu is grateful for the opportunities to make several extended visits to the Departments of Mathematics at Texas A\&M University and at The University of Texas at Austin, during which he benefited greatly from the guidance of Thomas Schlumprecht and Edward Odell and from many helpful discussions with them. Rui Liu and Jie Shen also thank Longyun Ding and Su Gao for their valuable advice and discussions on set theory, and for their encouragement and support over the years. Jie Shen is grateful to Liang Yu for his lectures on recursion theory at the 2025 Fudan Logic Summer School. He also thanks Zhaokuan Hao, Ruizhi Yang, Ningyuan Yao, and Will Johnson, as well as the School of Philosophy at Fudan University, for organizing and hosting the summer school.

\section*{AI assistance, provenance, and authors' responsibility}\label{sec:ai-assistance}

\begin{enumerate}[label=\textup{(\arabic*)},leftmargin=*]
	\item \textbf{The authors' original design.}
	The authors formulated the main theorem and developed the specific architecture of the construction independently of generative-AI assistance.
	We incorporated oracle Turing machines and recursion-theoretic methods into the AH-type construction developed here.
	The authors designed the two-sorted finite-extension framework, the protected module, and the finite certificates.
	We also designed the paired vector-valued probes, the use of cofinal realizations, and the organization of the four operator corners.
	We also developed the vector-valued Bourgain--Delbaen implementation needed for this architecture and the reduction of operator classification to local orbit estimates.
	These claims concern the specific design used in this paper.

	\item \textbf{Research and writing support.}
	We used  GPT-5.6 Sol (OpenAI) and
	GPT-6 Astra (OpenAI) to check notation, formulas, and the domains and codomains of maps, look for gaps in arguments, suggest clearer wording, generate figures and tables, and help find references. GPT-5.6 Sol (OpenAI) and GPT-6 Astra (OpenAI) helped us locate related work in the Bourgain--Delbaen, Gowers--Maurey, and Argyros--Haydon literature.
	We checked the relevant statements in the original sources.
	In particular, the technical notes in Sections~6--8 identify established methods for auxiliary averages, off-weight and missing-weight estimates, stopping trees, and counting arguments.
	The references are given at the points where these methods are used.

	\item \textbf{AI contributions to the proofs.}
	The stopping-event formulation in Lemma~\ref{lem:qnt-aux-avg} was prompted by an AI suggestion.
	We then checked the literature and found that the underlying counting and tree-truncation techniques were already established.

	The contribution to Lemmas~\ref{lem:car-dom-tre}--\ref{lem:pai-bi} was more substantial.
	We confirmed that the remaining gap we had identified in the proof concerned the dependent-sequence estimate for evaluations with low level.
	GPT-5.6 Sol (OpenAI) then supplied the carrier-labelled stopping-tree mechanism and the associated counting proof that we used to close this gap.
	We rewrote the proof in our notation and checked it independently against the definitions and assumptions of our construction.
	Later source checks confirmed earlier uses of weight and depth stopping, boundary-incidence counting, and summation over disjoint intervals.
	These sources are cited beside the relevant arguments.

	\item \textbf{Authors' responsibility.}
	We independently checked the AI-assisted arguments included in the paper and made all final decisions.
	The authors take full responsibility for the mathematical statements, proofs, citations, and final text.
	No AI system is treated as an author or as an authority for mathematical correctness.
\end{enumerate}

\bibliographystyle{amsplain}
\bibliography{cite}
\enlargethispage{24pt}

\end{document}